\documentclass{amsart}
\usepackage{mathrsfs}
\usepackage{amssymb,latexsym,amsmath,amsthm,enumitem,mathrsfs}
\usepackage[margin=1in]{geometry}
\usepackage{color}
\usepackage{stmaryrd}
\usepackage{mathrsfs}
\usepackage{lineno}
\usepackage{bbm}   
\usepackage{scalerel,stackengine}
\usepackage{ esint }
\usepackage{graphicx}
\usepackage{amssymb}

\newcommand{\R}{\mathbb R}

\newcommand{\Z}{\mathbb Z}
\newcommand{\C}{\mathbb C}
\newcommand{\N}{\mathbb N}
\newcommand{\T}{\mathbb T}

\renewcommand{\S}{\text{S}}
\newcommand{\w}{\omega}

\newcommand{\e}{\varepsilon}
\newcommand{\g}{\gamma}
\newcommand{\p}{\varphi}
\newcommand{\s}{\psi}
\newcommand{\z}{\zeta}

\renewcommand{\a}{\alpha}
\renewcommand{\b}{\beta}
\renewcommand{\d}{\delta}

\renewcommand{\P}{\mathbb{P}}

\newcommand{\mc}{\mathcal}
\newcommand{\mb}{\mathbb}

\def\set4{\mathcal I}
\def\tup14{(1,2,3,4)}

\newcommand\vwidehat[1]{\arraycolsep=0pt\relax%
\begin{array}{c}
\stretchto{
  \scaleto{
    \scalerel*[\widthof{\ensuremath{#1}}]{\kern-.5pt\bigwedge\kern-.5pt}
    {\rule[-\textheight/2]{1ex}{\textheight}} 
  }{\textheight} %
}{0.5ex}\\           
#1\\                 
\rule{-1ex}{0ex}
\end{array}
}
\newtheorem*{comm*}{Comment}
\newtheorem*{rmk}{Remark}
\newtheorem{definition}{Definition}
\newtheorem*{lemma*}{Lemma}
\newtheorem*{theorem*}{Theorem}
\newtheorem*{cor*}{Corollary}
\newtheorem{theorem}{Theorem}
\newtheorem{thm}{Theorem}[section]
\newtheorem{corollary}[thm]{Corollary}

\newtheorem{proposition}[thm]{Proposition}
\newtheorem*{proposition*}{Proposition}

\newtheorem{lemma}[thm]{Lemma}

\usepackage{ bbold }
\usepackage{bbm}   
\usepackage{scalerel,stackengine}
\stackMath
\newcommand\widecheck[1]{%
\savestack{\tmpbox}{\stretchto{%
  \scaleto{%
    \scalerel*[\widthof{\ensuremath{#1}}]{\kern-.6pt\bigwedge\kern-.6pt}%
    {\rule[-\textheight/2]{1ex}{\textheight}}
  }{\textheight}%
}{0.5ex}}%
\stackon[1pt]{#1}{\scalebox{-1}{\tmpbox}}%
}
\usepackage{ esint }

\newcommand{\supp}{\mathrm{supp}}

\begin{document}

\author{Dominique Maldague}
\address{Department of Mathematics\\
Massachusetts Institute of Technology\\
Cambridge, MA 02142-4307, USA}
\email{dmal@mit.edu}


\date{}

\title{A sharp square function estimate for the moment curve in $\R^3$ }
\maketitle

\begin{abstract}
    We prove a sharp (up to $C_\e R^\e$) $L^7$ square function estimate for the moment curve in $\R^3$. 
\end{abstract}
\section{Introduction}

We use a combination of high-low frequency analysis and induction to prove a sharp $L^7$ square function estimate for the moment curve $\mc{M}^3=\{(t,t^2,t^3):0\le t\le 1\}$. As in \cite{locsmooth} for the cone in $\R^3$, this is an example of using techniques developed in decoupling theory \cite{BD} to prove square function estimates. Our argument does not follow the inductive scheme in \cite{locsmooth}, though we do use their wave envelope estimate (Theorem 1.3 of \cite{locsmooth}). Instead, we add an element of induction to the high-low proof of decoupling for the moment curve, which is contained in \cite{M3smallcap} and is based on the argument from \cite{gmw} for the parabola.
We provide an overview of the proof technique in \textsection\ref{int}. It is worth noting that the only previous sharp square function estimates in restriction theory are for even $L^p$ exponents, which often makes Fourier analysis problems more approachable. 

For $R\ge 1$, consider the anisotropic neighborhood 
\[\mc{M}^3(R)=\{(\xi_1,\xi_2,\xi_3): \xi_1\in[0,1],\,|\xi_2-\xi_1^2|\le R^{-{2/3}},\,|\xi_3-3\xi_1\xi_2+2\xi_1^3|\le R^{-1} \}.  \]
Partition this neighborhood of $\mc{M}^3$ into canonical blocks $\theta$, which have the form  
\begin{equation}\label{momblocksintro}
    \theta=\{(\xi_1,\xi_2,\xi_3): lR^{-1/3}\le \xi_1<(l+1)R^{-1/3},\,|\xi_2-\xi_1^2|\le R^{-2/3},\,|\xi_3-3\xi_1\xi_2+2\xi_1^3|\le R^{-1} . \} 
\end{equation}
For Schwartz functions $f:\R^3\to\C$, define the Fourier projection onto $\theta$ by $f_\theta(x)=\int_\theta \widehat{f}(\xi)e^{2\pi i x\cdot\xi} d\xi$. Our main result is the following. 

\begin{theorem} \label{main} For any $\e>0$, there exists $C_\e<\infty$ such that
\begin{equation}\label{sqfn}    \int_{\R^3}|f|^7\le C_\e R^\e\int_{\R^3}|\sum_\theta|f_\theta|^2|^{\frac{7}{2}}   \end{equation}
for any Schwartz function $f:\R^3\to\C$ with Fourier transform supported in $\mc{M}^3(R)$.
\end{theorem}

We also obtain a version of Theorem \ref{main} for curves in $\R^3$ with torsion, which is explained in Appendix B. Theorem \ref{main} is sharp, up to the $C_\e R^\e$ factor. This may be seen from the constructive interference example which we now describe. Let $\s:\R^3\to[0,\infty)$ be a smooth bump function supported in the unit ball. For each $\theta$, let $\s_\theta=|\theta|^{-1}\s\circ T_\theta$ where $T_\theta:\R^3\to\R^3$ is an affine transformation mapping an ellipsoid comparable to $\theta$ to the unit ball. For a small universal constant $c>0$, the support of $\s_\theta$ is completely contained in a canonical block at scale $cR$, so Theorem \ref{main} implies that
\begin{equation}\label{sharpex} 
\int_{\R^3}|\sum_{\theta}\widecheck{\s}_\theta|^7\lesssim_\e R^\e \int_{\R^3}|\sum_\theta|\widecheck{\s}_\theta|^2|^{7/2}.
\end{equation}  
The function $|\widecheck{\s}_\theta|$ has amplitude $\sim 1$ on the set $\theta^*=\{x\in\R^3:|x\cdot\xi|\le 1\quad\forall\xi\in\theta-\theta\}$ and decays rapidly away from $\theta^*$. The union of the sets $\theta^*$ is a bush centered at the origin of $\sim R^{1/3}\times R^{2/3}\times R$ planks which are tangent to the light cone. A calculation shows that the right hand side of \eqref{sharpex} is dominated by the portion of the bush at a distance $\sim R$ from the origin, where the $\theta^*$ do not overlap, so 
\[ \int_{\R^3}|\sum_\theta|\widecheck{\s}_\theta|^2|^{7/2}\sim \sum_\theta|\theta^*|\sim R^{7/3}. \]
Finally, it is easy to see that in a neighborhood of radius $\sim 1$ at the origin, $|\sum_\theta\widecheck{\s}_\theta|\gtrsim R^{1/3}$, so the left hand side of \eqref{sharpex} is bounded below by $R^{7/3}$. This example also shows that no estimate of the form \eqref{sqfn} holds in $L^p$ if $p>7$.

Estimates of the form \eqref{sqfn} have a long history in harmonic analysis. The classical $L^4$ square function estimate for the parabola is based on geometric observations by Fefferman \cite{feffL4} and is recorded in \cite{cordoba}. The square function conjecture for paraboloids in $\R^n$ (Conjecture 5.19 of \cite{demeter}) is known to imply sharp results for the  Kakeya maximal function, the Bochner-Riesz means, and the Fourier restriction operator; see \cite{carbery} and the references therein. In future work, we intend to explore analogous applications for the moment curve. Although square function estimates are known to have many applications, there are few sharp results of the form \eqref{sqfn} in the literature. There has, however, been some partial progress on establishing square function estimates for the moment curve. In an unpublished work that was shared with the author, H. Jung proved a non-sharp version of Theorem \ref{main} with a positive power of $R$ in the upper bound. In \cite{revphil}, the authors used approximate solution counting for Vinogradov systems to obtain square function estimates for the moment curve in $\R^n$, but for non-sharp exponents $2\le p\le 2n$ which are even. The only other sharp square function estimate we are aware of is for the cone in $\R^3$ \cite{locsmooth}, which, by the work of \cite{mss}, resolves the local smoothing conjecture for the wave equation in $\R^{2+1}$. Our proof of Theorem \ref{main} uses both the square function estimate for the parabola in $\R^2$ and for the cone in $\R^3$.

The organization of this paper is as follows. In \textsection\ref{int}, we describe the main ideas behind the proof of Theorem \ref{main}. The proof of Theorem \ref{main} happens in two steps. First, we prove a version of \eqref{sqfn} where the terms in the square function are convolved with certain weights. Then we show that this averaged version implies Theorem \ref{main}. We introduce the square function constants $S_1(R)$ and $S_2(R)$ corresponding to these two versions of \eqref{sqfn} in \textsection\ref{const}. We set up tools for the high-low argument in \textsection\ref{tools} and discuss the relevant geometry related to the moment curve and to the cone in \textsection\ref{geo}. We devote \textsection\ref{keyalgo} to the key technical step of unwinding the pruning process (see \textsection\ref{int}). The high-low method then allows us to bound the \emph{broad} part of the left hand side of \eqref{sqfn}, when the integrand is dominated by a trilinear version of itself, which we carry out in \textsection\ref{broad}. Finally we prove Theorem \ref{main} by bounding $S_1(R)$ and then $S_2(R)$ in \textsection\ref{mainsec}. Appendix A contains proofs of auxiliary square function and wave envelope estimates we require for the parabola and the cone. In Appendix B, we explain how to adapt the argument to obtain Theorem \ref{main} for general curves $\g(t)$ with torsion.

I am grateful for the mentorship of Larry Guth, with whom I discussed many methods for approaching the moment curve. I also want to thank Ciprian Demeter for helpful conversations about background. DM is supported by the National Science Foundation under Award No. 2103249.

\subsection{Overview of the proof of Theorem \ref{main}\label{int}}

In this section, we fix $\e>0$ and a Schwartz function $f:\R^3\to\C$. Let $\theta$ denote canonical $R^{-1/3}\times R^{-2/3}\times R^{-1}$ moment curve blocks. By pigeonholing arguments, it suffices to assume that either $\|f_\theta\|_\infty\sim 1$ or $\|f_\theta\|_\infty=0$ for each $\theta$, and that for some $\a>0$ and $\b>0$,  
\[ \int_{\R^3}|f|^7\lesssim (\log R)\a^7|U_{\a,\b}| \]
where $U_{\a,\b}=\{x\in\R^3:|f(x)|>\a,\quad\frac{\b}{2}\le \sum_\theta|f_\theta(x)|^2\le \b\}$. 
Our goal is to show that
\begin{equation}\label{goal'} \a^7|U_{\a,\b}|\lesssim_\e R^\e \int_{\R^3}|\sum_\theta|f_\theta|^2|^{7/2}. \end{equation}
The initial step in bounding $|U_{\a,\b}|$ involves the $L^6$ trilinear restriction theorem for $\mc{M}^3$. Since this theorem bounds trilinear expressions, we are actually bounding the \emph{broad} part of $U_{\a,\b}$ (which is the subset where $|f|\lesssim |f_1f_2f_3|^{1/3}$ for Fourier projections $f_i$ of $f$ onto separated neighborhoods of $\mc{M}^3$). The \emph{narrow} (or {not broad}) part of $U_{\a,\b}$ is dealt with using a standard inductive argument. 

The high-low method partitions $U_{\a,\b}$ into $\le\e^{-1}$ many subsets $\Omega_k$ which we bound in separate cases. Let $R_{N-1}=R^{1-\e}$ and let $\tau_{N-1}$ denote $R_{N-1}^{-1/3}\times R_{N-1}^{-2/3}\times R_{N-1}$ moment curve blocks and let $g_{N-1}=\sum_{\tau_{N-1}}|f_{\tau_{N-1}}|^2$. The first subset we bound is 
\[ \Omega_{N-1}=\{x\in U_{\a,\b}:A\b\le g_{N-1}(x)\} \]
where $A>0$ is a constant that we will choose below. The high-low decomposition for $g_{N-1}$ is $g_{N-1}=g_{N-1}^\ell+g_{N-1}^h$ where $g_{N-1}^{\ell}=g_{N-1}*\widecheck{\eta}_{N}$ for a bump function $\eta_N$ equal to $1$ on $B_{R^{-1/3}}(0)$ and supported in $B_{2R^{-1/3}}(0)$. 
By a local $L^2$-orthogonality argument, $A\b\le g_{N-1}(x)$ implies that $g_{N-1}(x)\le 2|g_{N-1}^h(x)|$ when $A$ is a  sufficiently large constant.

Applying the $L^6$ multilinear restriction essentially yields 
\[ \a^6|\Omega_{N-1}|\lesssim \int_{R_{N-1}^{1/3}-\Omega_{N-1}}|g_{N-1}|^3 \]
where $R_{N-1}^{1/3}-\Omega_{N-1}$ means the $R_{N-1}^{1/3}$-neighborhood of $\Omega_{N-1}$. Since the Fourier support of $g_{N-1}$ is contained in $\cup_{\tau_{N-1}}(\tau_{N-1}-\tau_{N-1})\subset B_{2R_{N-1}^{-1/3}}(0)$, $g_{N-1}$ has roughly constant modulus on ${R_{N-1}^{1/3}}$-balls on the spatial side. Therefore, since $A\b\le g_{N-1}\lesssim|g_{N-1}^h|$ on $\Omega_{N-1}$, 
\[ \int_{R_{N-1}^{1/3}-\Omega_{N-1}}|g_{N-1}|^3 \lesssim \frac{1}{A\b}\int_{R_{N-1}^{1/3}-\Omega_{N-1}}|g_{N-1}^h|^4 .\]
Note that we chose to go from an $L^3$ expression of $g_{N-1}$ to an $L^4$ expression. We did this because the Fourier support of $g_{N-1}^h$ is contained in a neighborhood of the truncated cone, so we would like to use the sharp $L^4$ square function estimate for the cone \cite{locsmooth}. This yields the bound
\[\int_{\R^3}|g_{N-1}^h|^4\lesssim_\e R^\e \int_{\R^3}\big|\sum_{\tau_{N-1}}||f_{\tau_{N-1}}|^2-|f_{\tau_{N-1}}|^2*\widecheck{\eta}_N|^2\big|^2.  \]
Since $\|\widecheck{\eta}_{N}\|_1\sim 1$ and there are $\lesssim R^{\e/3}$ many $\theta$ contained in each $\tau_{N-1}$, the integral on the right hand side is bounded using Cauchy-Schwarz and Young's convolution inequality by a constant factor times
\[ R^{2\e}\int_{\R^3}|\sum_\theta|f_\theta|^4|^2. \]
The summary of the argument so far is that
\[ \a^6|\Omega_{N-1}|\lesssim \frac{1}{A\b}R^{3\e}\int_{\R^3}|\sum_\theta|f_\theta|^4|^2. \]
Comparing with \eqref{goal'}, we see that it suffices to check that $\a\lesssim \b$ and $\int_{\R^3}|\sum_\theta|f_\theta|^4|^2\lesssim \int_{\R^3}|\sum_\theta|f_\theta|^2|^{7/2}$. The first inequality is justified since morally, each $|f_\theta|$ may be thought of as a sum amplitude $1$ wave packets localized to non-overlapping translates of dual planks $\theta^*$, so $|f_\theta|\lesssim |f_\theta|^2$. Then for $x\in U_{\a,\b}$, we have $\a\lesssim |\sum_\theta f_\theta(x)|\lesssim \sum_\theta|f_\theta(x)|\lesssim \sum_\theta|f_\theta(x)|^2\lesssim \b$. The second inequality is justified by the assumption that $\|f_\theta\|_\infty\lesssim 1$, so $\sum_\theta|f_\theta|^4\lesssim \sum_\theta|f_\theta|^{7/2}$, and then using $\|\cdot\|_{\ell^{7/2}}\le \|\cdot\|_{\ell^2}$. This concludes the bound of $|\Omega_{N-1}|$. 

\[\]

\noindent\underline{Intermediate scales:} Let $R_k=R^{k\e}$, let $\tau_k$ be canonical $R_k^{-1/3}\times R_k^{-2/3}\times R_k^{-1}$ moment curve blocks, and let $g_k=\sum_{\tau_k}|f_{\tau_k}|^2$. Decompose $g_k$ into high-low parts $g_k=g_k^\ell+g_k^h$ by defining $g_k^\ell=g_k*\widecheck{\eta}_k$, where $\eta_k$ is a bump function equal to $1$ on $B_{R_{k+1}^{-1/3}}(0)$. The $k$th subset of $U_{\a,\b}$ that we consider is 
\[ \Omega_k=\{x\in U_{\a,\b}:A^{N-k}\b\le g_k\quad \text{and}\quad  g_l\le A^{N-l}\b \quad\forall l=k+1,\ldots,N-1\}. \]
As in the analysis of $\Omega_{N-1}$, we can show that on $\Omega_k$, $g_k$ is high-dominated. After applying the $L^6$ trilinear restriction and using that $A^{N-k}\b\le g_k\lesssim|g_k^h|$ on $\Omega_k$, we have
\[ \a^6|\Omega_k|\lesssim \frac{1}{A^{N-k}\b}\int_{\R^3}|g_k^h|^4. \]
Again, $g_k^h$ is Fourier supported on the $R_k^{-1/3}$-dilation of the truncated cone, so we may apply the $L^4$ square function estimate for the cone, yielding
\begin{equation}\label{last} \a^6|\Omega_k|\lesssim \frac{1}{A^{N-k}\b}C_\e R^\e \int_{\R^3}|\sum_{\tau_k}|f_{\tau_k}|^4|^2. \end{equation}
Here, unlike in the analysis of $\Omega_{N-1}$, $\tau_k$ may be much coarser than $\theta$, so we cannot use trivial inequalities and the assumption that $\|f_\theta\|_{L^\infty}\lesssim 1$ to arrive at the right hand side of \eqref{goal'}. To provide an alternative $L^\infty$ bound for each $f_{\tau_k}$, we perform a pruning process on the wave packets. This pruning process is the same as the one from \cite{gmw}, in which we argue that on $\Omega_k$, $f$ may be replaced by a version $f^k=\sum_{\tau_k} f_{\tau_k}^k$ where each $f_{\tau_k}^k$ only has wave packets with amplitude $\lesssim \b/\a$. The pruned $f_{\tau_k}^k$ satisfy the property that $\|f_{\tau_k}^k\|_{\infty}\lesssim \frac{\b}{\a}$. Using $f^k$ in place of $f$ in \eqref{last} and the good $L^\infty$ bound for each $f_{\tau_k}^k$, we arrive at the inequality
\[ \a^6|\Omega_k|\lesssim_\e \frac{1}{A^{N-k}\b}R^\e\frac{\b}{\a}\int_{\R^3}|\sum_{\tau_k}|f_{\tau_k}^k|^{7/2}|^2, \]
which, using $\|\cdot\|_{\ell^{7/2}}\le\|\cdot\|_{\ell^2}$, implies that 
\begin{equation}\label{int2} \a^7|\Omega_k|\lesssim_\e R^\e \int_{\R^3}|\sum_{\tau_k}|f_{\tau_k}^k|^{2}|^{7/2}. \end{equation}
The left hand side looks good because it is an $L^7$ expression. It remains to consider how to bound the right hand side by the $L^7$ integral of the square function at our desired scale $\theta$. It looks as though the right hand side is partial progress towards $\sum_\theta|f_\theta|^2$, so we would like to invoke induction to finish the argument. Indeed, if our goal were to show an $(\ell^2,L^7)$ decoupling estimate and had shown \[ \a^7|\Omega_k|\lesssim_\e R^\e\Big(\sum_{\tau_k}\|f_{\tau_k}\|_{L^7(\R^3)}^2 \Big)^{7/2}, \]
then we could rescale each $f_{\tau_k}$ and invoke induction on scales to justify $\|f_{\tau_k}\|_{L^7(\R^3)}\lesssim (\sum_{\theta\subset\tau_k}\|f_\theta\|_{L^7(\R^3)}^2)^{1/2}$, which combines with the displayed inequality to give the desired $(\ell^2,L^7)$ estimate. The main difficulty of proving a square function estimate compared to a decoupling estimate is that the integrand on the right hand side of \eqref{int2} involves all of the $f_{\tau_k}^k$, so we cannot rescale each $f_{\tau_k}^k$ individually and invoke induction. 

To address this problem, we use the more-detailed wave envelope estimate Theorem 1.3 from \cite{locsmooth} in place of the $L^4$ square function estimate to bound $g_k^h$. The wave envelope estimate was invented in \cite{locsmooth} as a tool to prove the $L^4$ square function estimate for the cone in $\R^3$ and it has two main features: (1) it is stronger than the square function estimate and (2) it behaves well in induction on scales. To simplify the explanation of our strategy, we will consider a special case that arises after applying the wave envelope estimate in place of the square function estimate for the cone. Otherwise repeating the reasoning that led to \eqref{int2}, we have
\begin{equation}\label{easy}  \a^7|\Omega_k|\lesssim_\e R^\e \int\sum_{\tau_k}|f_{\tau_k}^k|^7. \end{equation}
At first sight, the integral on the right hand side of \eqref{easy} looks similar to 
\begin{equation}\label{realeasy}
\sum_{\tau_k}\int|f_{\tau_k}|^7. \end{equation}
Since each $f_{\tau_k}$ is integrated individually, this expression can be handled by induction as follows. Let $\S(R)$ be the smallest constant for which 
\[ \int|\sum_\theta f_\theta|^7\le S(R)\int|\sum_\theta|f_\theta|^2|^{7/2} \]
for any $f$ satisfying the hypotheses of Theorem \ref{main}. Using affine rescaling of the moment curve, \eqref{realeasy} is bounded by
\[ S(R/R_k)\sum_{\tau_k}\int|\sum_{\theta\subset\tau_k}|f_\theta|^2|^{7/2}\le S(R/R_k)\int|\sum_{\theta}|f_\theta|^2|^{7/2} .  \]
Supposing that $R_k>R^{C_0\e}$, this would be a favorable scenario since the multi-scale inequality $S(R)\lesssim_\e R^{C\e} S(R/R_k)$ implies that $S(R)\lesssim_\d R^\d$ for any $\d>0$. The issue with this argument is that the $f_{\tau_k}^k$ in \eqref{easy} are different from the $f_{\tau_k}$ in \eqref{realeasy}. The pruning process which leads to the favorable $L^\infty$ bounds for $f_{\tau_k}^k$ also changes the Fourier support from $\cup_{\theta\subset\tau_k}\theta$ to potentially all of $\tau_k$. A rescaling argument may still be used to argue that
\[ \sum_{\tau_k}\int|f_{\tau_k}^k|^7\le S(R^\e)\sum_{\tau_k}\int|\sum_{\tau_{k+1}}f_{\tau_{k+1}}^{k+1}|^2|^{7/2},  \]
but then we need to analyze the expression on the right hand side. The idea is to dyadically decompose Fourier space into annuli and perform a similar analysis as our initial bounds for $g_k^h$, except using an $L^{7/2}$ wave envelope estimate. The bulk of the new technical work in this paper is to perform an ``unpruning" iteration (going from expressions with $f^k$ to $f^{k+1}$ and eventually to $f^N=f$) while carefully keeping track of constants. Eventually, we bound the right hand side of \eqref{easy} by an expression like 
\[ C_{\d,\e} R^{\d+\e} [S(R^\e)]^{\e^{-1}-C_0}\int|\sum_\theta|f_\theta|^2|^{7/2}  \]
where $\d>0$ may be arbitrarily small. This is again a favorable case since $S(R)\le C_{\d,\e} R^{\d+\e} S(R^\e)^{\e^{-1}-C_0}$ for any $\d>0$ and any $\e>0$ implies the desired bound for $S(R)$.

\noindent\underline{Small scale: $R_k\le R^{C_0\e}$. }
In this case, the remaining subset of $U_{\a,\b}$ is 
\[ L=\{x\in U_{\a,\b}:g_l\le A^{N-l}\b\quad\forall l=C_0,\ldots,N-1\}. \]
The argument for the intermediate scales case no longer works when $R_k$ is too small since we could conclude, for example, that $S(R)\le C_{\d,\e}R^\d S(R^\e)^{\e^{-1}}$. This would not lead to the desired bound $S(R)\lesssim_\d R^\d$. Instead, we use trivial bounds based on the definition of $L$. For each $x\in L$, for any $l=C_0,\ldots,N-1$, we may write
\[ \a^2=|\sum_{\tau_l}f_{\tau_l}(x)|^2\lesssim \#\tau_l\sum_{\tau_l}|f_{\tau_l}|^2(x)\lesssim R_l^{1/3}A^{N-l}\b\lesssim R_l^{1/3}A^{N-l}\sum_\theta|f_\theta|^2. \]
Then using $l=C_0$,
\[ \a^7|L|\lesssim R^{C_0\e} \int_{L}|\sum_\theta|f_\theta|^2|^{7/2},\]
which gives the bound $S(R)\lesssim R^{C_0\e}$ directly.

The conclusion of the high-low argument from each of the previous cases is that for any $\d,\e>0$ and $C_0>0$,
\[ S(R)\lesssim_{\d,\e} R^{\d+\e}\Big[R^{C_0\e}+S(R^\e)^{\e^{-1}-C_0}\Big]. \]
The exact version of the multi-scale inequality we prove is in Lemma \ref{multiscale'}. We show in \textsection\ref{ind} how the multi-scale inequality implies the desired bound for $S(R)$.

\subsection{Definitions of the square function constants $S_1(R)$ and $S_2(R)$ \label{const}}

Technically, the result of our high-low argument is 
\begin{equation}\label{ver1} \int_{\R^3}|f|^7\lesssim_\e R^\e \int_{\R^3}|\sum_\theta|f_\theta|^2*\w_\theta|^{7/2} \end{equation}
for appropriate $L^1$-normalized weight functions $\w_\theta$ adapted to $\theta^*$. While the locally constant property (see Lemma \ref{locconst}) tells us that pointwise, 
\[ \sum_\theta|f_\theta|^2(x)\lesssim\sum_\theta|f_\theta|^2*\w_\theta(x),\]
it is not generally true that the reverse inequality holds, either pointwise or in $L^{7/2}$. Indeed, let ${\bf{v}}_\theta\in\R^3$ be a unit vector in the direction of the $R$-long side of $\theta^*$, let $T_\theta$ be an $R^{2/3}\times R^{2/3}\times R$ tube centered at the origin with orientation ${\bf{v}}_\theta$, and let each $|f_\theta|^2(x)\approx \chi_{T_\theta}(x-R{\bf{v}}_\theta)$. Then 
\[ \int |\sum_\theta|f_\theta|^2|^{7/2}\sim \sum_\theta\int|f_\theta|^7\sim R^{8/3}\]
since the $|f_\theta|^2$ are essentially disjointly supported. On the other hand, after averaging, $\sum_\theta|f_\theta|^2*\w_\theta\gtrsim R^{1/3}$ on a ball of radius $\sim R^{2/3}$ centered at the origin, giving the lower bound
\[ \int|\sum_\theta|f_\theta|^2*\w_\theta|^{7/2}\gtrsim R^2(R^{1/3})^{7/2}=R^{19/6}. \]
Therefore, the version of a square function estimate \eqref{ver1} that we obtain from the high-low argument is ostensibly weaker than our goal. After proving \eqref{ver1}, we use an inductive argument to upgrade it to \eqref{sqfn}. 

\begin{definition} \label{S1df}Let $R\ge 10$. Let $S_1(R)$ be the infimum of $A>0$ such that
\[ \int_{\R^3}|f|^7\le A\int_{\R^3}|\sum_\theta|f_\theta|^2*\w_\theta|^{7/2} \]
for any Schwartz function $f:\R^3\to\C$ with Fourier transform supported in $\mc{M}^3(R)$. 
\end{definition}

\begin{definition} \label{S2df} Let $R\ge 10$. Let $S_2(R)$ be the infimum of $B>0$ such that
\[ \int_{\R^3}|f|^7\le B\int_{\R^3}|\sum_\theta|f_\theta|^2|^{7/2} \]
for any Schwartz function $f:\R^3\to\C$ with Fourier transform supported in $\mc{M}^3(R)$.  

\end{definition}

\section{Set-up for the high-low analysis \label{tools}}

In this section, we set-up the notation and basic properties of the high-low analysis of square functions at various scales. This is analogous to the high-low set-up from \cite{gmw}.

Begin with precise definitions of canonical blocks of the moment curve. 

\begin{definition}[Canonical moment curve blocks]
For $S\in 2^{\N}$, consider the anisotropic neighborhood 
\[\mc{M}^3(S^3)=\{(\xi_1,\xi_2,\xi_3): \xi_1\in[0,1],\,|\xi_2-\xi_1^2|\le S^{-2},\,|\xi_3-3\xi_1\xi_2+2\xi_1^3|\le S^{-3} \}.  \]
Define ${\bf{S}}(S^{-1})$ to be the following collection of canonical moment curve blocks at scale $S$ which partition $\mc{M}^3(S^3)$:
\[ \bigsqcup\limits_{l=0}^{S-1}\{(\xi_1,\xi_2,\xi_3): lS^{-1}\le \xi_1<(l+1)S^{-1},\,|\xi_2-\xi_1^2|\le S^{-2},\,|\xi_3-3\xi_1\xi_2+2\xi_1^3|\le S^{-3} \}.  \]
\end{definition}

The (not unit-normalized) Frenet frame for the moment curve $\mc{M}^3$ at $t\in[0,1]$ is 
\begin{align}
{\bf{T}}(t)&=(1,2t,3t^2)\nonumber \\
\label{Frenet} 
{\bf{N}}(t)&=(-2t-9t^3, 1-9t^4, 3t+6t^3)\\\nonumber
{\bf{B}}(t)&=(3t^2,-3t,1). 
\end{align}
If $\tau\in{\bf{S}}(S^{-1})$ is the $\ell$th moment curve block
\[ \{(\xi_1,\xi_2,\xi_3): lS^{-1}\le \xi_1<(l+1)S^{-1},\,|\xi_2-\xi_1^2|\le S^{-2},\,|\xi_3-3\xi_1\xi_2+2\xi_1^3|\le S^{-3} \},\]
then $\tau$ is comparable to the set
\[ \{\g(lS^{-1})+A{\bf{T}}(lS^{-1})+B{\bf{N}}(lS^{-1})+C{\bf{B}}(lS^{-1}): |A|\le S^{-1},|B|\le S^{-2},|C|\le S^{-3} \} . \]
By comparable, we mean that there is an absolute constant $C>0$ for which $C^{-1}\tau$ is contained in the displayed set and $C\tau$ contains the displayed set, where the dilations are taken with respect to the centroid of $\tau$. Define the dual set $\tau^*$ by 
\begin{equation}\label{dualdef}
\tau^* = \{A{\bf{T}}(lS^{-1})+B{\bf{N}}(lS^{-1})+C{\bf{B}}(lS^{-1}): |A|\le S,|B|\le S^2,|C|\le S^{3} \}.  
\end{equation}
We sometimes refer to the set $\tau^*$ as well as its translates as wave packets.

Next, we fix some notation for the scales. Let $\e>0$. To prove Theorem \ref{main}, it suffices to assume that $R$ is larger than a constant which depends on $\e$. Consider scales $R_k\in 8^\N$ closest to $R^{k\e}$, for $k=1,\ldots,N$ and $R_N\le R\le R^\e R_N$. Since $R$ differs from $R_N$ at most by a factor of $R^\e$, we will assume that $R=R_N$. The relationship between the parameters is
\[ 1=R_0\le R_k^{\frac{1}{3}}\le R_{k+1}^{\frac{1}{3}}\le R_N^{\frac{1}{3}}. \]

Fix notation for moment curve blocks of various sizes. 
\begin{enumerate}
    \item Let $\theta$ denote $\sim R^{-\frac{1}{3}}\times R^{-\frac{2}{3}}\times R^{-1}$ moment curve blocks from the collection ${\bf{S}}(R^{-1/3})$. 
    \item Let $\tau_k$ denote $\sim R_k^{-\frac{1}{3}}\times R_k^{-\frac{2}{3}}\times R_k^{-1}$ moment curve blocks from the collection ${\bf{S}}(R_k^{-1/3})$. 
\end{enumerate}
The definitions of $\theta,\tau_k$ provide the additional property that if $\tau_k\cap\tau_{k+m}\not=\emptyset$, then $\tau_{k+m}\subset\tau_k$ (and similarly for the $\tau_k$).

Fix a ball $B_R\subset\R^3$ of radius $R$ as well as a Schwartz function $f:\R^3\to\C$ with Fourier transform supported in $\mc{M}^3(R)$. The parameters $\a,\b>0$ describe the set 
\[ U_{\a,\b}=\{x\in B_{R}:|f(x)|\ge \a,\quad\frac{\b}{2}\le \sum_{\theta\in{\bf{S}}(R^{-1/3})}|f_\theta|^2*\w_{\theta}(x)\le \b\}.\]
The weight function $\w_{\theta}$ is defined in Definition \ref{M3ballweight} below. We assume throughout this section (and until \textsection\ref{M3pigeon}) that the $f_\theta$ satisfy the extra condition that
\begin{equation}\label{unihyp} 
\frac{1}{2}\le \|f_\theta\|_{L^\infty(\R^3)}\le 2\qquad\text{or}\qquad \|f_\theta\|_{L^\infty(\R^3)}=0. \end{equation}

\subsection{A pruning step \label{prusec}}

We define wave packets associated to $f_{\tau_k}$ and sort them according to an amplitude condition which depends on the parameters $\a$ and $\b$. 

For each $\tau_k$, let $\T_{\tau_k}$ contain $\tau_k^*$ and its translates $T_{\tau_k}$ which tile $\R^3$. Fix an auxiliary function $\p(\xi)$ which is a bump function supported in $[-\frac{1}{4},\frac{1}{4}]^3$. For each $m\in\Z^3$, let 
\[ \s_m(x)=c\int_{[-\frac{1}{2},\frac{1}{2}]^3}|\widecheck{\p}|^2(x-y-m)dy, \]
where $c$ is chosen so that $\sum_{m\in\Z^3}\s_m(x)=c\int_{\R^3}|\widecheck{\p}|^2=1$. Since $|\widecheck{\p}|$ is a rapidly decaying function, for any $n\in\N$, there exists $C_n>0$ such that
\[ \s_m(x)\le c\int_{[-\frac{1}{2},\frac{1}{2}]^3}\frac{C_n}{(1+|x-y-m|^2)^n}dy \le \frac{\tilde{C}_n}{(1+|x-m|^2)^n}. \]
Define the partition of unity $\s_{T_{\tau_k}}$ associated to ${\tau_k}$ to be $\s_{T_{\tau_k}}(x)=\s_m\circ A_{\tau_k}$, where $A_{\tau_k}$ is a linear transformations taking $\tau_k^*$ to $[-\frac{1}{2},\frac{1}{2}]^3$ and $A_{\tau_k}(T_{\tau_k})=m+[-\frac{1}{2},\frac{1}{2}]^3$. The important properties of $\s_{T_{\tau_k}}$ are (1) rapid decay off of $T_{\tau_k}$ and (2) Fourier support contained in $\tau_k$ translated to the origin. We sort the wave packets $\T_{\tau_k}=\T_{\tau_k}^g\sqcup\T_{\tau_k}^b$ into ``good" and ``bad" sets, and define corresponding versions of $f$, as follows. 


\begin{rmk} In the following definitions, let $K\ge 1$  be a large parameter which will be used to define the broad set in Proposition \ref{mainprop}. Also, $A=A(\e)\gg 1$ is a large enough constant (determined by Lemma \ref{ftofk}) which also satisfies $A\ge D$, where $D$ is from Lemma \ref{low}.
\end{rmk}

\begin{definition}[Pruning with respect to $\tau_k$]\label{taukprune} Let $f^{N}=f$, $f^{N}_{\tau_N}=f_{\theta}$. 
For each $1\le k\le N-1$, let 
\begin{align*} \T_{\tau_k}^{g}&=\{T_{\tau_k}\in\T_{\tau_{k}}:\|\s_{T_{\tau_{k}}}^{1/2}f_{\tau_{k}}^{k+1}\|_{L^\infty(R^3)}\le K^3A^{N-k+1}\frac{\b}{\a}\}, \\
f_{\tau_{k}}^{k}=\sum_{T_{\tau_k}\in\T_{\tau_k}^{g}}&\s_{T_{\tau_k}}f^{k+1}_{\tau_k}\qquad\text{and}\qquad  f_{\tau_{k-1}}^{k}=\sum_{\tau_k\subset\tau_{k-1}}f_{\tau_k}^k .
\end{align*}
\end{definition}
For each $k$, the $k$th version of $f$ is $f^k=\underset{\tau_k}{\sum} f_{\tau_k}^k$.

\begin{rmk}
We may assume that $\a\lesssim R^{C_0}\b$. This will be discussed in Proposition \ref{wpd} and Corollary \ref{wpdcor}, which involve pigeonholing the wave packets of $f$.
\end{rmk}

\begin{lemma}[Properties of $f^k$] \label{pruneprop}
\begin{enumerate} 
\item\label{item1} $| f_{\tau_{k}}^k (x) | \le |f_{ \tau_{k}}^{k+1}(x)|\lesssim \#\theta\subset\tau_k.$
\item \label{item2} $\| f_{\tau_k}^k \|_{L^\infty(\R^3)} \le K^3A^{N-k+1}\frac{\b}{\a}$.
\item\label{item3} For $R$ sufficiently large depending on $\e$, $\text{supp} \widehat{f_{\tau_k}^{k}}\subset3\tau_k. $
\end{enumerate}
\end{lemma}
\begin{proof} For the first property, recall that $\sum_{T_{\tau_k} \in \T_{\tau_k}}\s_{T_{\tau_k}}$ is a partition of unity so we may iterate the inequalities 
\begin{align*}
|f_{\tau_k}^k|\le |f_{\tau_k}^{k+1}|&\le \sum_{\tau_{k+1}\subset\tau_k}|f_{\tau_{k+1}}^{k+1}|\le\cdots\le \sum_{\tau_N\subset\tau_k}|f_{\tau_N}^N|= \sum_{\theta \subset\tau_k}|f_{\theta}|.  
\end{align*}
The first property follows from our assumption \eqref{unihyp} that each $\|f_\theta\|_{L^\infty(\R^3)}\lesssim 1$. For the $L^\infty$ bound in the second property, write
\[ |f_{ \tau_k}^k(x)| = |\sum_{\substack{T_{\tau_k} \in \T_{\tau_k^h}}} \s_{T_{\tau_k}}(x) f_{ \tau_k}^{k+1}(x)|\le \sum_{\substack{T_{\tau_k} \in \T_{\tau_k^h}}} \s_{T_{\tau_k}}^{1/2}(x) \|\s_{T_{\tau_k}}^{1/2}f_{ \tau_k}^{k+1}\|_\infty\lesssim \|\s_{T_{\tau_k}}^{1/2}f_{ \tau_k}^{k+1}\|_\infty. \]
\noindent By the definition of $\T_{\tau_k}^h$, $\|\s_{T_{\tau_k}}^{1/2}f_{\tau_k}^{k+1}\|_\infty\le K^3 A^{N-k+1}\frac{\b}{\a}$.

The third property depends on the Fourier support of $\s_{T_{\tau_k}}$, which is contained in $\tau_k$ shifted to the origin. Note if each $f_{\tau_k}^{k+1}$ has Fourier support in $\cup_{\tau_{k+1}\subset\tau_k}3\tau_{k+1}$, then $\supp\widehat{f_{\g_k}^k}$ is contained in $3\tau_k$. 

\end{proof}

\begin{definition} \label{M3ballweight} Let $\phi:\R^3\to\R$ be a smooth, radial function supported in $[-\frac{1}{4},\frac{1}{4}]^3$ and satisfying $|\widecheck{\phi}(x)|\ge 1$ when $|x|\le 1$. Then define $w:\R^3\to[0,\infty)$ by 
\[ w(x)=|\widecheck{\phi}|^2(x)+\sum_{k=0}^\infty\frac{1}{2^{100 k}}\int_{2^k\le |y|\le 2^{k+1}}|\widecheck{\phi}|^2(x-y)dy. \]
Let $B\subset\R^3$ denote the unit ball centered at the origin. For any set $U=T(B)$ where $T$ is an affine transformation $T:\R^3\to\R^3$, define
\[ w_{U}(x)=|U|^{-1}w(T^{-1}(x)). \]
For each $\tau_k$, let $A_{\tau_k}$ be a linear transformation mapping $\tau_k^*$ to the unit cube and define $\w_{\tau_k}$ by
\[  \w_{\tau_k}(x)=|\tau_k^*|^{-1}w(A_{\tau_k}(x)). \]
\end{definition}
Let the capital-W version of weight functions denote the $L^\infty$-normalized (as opposed to $L^1$-normalized) versions, so for example, for any ball $B_s$, $W_{B_s}(x)=|B_s|w_{B_s}(x)$. If a weight function has subscript which is only a scale, say $s$, then the functions $w_s,W_s$ are weight functions localized to the $s$-ball centered at the origin.

\begin{rmk}
Note the additional property that $\widehat{w}$ is supported in $[-\frac{1}{2},\frac{1}{2}]^3$, so each $w_{B_s}$ is Fourier supported in an $s^{-1}$-neighborhood of the origin. Finally, note the property that if $A_1,A_2$ are affine transformations of the unit ball and $A_1\subset A_2$, then $w_{A_1}*w_{A_2}\lesssim w_{A_2}$. 
\end{rmk}

Next, we record the locally constant property. By locally constant property, we mean that if a function $f$ has Fourier transform supported in a convex set $A$, then $|f|$ is bounded above by an averaged version of $|f|$ over a dual set $A^*$. 

\begin{lemma}[Locally constant property]\label{locconst} For each $\tau_k$ and $T_{\tau_k}\in\T_{\tau_k}$, 
\begin{align*} 
\|f_{\tau_k}\|_{L^\infty(T_{\tau_k})}^2\lesssim |f_{\tau_k}|^2*\w_{\tau_k}(x)\qquad\text{for any}\quad x\in T_{\tau_k} .\end{align*}
Also, for any $R_k^{1/3}$-ball $B_{R_k^{1/3}}$, 
\begin{align*} 
\|\sum_{\tau_k}|f_{\tau_k}|^2\|_{L^\infty(B_{R_k^{1/3}})}\lesssim |f_{\tau_k}|^2*w_{B_{R_k^{1/3}}}(x)\qquad\text{for any}\quad x\in B_{R_k^{1/3}} .\end{align*}
\end{lemma}
Because the pruned versions of $f$ and $f_{\tau_k}$ have essentially the same Fourier supports as the unpruned versions, the locally constant lemma applies to the pruned versions as well.

\begin{proof}[Proof of Lemma \ref{locconst}] For the first claim, we write the argument for $f_{\tau_k}$ in detail. Let $\rho_{\tau_k}$ be a bump function equal to $1$ on $\tau_k$ and supported in $2\tau_k$. Then using Fourier inversion and H\"{o}lder's inequality, 
\[ |f_{\tau_k}(y)|^2=|f_{\tau_k}*\widecheck{\rho_{\tau_k}}(y)|^2\le\|\widecheck{\rho_{\tau_k}}\|_1 |f_{\tau_k}|^2*|\widecheck{\rho_{\tau_k}}|(y). \]
Since $\rho_{\tau_k}$ may be taken to be an affine transformation of a standard bump function adapted to the unit ball, $\|\widecheck{\rho_{\tau_k}}\|_1$ is a constant. The function $\widecheck{\rho_{\tau_k}}$ decays rapidly off of $\tau_k^*$, so $|\widecheck{\rho_{\tau_k}}|\lesssim \w_{{\tau_k}}$.
Since for any $T_{\tau_k}\in\T_{\tau_k}$, $\w_{\tau_k}(y)\sim\w_{\tau_k}(y')$ for all $y,y'\in T_{\tau_k}$, we have
\begin{align*} \sup_{x\in T_{\tau_k}}|f_{\tau_k}|^2*\w_{\tau_k}(x)&\le \int|f_{\tau_k}|^2(y)\sup_{x\in T_{\tau_k}}\w_{\tau_k}(x-y)dy\\
&\sim \int|f_{\tau_k}|^2(y)\w_{\tau_k}(x-y)dy\qquad \text{for all}\quad x\in T_{\tau_k}. 
\end{align*}

For the second part of the lemma, repeat analogous steps as above, except begin with $\rho_{\tau_k}$ which is identically $1$ on a ball of radius $2R_k^{-1/3}$ containing $\tau_k$. Then 
\[  \sum_{\tau_k}|f_{\tau_k}(y)|^2=\sum_{\tau_k}|f_{\tau_k}*\widecheck{\rho_{\tau_k}}(y)|^2\lesssim \sum_{\tau_k}|f_{\tau_k}|^2*|\widecheck{\rho_{R_k^{-1/3}}}|(y),\]
where we used that each $\rho_{\tau_k}$ is a translate of a single function $\rho_{R^{-1/3}}$. The rest of the argument is analogous to the first part. 
\end{proof}

The following local $L^2$-orthogonality lemma which is Lemma 3 in \cite{M3smallcap}.  
\begin{lemma}[Local $L^2$ orthogonality]\label{L2orth} Let $U=T(B)$ where $B$ is the unit ball centered at the origin and $T:\R^3\to\R^3$ is an affine transformation. Let $h:\R^3\to\C$ be a Schwartz function with Fourier transform supported in a disjoint union $X=\sqcup_k X_k$, where $X_k\subset B$ are Lebesgue measurable. If the maximum overlap of the sets $X_k+U^*$ is $L$, then
\[ \int |h_X|^2w_U\lesssim L\sum_{X_k}\int|h_{X_k}|^2w_U, \]
where $h_{X_k}=\int_{X_k}\widehat{h}(\xi)e^{2\pi i x\cdot\xi}d\xi$.
\end{lemma}
Here, we may take $\{x:|x\cdot\xi|\le 1\quad\forall \xi\in U-U\}$ as the definition of $U^*$. We will include a sketch of the proof for future reference. 
\begin{proof} By Plancherel's theorem, we have
\begin{align*}
    \int|h_X|^2w_U&=\int h_X \overline{h_Xw_U}=\int \widehat{h_X}\overline{\widehat{h_X}*\widehat{w_U}}.
\end{align*}
Since $\widehat{h_X}=\sum_k\widehat{h_{X_k}}$, $\int \widehat{h_X}\overline{\widehat{h_X}*\widehat{w_U}}=\sum_{X_k}\sum_{X_k'}\int\widehat{h_{X_k}}\overline{\widehat{h_{X_k'}}*\widehat{w_U}}$. For each $X_k$, the integral on the right hand side vanishes except for $\lesssim L$ many choices of $X_k'$. 

\end{proof}

\subsection{High-low frequency decomposition of square functions}

\begin{definition}[Auxiliary functions] Let $\eta:\R^3\to[0,\infty)$ be a radial, smooth bump function satisfying $\eta(x)=1$ on $B_{1/2}$ and $\supp\eta\subset B_1$. Then for each $s>0$, let 
\[ \eta_{\le s}(\xi) =\eta(s^{-1}\xi) .\]
We will sometimes abuse notation by denoting $h*\widecheck{\eta}_{>s}=h-h*\widecheck{\eta}_{\le s}$, where $h$ is some Schwartz function. Also define $\eta_{s}(x)=\eta_{\le s}-\eta_{\le s/2}$. 
\end{definition}

\vspace{3mm}
Fix $N_0<N-1$ which will be specified in \textsection\ref{ind}. 
\begin{definition} For $N_0\le k\le N-1$, let 
\[ g_k(x)=\sum_{\tau_k}|f_{\tau_k}^{k+1}|^2*\w_{\tau_k}, \qquad g_k^{\ell}(x)=g_k*\widecheck{\eta}_{\le R_{k+1}^{-1/3}}, \qquad\text{and}\qquad g_k^h=g_k-g_k^{\ell}. \]
\end{definition}
\vspace{3mm}

In the following definition, $A\gg 1$ is the same constant that goes into the pruning definition of $f^k$. 
\begin{definition} \label{impsets}Define the high set by 
\[ H=\{x\in U_{\a,\b}: A \b \le g_{N-1}(x)\}. \]
For each $k=N_0,\ldots,N-2$, let $H=\Omega_{N-1}$ and let
\[ \Omega_k=\{x\in U_{\a,\b}\setminus \cup_{l=k+1}^{N-1}\Omega_{l}: A^{N-k}\b\le g_k(x) \}. \] 
Define the low set to be
\[ L=U_{\a,\b}\setminus[\cup_{k=N_0}^{N-1}\Omega_k]. \]
\end{definition}
\vspace{3mm}

\begin{lemma}[Low lemma]\label{low} There is an absolute constant $D>0$ so that for each $x$, $|g_k^\ell(x)|\le D g_{k+1}(x)$. 
\end{lemma}
\begin{proof} We perform a pointwise version of the argument in the proof of local/global $L^2$-orthogonality (Lemma \ref{L2orth}). For each $\tau_k^{k+1}$, by Plancherel's theorem,
\begin{align}
|f_{\tau_k}^{k+1}|^2*\widecheck{\eta}_{<R_{k+1}^{-1/3}}(x)&= \int_{\R^3}|f_{\tau_k}^{k+1}|^2(x-y)\widecheck{\eta}_{<R_{k+1}^{-1/3}}(y)dy \nonumber \\
&=  \int_{\R^3}\widehat{f_{\tau_k}^{k+1}}*\widehat{\overline{f_{\tau_k}^{k+1}}}(\xi)e^{-2\pi i x\cdot\xi}\eta_{<R_{k+1}^{-1/3}}(\xi)d\xi \nonumber \\
&=  \sum_{\tau_{k+1},\tau_{k+1}'\subset\tau_k}\int_{\R^3}e^{-2\pi i x\cdot\xi}\widehat{f_{\tau_{k+1}}^{k+1}}*\widehat{\overline{f_{\tau_{k+1}'}^{k+1}}}(\xi)\eta_{<R_{k+1}^{-1/3}}(\xi)d\xi .\label{dis2}\nonumber
\end{align}
The integrand is supported in $(2\tau_{k+1}-2\tau_{k+1}')\cap B_{R_{k+1}^{-1/3}}$. This means that the integral vanishes unless $\tau_{k+1}$ is within $\sim  R_{k+1}^{-1/3}$ of $\tau_{k+1}'$, in which case we write $\tau_{k+1}\sim\tau_{k+1}'$. Then 
\[\sum_{\tau_{k+1},\tau_{k+1}'\subset\tau_k}\int_{\R^2}e^{-2\pi i x\cdot\xi}\widehat{f}_{\tau_{k+1}}^{k+1}*\widehat{\overline{f}_{\tau_{k+1}'}^{k+1}}(\xi)\eta_{<R_{k+1}^{-1/3}}(\xi)d\xi=\sum_{\substack{\tau_{k+1},\tau_{k+1}'\subset\tau_k\\
\tau_{k+1}\sim\tau_{k+1}'}}\int_{\R^2}e^{-2\pi i x\cdot\xi}\widehat{f}_{\tau_{k+1}}^{k+1}*\widehat{\overline{f}_{\tau_{k+1}'}^{k+1}}(\xi)\eta_{<R_{k+1}^{-1/3}}(\xi)d\xi. \]
Use Plancherel's theorem again to get back to a convolution in $x$ and conclude that
\begin{align*}
|g_k*\widecheck{\eta}_{<R_{k+1}^{-1/3}}(x)|&=\Big|\sum_{\substack{\tau_{k+1},\tau_{k+1}'\subset\tau_k\\
\tau_{k+1}\sim\tau_{k+1}'}}(f_{\tau_{k+1}}^{k+1}\overline{f_{\tau_{k+1}'}^{k+1}})*\w_{\tau_k}*\widecheck{\eta}_{<R_{k+1}^{-1/3}}(x) \Big|\lesssim \sum_{\tau_k} \sum_{\tau_{k+1}\subset\tau_k}|f_{\tau_{k+1}}^{k+1}|^2*\w_{\tau_k}*|\widecheck{\eta}_{<R_{k+1}^{-1/3}}|(x)
. 
\end{align*}
By the locally constant property (Lemma \ref{locconst}) and \eqref{item1} of Lemma \ref{pruneprop},
\[ \sum_{\tau_k} \sum_{\tau_{k+1}\subset\tau_k}|f_{\tau_{k+1}}^{k+1}|^2*\w_{\tau_k}*|\widecheck{\eta}_{<R_{k+1}^{-1/3}}|(x)\lesssim \sum_{\tau_k} \sum_{\tau_{k+1}\subset\tau_k}|f_{\tau_{k+1}}^{k+2}|^2*w_{\tau_{k+1}}*\w_{\tau_k}*|\widecheck{\eta}_{<R_{k+1}^{-1/3}}|(x)\lesssim g_{k+1}(x). \]
It remains to note that
\[ w_{\tau_{k+1}}*\w_{\tau_k}*|\widecheck{\eta}_{<R_{k+1}^{-1/3}}|(x)\lesssim w_{\tau_{k+1}}(x) \]
since $\tau_k^*\subset\tau_{k+1}^*$ and $\widecheck{\eta}_{<R_{k+1}^{-1/3}}$ is an $L^1$-normalized function that is rapidly decaying away from $B_{R_{k+1}^{1/3}}(0)$. 

\end{proof}

\begin{corollary}[High-dominance on $\Omega_k$]\label{highdom} For $R$ large enough depending on $\e$, $g_k(x)\le 2|g_k^h(x)|$ for all $x\in\Omega_k$. 
\end{corollary}
\begin{proof}
This follows directly from Lemma \ref{low}. Indeed, since $g_k(x)=g_k^{\ell}(x)+g_k^h(x)$, the inequality $g_k(x)>2|g_k^h(x)|$ implies that $g_k(x)<2|g_k^{\ell}(x)|$. Then by Lemma \ref{low}, $|g_k(x)|<2D g_{k+1}(x)$. Since $x\in\Omega_k$, $g_{k+1}(x)\le A^{N-k-1}\b$, which altogether gives the upper bound
\[ g_k(x)\le 2D A^{N-k-1}\b. \]
The contradicts the property that on $\Omega_k$, $A^{N-k}\b\le g_k(x)$, for $A$ sufficiently larger than $D$, which finishes the proof. 

\end{proof}

\begin{lemma}[Pruning lemma]\label{ftofk} For any $s\ge R^{-\e/3}$ and $\tau\in{\bf{S}}(s)$, 
\begin{align*} 
|\sum_{\tau_k\subset\tau}f_{\tau_k}-\sum_{\tau_k\subset\tau}f_{\tau_k}^{k+1}(x)|&\le \frac{\a}{A^{1/2}K^3} \qquad\text{for all $x\in \Omega_k$}, \qquad N_0\le k\le N-1,\\
\text{and}\qquad |\sum_{\tau_B\subset\tau}f_{\tau_B}-\sum_{\tau_B\subset\tau}f_{\tau_B}^{B}(x)|&\le \frac{\a}{A^{1/2}K^3}\qquad \text{ for all $x\in L$}. \end{align*}
\end{lemma}

\begin{proof} 
Begin by proving the first claim about $\Omega_k$. By the definition of the pruning process, we have 
\begin{equation}\label{diffs} f_{\tau}=f^{N-1}_{\tau}+(f_{\tau}^N-f^{N-1}_{\tau})=\cdots=f^{k+1}_{\tau}(x)+\sum_{m=k+1}^{N-1}(f^{m+1}_{\tau}-f^{m}_{\tau})\end{equation}
where formally, the subscript $\tau$ means $f_\tau=\sum_{\theta\subset\tau}f_\theta$ and $f_{\tau}^m=\sum_{\tau_m\subset\tau}f_{\tau_m}^m$. We will show that each difference in the sum is much smaller than $\a$.
For each $N-1\ge m\ge k+1$ and $\tau_m$, 
\begin{align*}
    |f_{\tau_m}^m(x)-f_{\tau_m}^{m+1}(x)|&=|\sum_{T_{\tau_m}\in\T_{\tau_m}^{b}}\s_{T_{\tau_m}}(x)f_{\tau_m}^{m+1}(x)|  = \sum_{T_{\tau_m}\in T_{\tau_m}^b} |\s_{T_{\tau_m}}^{1/2}(x)f_{\tau_m}^{m+1}(x)|\s_{T_{\tau_m}}^{1/2}(x) \\
     & \le\sum_{T_{\tau_m}\in \T_{\tau_m}^b}  K^{-3}A^{-(N-m+1)}\frac{\a}{\b}  \| \s_{T_{\tau_m}}^{1/2}f_{{\tau_m}}^{m+1} \|_{L^\infty(\R^3)}^2  \s_{T_{\tau_m}}^{1/2}(x) \\
     & \lesssim K^{-3}A^{-(N-m+1)}\frac{\a}{\b}\sum_{T_{\tau_m}\in \T_{\tau_m}^b}
      \sum_{\tilde{T}_{{\tau_m}}\in\T_{\tau_m}} \| \s_{T_{\tau_m}}|f_{{\tau_m}}^{m+1}|^2 \|_{L^\infty(\tilde{T}_{{\tau_m}})} \s_{T_{\tau_m}}^{1/2}(x) \\
     & \lesssim K^{-3}A^{-(N-m+1)}\frac{\a}{\b} \sum_{T_{\tau_m},\tilde{T}_{\tau_m}\in \T_{\tau_m}} \| \s_{T_{\tau_m}}\|_{L^\infty(\tilde{T}_{\tau_m})}\||f_{{\tau_m}}^{m+1} |^2\|_{{L}^\infty(\tilde{T}_{{\tau_m}})} \s_{T_{\tau_m}}^{1/2}(x) .
\end{align*}
Let $c_{\tilde{T}_{\tau_m}}$ denote the center of $\tilde{T}_{\tau_m}$ and note the pointwise inequality
\[ \sum_{{T}_{\tau_m}}\|\s_{T_{\tau_m}}\|_{L^\infty(\tilde{T}_{\tau_m})}\s_{T_{\tau_m}}^{1/2}(x)\lesssim |\tau_m^*|\w_{\tau_m}(x-c_{\tilde{T}_{\tau_m}}) ,\]
which means that
\begin{align*}
|f_{\tau_m}^m(x)-f_{\tau_m}^{m+1}(x)| & \lesssim K^{-3}A^{-(N-m+1)}\frac{\a}{\b} |\tau_m^*|\sum_{\tilde{T}_{\tau_m}\in \T_{\tau_m}} \w_{\tau_m}(x-c_{\tilde{T}_{\tau_m}})\||f_{{\tau_m}}^{m+1} |^2\|_{{L}^\infty(\tilde{T}_{{\tau_m}})} \\
&\lesssim K^{-3}A^{-(N-m+1)}\frac{\a}{\b}|\tau_m^*| \sum_{\tilde{T}_{\tau_m}\in \T_{\tau_m}} \w_{\tau_m}(x-c_{\tilde{T}_{\tau_m}})|f_{{\tau_m}}^{m+1} |^2*\w_{\tau_m}(c_{\tilde{T}_{\tau_m}})\\
&\lesssim K^{-3}A^{-(N-m+1)}\frac{\a}{\b} |f_{{\tau_m}}^{m+1} |^2*\w_{\tau_m}(x)
\end{align*}
where we used the locally constant property in the second to last inequality. The last inequality is justified by the fact that $\w_{\tau_m}(x-c_{\tilde{T}_{\tau_m}})\sim \w_{\tau_m}(x-y)$ for any $y\in\tilde{T}_{\tau_m}$, and we have the pointwise relation $\w_{\tau_m}*\w_{\tau_m}\lesssim \w_{\tau_m}$. 
Then 
\[
    |\sum_{\tau_m\subset\tau}(f_{\tau_m}^m(x)-f_{\tau_m}^{m+1}(x))|\lesssim K^{-3}A^{-(N-m+1)}\frac{\a}{\b}\sum_{\tau_m\subset\tau}|f_{\tau_m}^{m+1}|^2*\w_{\tau_m}(x)\sim K^{-3}A^{-(N-m+1)}\frac{\a}{\b}g_m(x). \]
At this point, choose $A$ sufficiently large determined by the proof of Corollary \ref{highdom} and so that if $g_m(x)\le A^{N-m}\b$, then the above inequality implies that
\[ |\sum_{\tau_m\subset\tau}(f_{\tau_m}^m(x)-f_{\tau_m}^{m+1}(x))|\le  \e K^{-3}A^{-1/2}\a  .\]
This finishes the proof since the number of terms in \eqref{diffs} is bounded by $N\le \e^{-1}$. The argument for the pruning on $L$ is analogous.  
\end{proof}

\section{Geometry for the cone and the moment curve \label{geo}}

We have seen in Corollary \ref{highdom} that on $\Omega_k$, $g_k$ is high-dominated.  We will now describe how to use the wave envelope estimate for the cone (Theorem 1.3 from \cite{locsmooth}) to control the high part of $g_k$.

Begin by describing the cone set-up using the rotated coordinate system from \textsection{5} of \cite{locsmooth}. Define the truncated cone by $\Gamma = \{r(1,\w,\frac{1}{2}\w^2):\frac{1}{2}\le r\le 1,\quad|\w|\le 1\}$. We consider the neighborhood  
\[ \Gamma(S^2)=\{(\nu_1,\nu_2,\nu_3): \frac{1}{2}\le \nu_1\le 1,\,\, |\nu_3-\frac{1}{2}\frac{\nu_2^2}{\nu_1}|\le S^{-2}  \}  \]
where $S\ge 1$ is dyadic. 
For $|\w|\le 1$, we will use the (not unit-normalized) frame
\[ \begin{cases}
{\bf{c}}(\w)&=(1,\w,\frac{1}{2}\w^2) \\
{\bf{b}}(\w)&=(-\w,1-\frac{1}{2}\w^2,\w) \\
{\bf{t}}(\w)&=(\frac{1}{2}\w^2,-\w,1) 
\end{cases}.  \]
Define a cone plank $\tau$ of dimension  $1\times S^{-1}\times S^{-2}$ and centered at $(1,\w,\frac{1}{2}\w^2)\in\Gamma$ by 
\begin{equation}\label{rotconeplank}
    \tau=\{A{\bf{c}}(\w)+B{\bf{b}}(\w)+C{\bf{t}}(\w):\quad \frac{1}{2}\le A\le 1,\quad |B|\le S^{-1},\quad |C|\le S^{-2}  \}.
\end{equation}
Let ${\bf{S}}_{S^{-1}}$ denote a collection of $1\times S^{-1}\times S^{-2}$ conical blocks $\tau$ which approximately partition $\Gamma(S^2)$. The definitions and notation we use are compatible with those in \textsection3 of \cite{locsmooth}.

Recall that our goal in this section is to bound the high part of $g_k$. The summands of $g_k$ have Fourier support in $2\tau_k-2\tau_k$, where $\tau_k\in{\bf{S}}(R_k^{-1/3})$ (noting that ${\bf{S}}(R_k^{-1/3})$ refers to \emph{moment curve} blocks defined at the beginning of \textsection\ref{tools}). Removing the high part cuts away a ball of radius $R_{k+1}^{-1/3}$. In Proposition \ref{geo3}, we will essentially show that $(2\tau_k-2\tau_k)\setminus B_{R_{k+1}^{-1/3}}$ may be identified with a \emph{conical} plank $\tau\in{\bf{S}}_{R_k^{-1/3}}$. 

Let $S\ge1$ be a dyadic parameter that will be chosen to be sufficiently large in Proposition \ref{geo3}. 
Suppose that $\tau\in{\bf{S}}(S^{-1})$ is the $l$th piece, meaning that 
\[ \tau=\{(\xi_1,\xi_2,\xi_3): lS^{-1}\le \xi_1<(l+1)S^{-1},\,|\xi_2-\xi_1^2|\le S^{-2},\,|\xi_3-3\xi_1\xi_2+2\xi_1^3|\le S^{-3} \}   \]
where $l\in\{0,\ldots,S-1\}$. Using the Frenet frame description from \eqref{Frenet}, the set $(10\tau-10\tau)\setminus B_{(4S)^{-1}}(0)$ is contained in 
\begin{align}\label{set} \tilde{\tau}:=\{A{\bf{T}}(lS^{-1})+B{\bf{N}}(lS^{-1})+C{\bf{B}}(lS^{-1}):A\sim  S^{-1},\quad |B|\lesssim S^{-2},\quad |C|\lesssim S^{-3}\}. \end{align}
Define the linear transformation 
$T:\R^3\to\R^3$ by
\begin{equation}\label{lintrans} T(x,y,z):=(x,\frac{y}{2},\frac{z}{6}). \end{equation} 
\begin{proposition} \label{geo3} Let $S\in 2^\N$ be larger than some absolute constant. After dilating by $S$, the sets $T[(10\tau-10\tau)\setminus B_{(4S)^{-1}}(0)]$ with $\tau\in{\bf{S}}(S^{-1})$, are comparable to the cone planks from ${\bf{S}}_{S^{-1}}$.  \end{proposition}

\begin{proof} The image of the set \eqref{set} under $T$ is
\begin{align*}
T(\tilde{\tau})=\{A(1,lS^{-1},\frac{1}{2}(lS^{-1})^2)+&BT({\bf{N}}(lS^{-1}))+C(\frac{1}{2}T({\bf{B}}(lS^{-1})): A\sim S^{-1},\quad |B|\lesssim S^{-2},\quad |C|\le S^{-3}\}.
\end{align*}
Define $\w=lS^{-1}$. Since $T({\bf{T}}(lS^{-1}))={\bf{c}}(lS^{-1})$ and $T({\bf{N}}(lS^{-1}))\cdot{\bf{t}}(lS^{-1})=0$, it is easy to see that the set $T(\tilde{\tau})$ is comparable to the $S^{-1}$ dilation of \eqref{rotconeplank}.

\end{proof}

Next, we define moment curve wave envelopes, which are roughly the smallest convex sets containing wave packets from neighboring moment curve blocks.  
\begin{definition}\label{waveenvdef} Let $1\le S^3\le R $. For $\tau\in{\bf{S}}(S^{-1})$ which is the $\ell$th block, define the moment curve wave envelope $V_{\tau,R}$ to be  
\[V_{\tau,R}=\{A{\bf{T}}(lS^{-1})+B{\bf{N}}(lS^{-1})+C{\bf{B}}(lS^{-1}): |A|\le S^{-2}R,\quad|B|\le S^{-1}R,\quad|C|\le R \} . \]
\end{definition}

We will compare these with wave envelopes for the cone defined in \textsection 3 from \cite{locsmooth}. Again let $1\le S^3\le R$. Using our notation, if $\tau'\in{\bf{S}}_{S^{-1}}$ is given by \eqref{rotconeplank}, then the cone wave envelope $U_{\tau',R^{2/3}}$ is defined by
\begin{align} \label{rotconeenv} 
U_{\tau',R^{2/3}}=\{A{\bf{c}}(\w)+&B{\bf{b}}(\w)+C{\bf{t}}(\w): |A|\le S^{-2}R^{2/3},\quad |B|\le S^{-1}R^{2/3},\quad |C|\le R^{-2/3}  \}. \nonumber
\end{align}

After applying a linear transformation, the $R^{-1/3}$-dilation of the moment curve wave envelope $V_{\tau,R}$ is comparable to a cone wave envelope $U_{\tau',R^{2/3}}$. 

\begin{proposition} \label{geo4} Let $T$ be the linear transformation \eqref{lintrans} and let $S$ be a sufficiently large dyadic number with $1\le S^3\le R$. Then for each $\tau\in{\bf{S}}(S^{-1})$, there is a $\tau'\in{\bf{S}}_{S^{-1}}$ for which $R^{-1/3}\cdot [(T^t)^{-1}V_{\tau,R}]$ is comparable to $U_{\tau',R^{2/3}}$.  \end{proposition}

\[ \]
\begin{proof} Note that $(T^t)^{-1}(x,y,z)=(x,2y,6z)$. The $R^{-1/3}$-dilation of the image of the wave envelope $V_{\tau,R}$ under $(T^t)^{-1}$ is
\begin{align*}
R^{-1/3}\cdot [(T^t)^{-1}V_{\tau,R}]=\{A(1,4lS^{-1},&18(lS^{-1})^2)+B(-2lS^{-1}-9(lS^{-1})^3,2-18(lS^{-1})^4,18lS^{-1}+36(lS^{-1})^3)\\
&+C(3(lS^{-1})^2,-6lS^{-1}, 6): |A|\lesssim S^{-2}R^{2/3},\quad |B|\lesssim S^{-1}R^{2/3},\quad |C|\lesssim R^{2/3}\}. 
\end{align*}
Applying a Gram-Schmidt process to the vectors, we see that the above set is comparable to 
\begin{align*}
\{A(1,lS^{-1},&\frac{1}{2}(lS^{-1})^2)+B(-lS^{-1},1-\frac{1}{2}(lS^{-1})^2,lS^{-1})\\
&+C(\frac{1}{2}(lS^{-1})^2,-lS^{-1}, 1): |A|\lesssim S^{-2}R^{2/3},\quad |B|\lesssim S^{-1}R^{2/3},\quad |C|\lesssim R^{2/3}\}. 
\end{align*}
Define $\w=lS^{-1}$. Then it is clear that the above set is comparable to $U_{\tau',R^{2/3}}$ where  $\tau'\in{\bf{S}}_{S^{-1}}$ is centered at $(1,\w,\frac{1}{2}\w^2)$, as require by the proposition. 

\end{proof}

\subsection{High-frequency analysis }

Now that we have identified moment curve blocks with cone planks and moment curve wave envelopes with cone wave envelopes, we are prepared to use Theorem 1.3 from \cite{locsmooth} to control the high part of square functions. Recall the theorem statement.
\begin{thm}[Theorem 1.3 from \cite{locsmooth}]\label{wee} For each $\d>0$, there exists $B_\d\in(0,\infty)$ so that 
\[ \int_{\R^3}|f|^4\le B_\d R^\d\sum_{R^{-1/2}\le\sigma\le 1}\sum_{\tau\in{\bf{S}}_{\sigma^{-1}R^{-1/2}}}\sum_{U\|U_{\tau,R}}|U|^{-1}\Big(\int_U\sum_{\substack{\theta\subset\tau\\\theta\in{\bf{S}}_{R^{-1/2}}}}|f_\theta|^2\Big)^2 \]
for any Schwartz function $f:\R^3\to\C$ with Fourier transform supported in $\Gamma(R)$. 

\end{thm}
The initial sum on the right hand side is over dyadic $\sigma$, $R^{-1/2}\le\sigma\le 1$. Whenever we sum over an interval $(a,b)$, we always mean the numbers in $2^\Z\cap(a,b)$.

\begin{lemma}[High lemma]\label{high1} For each $\d>0$, there is $B_\d\in(0,\infty)$ so that the following holds. For any $k$, there is some dyadic scale $R_{k+1}^{-1/3}\le s\le 10 R_k^{-1/3}$ for which 
\[ \int|g_k^h|^4\le B_{\d}R^{\d}(\log R)\sum_{s\le\sigma\le 1}\sum_{\tau\in{\bf{S}}(\sigma^{-1}s)}\sum_{V\|V_{\tau,s^{-3}}}|V|^{-1}\Big(\int_V\sum_{\substack{\tau_s\subset\tau\\\tau_s\in{\bf{S}}(s)}}|f_{\tau_s}^{k+1}|^4\Big)^2 .\]
\end{lemma}

\begin{proof} First describe the Fourier support of $g_k^h$. By \eqref{item3} of Lemma \ref{pruneprop}, the support of $\widehat{|f_{\tau_k}^{k+1}|^2}$ is $2(\tau_k-\tau_k)$. The high-frequency cutoff removes a ball of radius $R_{k+1}^{-1/3}$, so $g_k^h$ is Fourier supported within the annulus $R_{k+1}^{-1/3}\le |\xi|\le 10 R_k^{-1/3}$. By dyadic pigeonholing, there is some dyadic $s\in[R_{k+1}^{-1/3},R_k^{-1/3}]$ for which 
\[ \int|g_k^h|^4\lesssim(\log R)  \int|g_k^h*\widecheck{\eta}_s|^4, \]
where $\eta_s:\R^3\to[0,\infty)$ is a smooth function supported in the annulus $s/4\le |\xi|\le s$ (in the case that $s=R_k^{-1/3}$, let $\eta_s$ be supported on $s\le |\xi|\le 20R_k^{-1/3}$). In the proof of Lemma \ref{low}, we showed the pointwise equality
\begin{equation}\label{ptwiselo}g_k^h*\widecheck{\eta}_s(x)=\sum_{\tau_s}\sum_{\tau_s'\sim\tau_s}(f_{\tau_s}^{k+1}\overline{f_{\tau_s'}^{k+1}})*\w_{\tau_k}*\widecheck{\eta}_{>R_{k+1}^{-1/3}}*\widecheck{\eta}_s(x)\end{equation}
where $\tau_s\in{\bf{S}}(s)$ and $\tau_s'\sim\tau_s$ means that $\tau_s'\in{\bf{S}}(s)$ and $\text{dist}(2\tau_s,2\tau_s')\le  2s$. For each $\tau_s$, the sub-sum on the right hand side has Fourier transform supported in $10(\tau_{2s}-\tau_{2s})\setminus B_s(0)$ where $\tau_{2s}\in{\bf{S}}(2s)$ contains $\tau_s$. Now write 
\begin{equation}\label{formg} g_k^h*\widecheck{\eta}_s(x)=|\det T|^{-1}(\widehat{g_k^h}\eta_s\circ T^{-1})^{\widecheck{\,\,\,}}((T^{-1})^t x) . \end{equation}
Perform the change of variables $x\mapsto T^tx$ to get
\[ \int|g_k^h*\widecheck{\eta}_s(x)|^4dx=|\det T|^{-3}\int|(\widehat{g_k^h}\eta_s\circ T^{-1})^{\widecheck{\,\,\,}}(x)|^4dx. \]
By Proposition \ref{geo3}, we may view the sub-sums corresponding to each $\tau_k$ on the right hand side of \eqref{ptwiselo} as having Fourier support which is part of a tiling of the cone, after applying $T$ and dilating by a factor of $s^{-1}$. Therefore, we may apply the wave envelope estimate Theorem \ref{wee}, dilated by a factor of $s^{-1}$, to obtain 
\begin{align*}
\int|(\widehat{g_k^h}\eta_s*\circ T^{-1})^{\widecheck{\,\,\,}}&(x)|^4dx\lesssim_{\d}R^\d\sum_{s^{-1}\le\sigma\le 1}\sum_{\tau'\in{\bf{S}}_{\sigma^{-1}s}}\sum_{U\|U_{\tau',s^{-2}}}\\
&|s^{-1}U|^{-1}\Big(\int_{s^{-1} U}\sum_{\tau_{2s}``\subset" \tau'}|\sum_{\tau_s\subset\tau_{2s}}\sum_{\tau_s'\sim\tau_s}((\widehat{f_{\tau_s}^{k+1}}*\widehat{\overline{f_{\tau_s'}^{k+1}}})\widehat{\w}_{\tau_k}{\eta}_{>R_{k+1}^{-1/3}}{\eta}_s\circ T^{-1})^{\widecheck{\,\,\,}}(x)|^2dx\Big)^{2}
\end{align*}
where $\tau_{2s}``\subset"\tau'$ means that $(2s)^{-1}\cdot[ T(10\tau_{2s}-10\tau_{2s})\setminus B_{s}(0)]\subset\tau'$. It remains to undo the initial steps which allowed us to apply the wave envelope estimate for the cone. First multiply both sides of the above inequality by $|\det T|^{-3}$. Then do the change of variables $x\mapsto (T^{-1})^tx$ to obtain
\begin{align*}
|\det T&|^{-3}|s^{-1}U|^{-1}\Big(\int_{s^{-1} U}\sum_{\tau_{2s}``\subset" \tau'}|\sum_{\tau_s\subset\tau_{2s}}\sum_{\tau_s'\sim\tau_s}((\widehat{f_{\tau_s}^{k+1}}*\widehat{\overline{f_{\tau_s'}^{k+1}}})\widehat{\w}_{\tau_k}{\eta}_{>R_{k+1}^{-1/3}}{\eta}_s\circ T^{-1})^{\widecheck{\,\,\,}}(x)|^2dx\Big)^{2}\\
&= |\det T|^{-4}|s^{-1}T^tU|^{-1}\Big(\int_{s^{-1} T^tU}\sum_{\tau_{2s}``\subset" \tau'}|\sum_{\tau_s\subset\tau_{2s}}\sum_{\tau_s'\sim\tau_s}((\widehat{f_{\tau_s}^{k+1}}*\widehat{\overline{f_{\tau_s'}^{k+1}}})\widehat{\w}_{\tau_k}{\eta}_{>R_{k+1}^{-1/3}}{\eta}_s\circ T^{-1})^{\widecheck{\,\,\,}}((T^{-1})^t x)|^2dx\Big)^{2}\\
&= |s^{-1}T^tU|^{-1}\Big(\int_{s^{-1} T^tU}\sum_{\tau_{2s}``\subset" \tau'}|\sum_{\tau_s\subset\tau_{2s}}\sum_{\tau_s'\sim\tau_s}({f_{\tau_s}^{k+1}}{\overline{f_{\tau_s'}^{k+1}}})*{\w}_{\tau_k}*\widecheck{\eta}_{>R_{k+1}^{-1/3}}*\widecheck{\eta}_s(x)|^2dx\Big)^{2}\\
&\sim |V|^{-1}\Big(\int_{V}\sum_{\tau_{2s}\subset\tau}|\sum_{\tau_s\subset\tau_{2s}}\sum_{\tau_s'\sim\tau_s}({f_{\tau_s}^{k+1}}{\overline{f_{\tau_s'}^{k+1}}})*{\w}_{\tau_k}*\widecheck{\eta}_{>R_{k+1}^{-1/3}}*\widecheck{\eta}_s(x)|^2dx\Big)^{2} 
\end{align*}
where by Proposition \ref{geo4}, we identify $\tau\in{\bf{S}}(\sigma^{-1}s)$ with $\tau'\in{\bf{S}}_{\sigma^{-1} s}$ via $s^{-1}T[(10\tau-10\tau)\setminus B_{\sigma^{-1}s/4}(0)]\sim\tau'$ and identify $s^{-1} T^tU$ with a $V\|V_{\tau,s^{-3}}$. Next, by Cauchy-Schwartz, since the number of $\tau_s\subset\tau_{2s}$ and the number of $\tau_s'$ satisfying $\tau_s'\sim\tau_s$ is $O(1)$, it suffices to replace the above integrals by 
\[  |V|^{-1}\Big(\int_{V}\sum_{\tau_{s}\subset\tau}||{f_{\tau_s}^{k+1}}|^2*\rho(x)|^2dx\Big)^{2}\]
where $\rho=|\w_{\tau_k}*\widecheck{\eta}_{>R_{k+1}^{-1/3}}*\widecheck{\eta}_s|$. Note that $\|\rho\|_1\sim 1$. Again, by Cauchy-Schwarz, the above integral is bounded by 
\[  |V|^{-1}\Big(\int_{V}\sum_{\tau_{s}\subset\tau}|{f_{\tau_s}^{k+1}}|^4*\rho(x)dx\Big)^{2}. \]
Write $\chi_{V_{0}}$ for the characteristic function of $2V_{\tau,s^{-3}}$. After summing the above integral over $V\|V_{\tau,s^{-3}}$, we have the bound
\[ \sum_{V\|V_{\tau,s^{-3}}}|V|^{-1}\Big(\int_{V}\sum_{\tau_{s}\subset\tau}|{f_{\tau_s}^{k+1}}|^4*\rho(x)dx\Big)^{2}\lesssim |V|^{-2}\int \Big(\sum_{\tau_{s}\subset\tau}|{f_{\tau_s}^{k+1}}|^4*\rho*\chi_{V_0}\Big)^{2}. \]
By Cauchy-Schwarz and Young's convolution inequality, the right hand side is bounded by 
\[|V|^{-2}\int \Big(\sum_{\tau_{s}\subset\tau}|{f_{\tau_s}^{k+1}}|^4*\chi_{V_0}\Big)^{2}\lesssim  \sum_{V\|V_{\tau,s^{-3}}}|V|^{-1}\Big(\int_V\sum_{\tau_{s}\subset\tau}|{f_{\tau_s}^{k+1}}|^4\Big)^{2}. \]

\end{proof}

\section{Key iterations that unwind the pruning process\label{keyalgo}} 
\subsection{An $L^{7/2}$ square function estimate for the parabola. }
Let $\mb{P}^1=\{(t,t^2):0\le t\le 1\}$ and for $r\ge 1$, let $\mc{N}_{r^{-1}}(\P^1)$ denote the $r^{-1}$-neighborhood of $\P^1$ in $\R^2$. Define the collection of canonical $\sim r^{-1/2}\times r^{-1}$ parabola blocks as follows. Let $s\in 2^\Z$ be the smallest number satisfying $r^{-1}\le s^2$. Then write $\mc{N}_{s^2}(\P^1)$ as 
\begin{equation*} 
\bigsqcup_{1\le l\le s^{-1}-2} \{(\xi_1,\xi_2)\in\mc{N}_{s^2}(\P^1):ls\le \xi_1<(l+1)s \}  \end{equation*}
and the two end pieces
\[  \{(\xi_1,\xi_2)\in\mc{N}_{s^2}(\P^1):\xi_1<s\} \sqcup  \{(\xi_1,\xi_2)\in\mc{N}_{s^2}(\P^1):1-s\le \xi_1\} \]
We use the notation $\ell(\tau)=r^{-1/2}$ in two ways: (1) to describe $\tau$ as one of the blocks from the above partition and (2) to index the set of $\tau$ from the above partition.

\begin{thm}[Cylindrical $L^{7/2}$ square function estimate over $\P^1$] \label{cylPS}
Let $\mb{P}^1=\{(t,t^2):0\le t\le 1\}$ and for $r\ge1$, let $\mc{N}_{r^{-1}}(\P^1)$ denote the $r^{-1}$-neighborhood of $\mc{P}^1$ in $\R^2$. If $h:\R^3\to\C$ is a Schwartz function with Fourier transform supported in $\mc{N}_{r^{-1}}(\P^1)\times\R$, then
\[ \int_{\R^3}|h|^{7/2}\lesssim_\e r^{\e}\int_{\R^3}(\sum_{\z}|h_\z|^2)^{7/4} \]
where the $\z$ are products of approximate rectangles $\theta$, $\ell(\theta)={r^{-1/2}}$, with $\R$.
\end{thm}

The local version of Theorem \ref{cylPS} is 
\begin{corollary}\label{Pcor} Let $B_r$ be an $r$-ball in $\R^3$. If $h:\R^3\to\C$ is a Schwartz function with Fourier transform supported in $\mc{N}_{r^{-1}}(\P^1)\times\R$, then
\[ \int_{\R^3}|h|^{7/2}W_{B_r}\lesssim_\e r^{\e}\int_{\R^3}(\sum_{\z}|h_\z|^2)^{7/4}W_{B_r} \]
where the $\z$ are products of approximate rectangles $\theta$, $\ell(\theta)={r^{-1/2}}$, with $\R$.
\end{corollary}

We delay the proofs of Theorem \ref{cylPS} and Corollary \ref{Pcor} to \textsection \ref{cylPSsec} in Appendix A. 

\begin{lemma}\label{algo1} Let $R_{k-1}\le r\le R_{k}$. For each $r^{-\frac{1}{3}}\le \sigma\le 1$, $\tau\in{\bf{S}}({\sigma^{-1}r^{-\frac{1}{3}}})$, and $V\|V_{\tau,r}$, we have 
\begin{align*}
\fint_V\sum_{\substack{\tau'\subset\tau\\\tau'\in{\bf{S}}(r^{-\frac{1}{3}})}}|f_{\tau'}^{k}|^{7/2}\le E_{\e^7}R^{\e^7}|V|^{-1}\int\sum_{\substack{\tau'\subset\tau\\\tau'\in{\bf{S}}(r^{-\frac{1}{3}})}}|\sum_{\substack{\tau''\subset\tau'\\\tau''\in{\bf{S}}(\sigma^{\frac{1}{2}}r^{-\frac{1}{3}})}}|f_{\tau''}^{m}|^2|^{7/4} W_V
\end{align*}
where $m\ge k$ satisfies $R_{m}^{-\frac{1}{3}}\le \sigma^{\frac{1}{2}}r^{-\frac{1}{3}}\le R_{m-1}^{-\frac{1}{3}}$.
\end{lemma}
The weight function $W_V$ is defined by $w(T^{-1}x)$, where $T$ is an affine transformation mapping the unit cube to $V$ and $w$ is from Definition \ref{M3ballweight}. 
\begin{proof}[Proof of Lemma \ref{algo1}]
This would be a straightforward consequence of Corollary \ref{Pcor} after a moment curve rescaling, except for the fact that the pruning process alters the Fourier support of $f$. Let $s=\min(r^{-\frac{1}{3}},R_{m-1}^{-\frac{1}{3}})$. First we will show that 
\begin{equation}\label{first}
\fint_V\sum_{\substack{\tau'\subset\tau\\\tau'\in{\bf{S}}(r^{-\frac{1}{3}})}}|f_{\tau'}^{k}|^{7/2}\lesssim_\e R^{\e^8}|V|^{-1}\int\sum_{\substack{\tau'\subset\tau\\\tau'\in{\bf{S}}(r^{-\frac{1}{3}})}}|\sum_{\substack{\tau''\subset\tau'\\\tau''\in{\bf{S}}(s)}}|f_{\tau''}^{m}|^2|^{7/4} W_V, 
\end{equation}
for $R_m^{-\frac{1}{3}}\le \sigma^{\frac{1}{2}}r^{-\frac{1}{3}}\le R_{m-1}^{-\frac{1}{3}}$. If $r^{-\frac{1}{3}}\le R_{m-1}^{-\frac{1}{3}}$, then $|f_{\tau'}^k|\le |f_{\tau'}^m|$ and \eqref{first} is trivially true, so assume that $R_{m-1}^{-\frac{1}{3}}<r^{-\frac{1}{3}}$. This assumption also implies that $R_{m-1}^{-\frac{1}{3}}\le R_k^{-\frac{1}{3}}$, so $k<m$. We will perform an ``unwinding the pruning" process using successive applications of Corollary \ref{Pcor}. 

Begin by performing a moment curve rescaling. Suppose that $\tau\in{\bf{S}}(\sigma^{-1}r^{-\frac{1}{3}})$ is the $l$th piece, meaning  
\[ \{(\xi_1,\xi_2,\xi_3): l(\sigma^{-1}r^{-\frac{1}{3}})\le \xi_1<(l+1)(\sigma^{-1}r^{-\frac{1}{3}}),\,|\xi_2-\xi_1^2|\le (\sigma^{-1}r^{-\frac{1}{3}})^{2},\,|\xi_3-3\xi_1\xi_2+2\xi_1^3|\le (\sigma^{-1}r^{-\frac{1}{3}})^{3} \}  \]
where $0\le l\le \sigma r^{\frac{1}{3}}-1$. 
Define the affine transformation $T:\R^3\to\R^3$ by $T(\xi_1,\xi_2,\xi_3)=(\xi_1',\xi_2',\xi_3')$ where
\[ \begin{cases}\xi_1'&=\frac{\xi_1-l(\sigma^{-1}r^{-\frac{1}{3}})}{(\sigma^{-1}r^{-\frac{1}{3}})} \\
\xi_2'&= \frac{\xi_2-2(l\sigma^{-1}r^{-\frac{1}{3}})\xi_1+(l\sigma^{-1}r^{-\frac{1}{3}})^2}{(\sigma^{-1}r^{-\frac{1}{3}})^2}\\
\xi_3'&=\frac{\xi_3-3(l\sigma^{-1}r^{-\frac{1}{3}})\xi_2+3(l\sigma^{-1}r^{-\frac{1}{3}})^2\xi_1-(l\sigma^{-1}r^{-\frac{1}{3}})^3}{(\sigma^{-1}r^{-\frac{1}{3}})^3}.
\end{cases} \]
It is not difficult to verify that for each $k\le l\le m-1$ and $\tau_l\in{\bf{S}}(R_l^{-\frac{1}{3}})$, $T(\tau_l)$ is a canonical moment curve block in ${\bf{S}}(\sigma r^{\frac{1}{3}}R_l^{-\frac{1}{3}})$. For each $\tau''\in{\bf{S}}(\sigma^{\frac{1}{2}}r^{-\frac{1}{3}})$, $T(\tau'')$ is a canonical moment curve block in ${\bf{S}}(\sigma^{\frac{3}{2}})$. 

Let $T(\xi)=A\xi+b$ where $A:\R^3\to\R^3$ is a linear transformation and $\xi,b\in\R^3$. For each $\tau''\in{\bf{S}}(\sigma^{\frac{1}{2}}r^{-\frac{1}{3}}$, $f_{\tau''}^{k}$. Perform the change of variables $x\mapsto A^Tx$, obtaining 
\[\fint_V\sum_{\substack{\tau'\subset\tau\\\tau'\in{\bf{S}}(r^{-\frac{1}{3}})}}|f_{\tau'}^{k}|^{7/2}=|(A^T)^{-1}(V)|^{-1}\int\sum_{\substack{\tau'\subset\tau\\\tau'\in{\bf{S}}(r^{-\frac{1}{3}})}}|f_{\tau'}^{k}\circ A^T|^{7/2}W_V\circ A^T. \]
Write $f_{\tau'}^k\circ A^T=g_{T(\tau')}^k$ where $g_{T(\tau')}^k$ is the Fourier projection of a Schwartz function onto $2T(\tau')$ (the dilation by $2$ reflects the fact that $f_{\tau'}^k$ is supported in $2\tau'$). In fact, since $f_{\tau'}^k$ is Fourier supported in $\cup_{\tau_k\subset\tau'}2\tau_k$, $g_{T(\tau')}^k$ is Fourier supported in $\cup_{\tau_k\subset\tau'}2T(\tau_k)$.

Next, consider what the set $(A^T)^{-1}(V)$ is. The matrix representations of $A$ and $(A^T)^{-1}$ are 
\[ A=\begin{bmatrix} \frac{1}{\sigma^{-1}r^{-\frac{1}{3}}} & 0 & 0\\
-\frac{2l}{\sigma^{-1}r^{-\frac{1}{3}}} & \frac{1}{\sigma^{-2}r^{-\frac{2}{3}}} & 0 \\
\frac{3l^2}{\sigma^{-1}r^{-\frac{1}{3}}} & -\frac{3l}{\sigma^{-2}r^{-\frac{2}{3}}} & \frac{1}{\sigma^{-3}r^{-1}} 
\end{bmatrix} \qquad\text{and}\qquad (A^T)^{-1}=\begin{bmatrix} \sigma^{-1}r^{-\frac{1}{3}} & 2l(\sigma^{-1}r^{-\frac{1}{3}})^2 & 3l^2(\sigma^{-1}r^{-\frac{1}{3}})^3 \\
0 & (\sigma^{-1}r^{-\frac{1}{3}})^2 & 3l(\sigma^{-1}r^{-\frac{1}{3}})^3 \\
0 & 0 & (\sigma^{-1}r^{-\frac{1}{3}})^3
\end{bmatrix}. \]
The set $V$ is a translation of $V_{\tau,r}$ which is comparable to the set
\[ \{A{\bf{T}}(l\sigma^{-1}r^{-\frac{1}{3}})+B{\bf{N}}(l\sigma^{-1}r^{-\frac{1}{3}})+C{\bf{B}}(l\sigma^{-1}r^{-\frac{1}{3}}):|A|\le \sigma^{-2}r^{\frac{1}{3}} ,\quad|B|\le \sigma^{-1}r^{\frac{2}{3}},\quad|C|\le r\} \]
where the vectors ${\bf{T}},{\bf{N}},{\bf{B}}$ are defined in \eqref{Frenet}. Using these explicit expressions, we see that $(A^T)^{-1}(V)$ is comparable to a $\sigma^{-3}$-ball in $\R^3$. 

For our first step in showing \eqref{first}, we would like to apply Corollary \ref{Pcor} to estimate $g_{T(\tau')}^k$ by a square function in $T(\tau_k)\subset T(\tau')$. Since $T(\tau_k)$ are moment curve blocks in ${\bf{S}}(\sigma r^{\frac{1}{3}}R_k^{-\frac{1}{3}})$, Corollary \ref{Pcor} may be employed on balls of radius $\sigma^{-2}r^{-\frac{2}{3}}R_k^{\frac{2}{3}}$. The assumptions that $\sigma^{\frac{1}{2}}r^{-\frac{1}{3}}\le R_{m-1}^{-\frac{1}{3}}\le R_k^{-\frac{1}{3}}$ mean that $r^{-\frac{2}{3}}R_k^{\frac{2}{3}}\le \sigma^{-1}$. Since $(A^T)^{-1}(V)$ is comparable to a $\sigma^{-3}$-ball which contains $\sigma^{-2}r^{-\frac{2}{3}}R_k^{\frac{2}{3}}$-balls, Corollary \ref{Pcor} applies, yielding
\[ \int\sum_{\substack{\tau'\subset\tau\\\tau'\in{\bf{S}}(r^{-\frac{1}{3}})}}|f_{\tau'}^{k}\circ A^T|^{7/2} W_{(A^T)^{-1}(V)}\le C_\e R^{\e^9} \int\sum_{\substack{\tau'\subset\tau\\\tau'\in{\bf{S}}(r^{-\frac{1}{3}})}}|\sum_{\tau_k\subset\tau'}|f_{\tau_k}^{k}\circ A^T|^2|^{7/4}W_{(A^T)^{-1}(V)}. \]
Then by \eqref{item1} of Lemma \ref{pruneprop}, $|f_{\tau_k}^{k}|\le|f_{\tau_k}^{k+1}|$, so 
\[ \int\sum_{\substack{\tau'\subset\tau\\\tau'\in{\bf{S}}(r^{-\frac{1}{3}})}}|f_{\tau'}^{k}\circ A^T|^{7/2}W_{(A^T)^{-1}(V)}\le C_{\e}R^{\e^9} \int\sum_{\substack{\tau'\subset\tau\\\tau'\in{\bf{S}}(r^{-\frac{1}{3}})}}|\sum_{\tau_k\subset\tau'}|f_{\tau_k}^{k+1}\circ A^T|^2|^{7/4}W_{(A^T)^{-1}(V)}. \]
If $k+1=m$, then we are done showing \eqref{first}. If $k+1<m$, then by Khintchine's inequality, we may select signs $c_{\tau_k}\in\{\pm 1\}$ so that
\[ \int\sum_{\substack{\tau'\subset\tau\\\tau'\in{\bf{S}}(r^{-\frac{1}{3}})}}|\sum_{\tau_k\subset\tau'}|f_{\tau_k}^{k+1}\circ A^T|^2|^{7/4}W_{(A^T)^{-1}(V)}\sim \int\sum_{\substack{\tau'\subset\tau\\\tau'\in{\bf{S}}(r^{-\frac{1}{3}})}}|\sum_{\tau_k\subset\tau'}c_{\tau_k}f_{\tau_k}^{k+1}\circ A^T|^{7/2}W_{(A^T)^{-1}(V)}.\]
The function $\sum_{\tau_k\subset\tau'}c_{\tau_k}f_{\tau_k}^{k+1}\circ A^T=g_{T(\tau')}^{k+1}$ where $g_{T(\tau')}^{k+1}$ has Fourier support in $\cup_{\tau_{k+1}\subset\tau'}2T(\tau_{k+1})$. Then we may apply Corollary \ref{Pcor} again to obtain 
\[ \int\sum_{\substack{\tau'\subset\tau\\\tau'\in{\bf{S}}(r^{-\frac{1}{3}})}}|\sum_{\tau_k\subset\tau'}c_{\tau_k}f_{\tau_k}^{k+1}\circ A^T|^{7/2}W_{(A^T)^{-1}(V)}\le C_\e R^{\e^9}\int\sum_{\substack{\tau'\subset\tau\\\tau'\in{\bf{S}}(r^{-\frac{1}{3}})}}|\sum_{\tau_{k+1}\subset\tau'}|f_{\tau_{k+1}}^{k+1}\circ A^T|^2|^{7/4}W_{(A^T)^{-1}(V)}. \]
Again, use $|f_{\tau_{k+1}}^{k+1}|\le|f_{\tau_{k+2}}^{k+2}|$ and halt if $m=k+2$ or find signs $c_{\tau_{k+1}}\in\{\pm1\}$ for which 
\[ \int\sum_{\substack{\tau'\subset\tau\\\tau'\in{\bf{S}}(r^{-\frac{1}{3}})}}|\sum_{\tau_{k+1}\subset\tau'}|f_{\tau_{k+1}}^{k+2}\circ A^T|^2|^{7/4}W_{(A^T)^{-1}(V)}\sim \int\sum_{\substack{\tau'\subset\tau\\\tau'\in{\bf{S}}(r^{-\frac{1}{3}})}}|\sum_{\tau_{k+1}\subset\tau'}c_{\tau_{k+1}}f_{\tau_{k+1}}^{k+1}\circ A^T|^{7/2}W_{(A^T)^{-1}(V)}. \]
Iterating this process and undoing the change of variables results in \eqref{first}. Note that since $m-k\le \e^{-1}$, the accumulated constant satisfies $(C_\e R^{\e^9})^{\e^{-1}}\le C_\e^{\e^{-1}}R^{\e^8}$. For the final step, we perform the same argument as above, using a change of variables and Khintchine's inequality to write
\[ |V|^{-1}\int\sum_{\substack{\tau'\subset\tau\\\tau'\in{\bf{S}}(r^{-\frac{1}{3}})}}|\sum_{\substack{\tau''\subset\tau'\\\tau''\in{\bf{S}}(s)}}|f_{\tau''}^{m}|^2|^{7/4} W_V\sim |{(A^T)^{-1}(V)}|^{-1}\int\sum_{\substack{\tau'\subset\tau\\\tau'\in{\bf{S}}(r^{-\frac{1}{3}})}}|\sum_{\substack{\tau''\subset\tau'\\\tau''\in{\bf{S}}(s)}}c_{\tau''}f_{\tau''}^{m}\circ A^T|^{7/2}W_{(A^T)^{-1}(V)}. 
\]
Since for $\tau''\in{\bf{S}}(\sigma^{\frac{1}{2}}r^{-\frac{1}{3}}$ with $\tau''\subset\tau$ we have $T(\tau'')\in{\bf{S}}(\sigma^{\frac{3}{2}})$, we may apply Corollary \ref{Pcor} one last time to get
\[\int\sum_{\substack{\tau'\subset\tau\\\tau'\in{\bf{S}}(r^{-\frac{1}{3}})}}|\sum_{\substack{\tau''\subset\tau'\\\tau''\in{\bf{S}}(s)}}c_{\tau''}f_{\tau''}^{m}\circ A^T|^{7/2}W_{(A^T)^{-1}(V)}\le C_\e R^{\e^8} \int \sum_{\substack{\tau'\subset\tau\\\tau'\in{\bf{S}}(r^{-\frac{1}{3}})}}|\sum_{\substack{\tau''\subset\tau'\\\tau''\in{\bf{S}}(\sigma^{\frac{1}{2}}r^{-\frac{1}{3}})}}|f_{\tau''}^{m}\circ A^T|^2|^{7/4}W_{(A^T)^{-1}(V)}. \]
Finally, undo the change of variables to get 
\[ |{(A^T)^{-1}(V)}|^{-1}\int\sum_{\substack{\tau'\subset\tau\\\tau'\in{\bf{S}}(r^{-\frac{1}{3}})}}|\sum_{\substack{\tau''\subset\tau'\\\tau''\in{\bf{S}}(\sigma^{\frac{1}{2}}r^{-\frac{1}{3}})}}|f_{\tau''}^{m}\circ A^T|^2|^{7/4}W_{(A^T)^{-1}(V)}=|V|^{-1}\int\sum_{\substack{\tau'\subset\tau\\\tau'\in{\bf{S}}(r^{-\frac{1}{3}})}}|\sum_{\substack{\tau''\subset\tau'\\\tau''\in{\bf{S}}(\sigma^{\frac{1}{2}}r^{-\frac{1}{3}})}}|f_{\tau''}^{m}
|^2|^{7/4}W_V.\]

\end{proof}

\subsection{An $\ell^{7/4}$-estimate for the cone. }

Let $\Gamma=\{(\xi_1,\xi_2,\xi_3)\in\R^3:\xi_1^2+\xi_2^2=\xi_3^2,\quad\frac{1}{2}\le\xi_3\le 1\}$ be the truncated cone. In this section, let ${\bf{S}}_{r^{-1/2}}$ denote the collection of $1\times r^{-1/2}\times r^{-1}$ blocks which tile $\mc{N}_{r^{-1}}(\Gamma)$, as defined in \textsection5 of \cite{ampdep}.

\begin{proposition}\label{cone3.5} For any Schwartz function $h:\R^3\to\C$ with $\widehat{h}$ supported in $\mc{N}_{r^{-1}}(\Gamma)$, we have 
\[ \int_{\R^3}|h|^{7/2}\lesssim_\e r^\e\sum_{r^{-1/2}\le \sigma\le 1}\sum_{\tau\in{\bf{S}}_{\sigma^{-1}r^{-1/2}}}\sum_{U\|U_{\tau_k,r}}|U|\Big(|U|^{-1}\int\sum_{\substack{\theta\subset\tau\\\theta\in{\bf{S}}_{r^{-1/2}}}}|h_\theta|^{7/4}W_U\Big)^{2} .\]
\end{proposition}

We delay the proof of Proposition \ref{cone3.5} to \textsection\ref{coneappsec}. 

\begin{lemma}\label{algo2} For $10\le r_1\le r_2$ and $R_{k-1}\le r_2\le R_{k}$, there exists some dyadic $s$, $R^{-\frac{1}{3}}\le s\le r_2^{-\frac{1}{3}}$ such that \begin{align}
\sum_{\tau_0\in{\bf{S}}(r_1^{-\frac{1}{3}})}\int_{\R^3}(&\sum_{\substack{\tau'\subset\tau_0\\\tau'\in{\bf{S}}(r_2^{-\frac{1}{3}})}}|f_{\tau'}^{k}|^2)^{7/2}\le (C\log R)^{2\e^{-1}}\int_{\R^3}(\sum_{\theta\in{\bf{S}}(R^{-1/3})}|f_\theta|^2)^{7/2} \label{LHS}\\
&+ (C\log R)^{\e^{-1}}B_{\d}R^{\d} \sum_{\tau_0\in{\bf{S}}(r_1^{-\frac{1}{3}})}\sum_{sr_1^{\frac{1}{3}}\le \sigma\le 1}\sum_{\substack{\tau\in{\bf{S}}(\sigma^{-1}s)\\\tau\subset\tau_0}} \sum_{V\|V_{\tau,s^{-3}}}|V|\Big(\fint_V\sum_{\substack{\tau'\subset\tau \\\tau'\in{\bf{S}}(s)}}|f_{\tau'}^{k_m}|^{7/2} \Big)^{2}\nonumber
\end{align}
where $k_m\ge k$ satisfies $R_{k_m}^{-\frac{1}{3}}\le s\le R_{k_m-1}^{-\frac{1}{3}}$.
\end{lemma}
\begin{proof}[Proof of Lemma \ref{algo2}]
Fix a $\tau_0$ and analyze each integral on the left hand side of \eqref{LHS}. Suppose that there is some dyadic value $s$ in the range $R_k^{-1/3}\le s\le r_2^{-1/3}$ which satisfies
\begin{equation}\label{sup1} \int_{\R^3}(\sum_{\substack{\tau'\subset\tau_0\\\tau'\in{\bf{S}}(r_2^{-\frac{1}{3}})}}|f_{\tau'}^{k}|^2)^{7/2}\lesssim(\log R)\int_{\R^3}|\sum_{\substack{\tau'\subset\tau_0\\\tau'\in{\bf{S}}(r_2^{-\frac{1}{3}})}}|f_{\tau'}^{k}|^2*\widecheck{\eta}_{s}|^{7/2}.\end{equation}
The next step repeats the argument from Lemma \ref{high1}, which we provide a sketch for now. By the proof of Lemma \ref{geo4}, the integral on the right hand side above is equivalent to 
\[\int_{\R^3}(\sum_{\substack{\tau_s\subset\tau_0\\\tau_s\in{\bf{S}}(s)}}\sum_{\substack{\tau_s'\sim\tau_s\\\tau_s'\in{\bf{S}}(s)}}f_{\tau_s}^{k}\overline{f_{\tau_s'}^k}*\widecheck{\eta}_{s})^{7/2} \]
where $\tau_s'\sim\tau_s$ means that $\text{dist}(\tau_s',\tau_s)\lesssim s$. By Propositions \ref{geo3} and \ref{geo4}, the Fourier support of the integrand may now be viewed as a canonical tiling of (an $s$-dilation of) the cone. Therefore, we may apply Proposition \ref{cone3.5} to bound the previous displayed integral by 
\[  C_\d R^\d\sum_{s\le\sigma\le 1}\sum_{\substack{\tau\in{\bf{S}}(\sigma^{-1}s)\\\tau\subset\tau_0}}\sum_{V\|V_{\tau,s^{-3}}}|V|^{-1}\Big(\int\sum_{\tau_s\subset\tau}|f_{\tau_s}^k|^{7/2}W_V\Big)^2 \]
where we are free to choose $\d>0$. Here, the iteration ends since we have proven the lemma. 

In the case that \eqref{sup1} does not hold, we may assume that 
\begin{equation}\label{sup2} \int_{\R^3}(\sum_{\substack{\tau'\subset\tau_0\\\tau'\in{\bf{S}}(r_2^{-\frac{1}{3}})}}|f_{\tau'}^{k}|^2)^{7/2}\lesssim(\log R)\int_{\R^3}|\sum_{\substack{\tau'\subset\tau_0\\\tau'\in{\bf{S}}(r_2^{-\frac{1}{3}})}}|f_{\tau'}^{k}|^2*\widecheck{\eta}_{<R_k^{-1/3}}|^{7/2}.\end{equation}
Applying the argument from Lemma \ref{low} to the integrand, we have 
\[ \int_{\R^3}|\sum_{\substack{\tau'\subset\tau_0\\\tau'\in{\bf{S}}(r_2^{-\frac{1}{3}})}}|f_{\tau'}^{k}|^2*\widecheck{\eta}_{<R_k^{-1/3}}|^{7/2}\lesssim \int_{\R^3}|\sum_{\substack{\tau_k\subset\tau_0}}|f_{\tau_k}^{k}|^2*|\widecheck{\eta}_{<R_k^{-1/3}}||^{7/2}.\]
Since $\widecheck{\eta}_{<R_k^{-1/3}}$ is $L^1$-normalized, we may ignore the $*|\widecheck{\eta}_{<R_k^{-1/3}}|$ by Young's convolution inequality. Then by Lemma \ref{pruneprop}, we have the pointwise inequality $|f_{\tau_k}^k|\le |f_{\tau_k}^{k+1}|$. The conclusion in this case is that
\[\int_{\R^3}(\sum_{\substack{\tau'\subset\tau_0\\\tau'\in{\bf{S}}(r_2^{-\frac{1}{3}})}}|f_{\tau'}^{k}|^2)^{7/2}\lesssim(\log R) \int_{\R^3}|\sum_{\substack{\tau_k\subset\tau_0}}|f_{\tau_k}^{k}|^2|^{7/2}. \]
Next, we iterate this procedure, considering the two cases \eqref{sup1} and \eqref{sup2} applied to the integral on the right hand side above. The number of steps in the iteration is the same as the number of times (the appropriate version of) \eqref{sup2} holds. Each time \eqref{sup2} holds, we refine the scale of our square function by a factor of $R^{\e/3}$ (moving from $\tau_k$ to $\tau_{k+1}$, say). Therefore, the total number of steps is bounded by $\e^{-1}$. In the case that \eqref{sup1} holds at one step, the iteration terminates with an accumulated constant of at most $(\log R)^{\e^{-1}}B_\d R^{\d\e^{-1}}$. Since we are free to choose $\d$, this proves the proposition. The final case is that \eqref{sup2} holds at all steps in the iteration. This leads to the first term from the upper bound in \eqref{LHS}.

\end{proof}

\subsection{Algorithm to fully unwind the pruning process \label{hisec}}

Recall some notation which was defined in \textsection\ref{geo}. 
\begin{enumerate}
    \item For each dyadic $\sigma$, $R_k^{-\frac{1}{3}}\sigma\le 1$ and each $\tau\in{\bf{S}}(\sigma^{-1}R_k^{-\frac{1}{3}})$, the dual set $\tau^*$ is a $\sigma R_k^{\frac{1}{3}}\times \sigma^2R_k^{\frac{2}{3}}\times \sigma^3R_k$ plank centered at the origin and comparable to $\{x\in\R^3:|x\cdot\xi|\le 1\qquad\forall\xi\in\tau-\tau\}$. The wave envelope $V_{\tau,R_k}$ is an anisotropically dilated version of $\tau^*$ with dimensions $\sigma^{-2}R_k^{-\frac{1}{3}}\times \sigma^{-1}R_k^{-\frac{2}{3}}\times R_k^{-1}$. 
    \item Let $V\|V_{\tau,R_k}$ denote an indexing set for $V$ which are translates of $V_{\tau,R_k}$ and which tile $\R^3$. 
\end{enumerate} 

\begin{proposition}\label{algo} For each $k$, we have 
\begin{align}
\label{algline1}\sum_{R_k^{-\frac{1}{3}}\le \sigma\le 1}\sum_{\tau\in{\bf{S}}(\sigma^{-1}R_k^{-\frac{1}{3}})}\sum_{V\|V_{\tau,R_k}}|V|^{-1}&\left(\int\sum_{\tau_k\subset\tau}|f_{\tau_k}^{k+1}|^{7/2}W_V\right)^2\le C_\e R^{3\e^2}\S_1(R^\e)^{N-k+1}\int(\sum_\theta|f_\theta|^2)^{7/2}.
\end{align}
\end{proposition}

Propositon \ref{algo} will follow from an algorithm which uses the Lemmas \ref{algo1} and \ref{algo2} as building blocks. 

\begin{proof}[Proof of Proposition \ref{algo}] Let $C>0$ be an absolute constant that we specify later in the proof. We will define an algorithm which at intermediate step $m$, produces an inequality 
\begin{align}\label{stepm} 
(\text{L.H.S. of \eqref{algline1}})\le (C\log R)^{4\e^{-1}m}&(B_{\e^7}R^{\e^7})^{m}(E_{\e^7}R^{\e^7})^{a}(R^{\e^3}\S_1(R^\e))^b  \\
&\times \sum_{s_m\le \sigma\le 1}\sum_{\substack{\tau\in{\bf{S}}(\sigma^{-1}s_m)}} \sum_{V\|V_{\tau,s_m^{-3}}}|V|\Big(\fint_V\sum_{\substack{\tau'\subset\tau \\\tau'\in{\bf{S}}(s_m)}}|f_{\tau'}^{k_m}|^{7/2} \Big)^{2}\nonumber
\end{align}
in which $a+b=m$ and $R^{-\frac{1}{3}}\le s_m\le \min(R_k^{-\frac{1}{3}}R^{-\frac{\e^3}{2}a},R_k^{-\frac{1}{3}}R^{-\frac{\e}{3}b})$ and $k_m\ge k+1$ satisfies $R_{k_m}^{-\frac{1}{3}}\le s_m\le R_{k_m-1}^{-\frac{1}{3}}$. Notice that \eqref{stepm} clearly holds with $m=0$, taking $k_m=k+1$ and $s_m=R_k^{-\frac{1}{3}}$. Assuming \eqref{stepm} holds for $m-1$, we will show that either the algorithm terminates and the proposition is proved or \eqref{stepm} holds for $m\ge 1$. We further suppose that $R^{-\frac{1}{3}}R^{\frac{\e}{3}}\le s_{m-1}$, otherwise proceed to the final case in which the algorithm terminates below.

\noindent\fbox{Step m:} Let $\sigma\in [s_{m-1},1]$ be a dyadic value satisfying
\begin{align}\label{stephyp} 
(\text{L.H.S. of \eqref{algline1}})\le (C\log &R)^{4\e^{-1}(m-1)}(B_{\e^7}R^{\e^7})^{m-1}(E_{\e^7}R^{\e^7})^{a}(R^{\e^3}\S_1(R^\e))^b  \\
&\times (C\log R)\sum_{\substack{\tau\in{\bf{S}}(\sigma^{-1}s_{m-1})}} \sum_{V\|V_{\tau,s_{m-1}^{-3}}}|V|\Big(\fint_V\sum_{\substack{\tau'\subset\tau \\\tau'\in{\bf{S}}(s_{m-1})}}|f_{\tau'}^{k_{m-1}}|^{7/2} \Big)^{2}\nonumber
\end{align}
in which $a+b=m-1$ and $R^{-\frac{1}{3}}\le s_{m-1}\le \min(R_k^{-\frac{1}{3}}R^{-\frac{\e^3}{2}a},R_k^{-\frac{1}{3}}R^{-\frac{\e}{3}b})$ and $k_{m-1}\ge k+1$ satisfies $R_{k_{m-1}}^{-\frac{1}{3}}\le s_{m-1}\le R_{k_{m-1}-1}^{-\frac{1}{3}}$. 
The analysis splits into two cases depending on whether $\sigma\ge R^{-\e^3}$ or $\sigma<R^{-\e^3}$.
\newline\noindent\fbox{Step m: $\sigma\ge R^{-\e^3}$.} Using Cauchy-Schwarz and then H\"{o}lder's inequality, for each $\tau\in{\bf{S}}(\sigma^{-1}s_{m-1})$,
\begin{align*} 
\sum_{V\|V_{\tau,s_{m-1}^{-3}}}|V|\Big(\fint_V\sum_{\substack{\tau'\subset\tau \\\tau'\in{\bf{S}}(s_{m-1})}}|f_{\tau'}^{k_{m-1}}|^{7/2} \Big)^{2}&\le\sum_{V\|V_{\tau,s_{m-1}^{-3}}}|V|(\#\tau'\subset\tau)\sum_{\substack{\tau'\subset\tau \\\tau'\in{\bf{S}}(s_{m-1})}}\Big(\fint_V|f_{\tau'}^{k_{m-1}}|^{7/2} \Big)^{2}\\
    &\lesssim\sum_{V\|V_{\tau,s_{m-1}^{-3}}}|V|(R^{\e^3})\sum_{\substack{\tau'\subset\tau \\\tau'\in{\bf{S}}(s_{m-1})}}\fint_V|f_{\tau'}^{k_{m-1}}|^{7} \\
    &\lesssim R^{\e^3}\sum_{\substack{\tau'\subset\tau \\\tau'\in{\bf{S}}(s_{m-1})}}\int_{\R^3}|f_{\tau'}^{k_{m-1}}|^7.  \end{align*}
Define $k_{m-1}^*$ to be $k_{m-1}^*=k_{m-1}+1$ if $s_{m-1}=R_{k_{m-1}}^{-\frac{1}{3}}$ and $k_{m-1}^*=k_{m-1}$ if $R_{k_{m-1}}^{-\frac{1}{3}}<s_{m-1}$. Note that by \eqref{item1} of Lemma \ref{pruneprop}, for each $\tau'\in{\bf{S}}(s_{m-1})$, $|f_{\tau'}^{k_{m-1}}|\le |f_{\tau'}^{k_{m-1}^*}|$. By rescaling for the moment curve, for each $\tau'\in{\bf{S}}(s_{m-1})$,
\[ \int_{\R^3}|f_{\tau'}^{k_{m-1}^*}|^7\le \text{S}(R^\e)\int_{\R^3}(\sum_{\tau_{k_{m-1}^*}\subset\tau'}|f_{\tau_{k_{m-1}^*}}^{k_{m-1}^*}|^2)^{7/2}. \]
If $k_{m-1}^*=N$, then the algorithm terminates with the inequality 
\begin{align}\label{term1'} 
(\text{L.H.S. of \eqref{algline1}})\le (C\log &R)^{4\e^{-1}(m-1)+1}(B_{\e^7}R^{\e^7})^{m-1}(E_{\e^7}R^{\e^7})^{a}(R^{\e^3}\S_1(R^\e))^{b+1}  \int_{\R^3}(\sum_\theta|f_\theta|^2)^{7/2} 
\end{align}
in which $a+b=m-1$ and $R^{-\frac{1}{3}}\le \min(R_k^{-\frac{1}{3}}R^{-\frac{\e^3}{2}a},R_k^{-\frac{1}{3}}R^{-\frac{\e}{3}b})$. By Lemma \ref{algo2}, there are two further cases. One case is that 
\[ \sum_{\tau'\in{\bf{S}}(s_{m-1})}\int_{\R^3}(\sum_{\tau_{k_{m-1}^*}\subset\tau'}|f_{\tau_{k_{m-1}^*}}^{k_{m-1}^*}|^2)^{7/2}\le (C\log R)^{2\e^{-1}}\int_{\R^3}(\sum_\theta|f_\theta|^2)^{7/2}\]
and the algorithm terminates with the inequality 
\begin{align}\label{term1} 
(\text{L.H.S. of \eqref{algline1}})\le (C\log &R)^{4\e^{-1}(m-1)+2\e^{-1}+1}(B_{\e^7}R^{\e^7})^{m-1}(E_{\e^7}R^{\e^7})^{a}(R^{\e^3}\S_1(R^\e))^{b+1}  \int_{\R^3}(\sum_\theta|f_\theta|^2)^{7/2} 
\end{align}
in which $a+b=m-1$ and $R^{-\frac{1}{3}}\le \min(R_k^{-\frac{1}{3}}R^{-\frac{\e^3}{2}a},R_k^{-\frac{1}{3}}R^{-\frac{\e}{3}b})$. We will verify further below that these termination conditions prove the proposition. 

The other case from Lemma \ref{algo2} is that there is some dyadic value $s_m$, $R^{-\frac{1}{3}}\le s_m\le R_{k_{m-1}^*}^{-\frac{1}{3}}$, which satisfies
\begin{align*}
\sum_{\tau\in{\bf{S}}(s_{m-1})}\int_{\R^3}(&\sum_{\substack{\tau_{k_{m-1}^*}\subset\tau'}}|f_{\tau_{k_{m-1}^*}}^{k_{m-1}^*}|^2)^{7/2}&\le (C\log R)^{\e^{-1}}B_{\e^7}R^{\e^7} \sum_{s_m\le \sigma\le 1}\sum_{\substack{\tau\in{\bf{S}}(\sigma^{-1}s_m)}} \sum_{V\|V_{\tau,s_m^{-3}}}|V|\Big(\fint_V\sum_{\substack{\tau'\subset\tau \\\tau'\in{\bf{S}}(s)}}|f_{\tau'}^{k_m}|^{7/2} \Big)^{2}
\end{align*}
where $k_m\ge k_{m-1}^*$ satisfies $R_{k_m}^{-\frac{1}{3}}\le  s_m\le R_{k_m-1}^{-\frac{1}{3}}$. By definition of $k_{m-1}^*$, $R_{k_{m-1}^*}^{-\frac{1}{3}}< s_{m-1}$. Then since $s_m\le R_{k_{m-1}^*}^{-\frac{1}{3}}< s_{m-1}\le R_k^{-\frac{1}{3}}R^{-\frac{\e}{3}b}$, $R_{k_{m-1}}=R^{k_{m-1}\e}$, and $R_kR^{\e b}=R^{(k+b)\e}$, it follows that $s_m\le R_k^{-\frac{1}{3}}R^{-\frac{\e}{3}(b+1)}$. We have shown in this case that 
\begin{align*} 
(\text{L.H.S. of \eqref{algline1}})\le (C\log R)^{4\e^{-1}(m+1)}&(B_{\e^7}R^{\e^7})^{m}(E_{\e^7}R^{\e^7})^{a}(R^{\e^3}\S_1(R^\e))^{b+1}  \\
&\times \sum_{s_m\le \sigma_0\le 1}\sum_{\substack{\tau\in{\bf{S}}(\sigma_0^{-1}s_m)}} \sum_{V\|V_{\tau,s_m^{-3}}}|V|\Big(\fint_V\sum_{\substack{\tau'\subset\tau \\\tau'\in{\bf{S}}(s_m)}}|f_{\tau'}^{k_m}|^{7/2} \Big)^{2}\nonumber
\end{align*}
in which $a+(b+1)=m+1$ and $R^{-\frac{1}{3}}\le s_m\le \min(R_k^{-\frac{1}{3}}R^{-\frac{\e^3}{2}a},R_k^{-\frac{1}{3}}R^{-\frac{\e}{3}(b+1)})$ and $k_m\ge k+1$ satisfies $R_{k_m}^{-\frac{1}{3}}\le s_m\le R_{k_m-1}^{-\frac{1}{3}}$, which verifies \eqref{stepm}. 

\noindent\fbox{Step m: $s_{m-1}\le \sigma\le R^{-\e^3}$.} Let $\tilde{s}_{m-1}=\max(R^{-\frac{1}{3}},\sigma^{\frac{1}{2}}s_{m-1})$. By Lemma \ref{algo1}, for each $\tau\in{\bf{S}}(\sigma^{-1}s_{m-1})$ and $V\|V_{\tau,s_{m-1}^{-3}}$, 
\[\fint_V\sum_{\substack{\tau'\subset\tau \\\tau'\in{\bf{S}}(s_{m-1})}}|f_{\tau'}^{k_{m-1}}|^{7/2} \le E_{\e^7}R^{\e^7}\fint_V\sum_{\substack{\tau'\subset\tau \\\tau'\in{\bf{S}}(s_{m-1})}}|\sum_{\substack{\tau''\subset\tau'\\\tau''\in{\bf{S}}(\tilde{s}_{m-1})}}|f_{\tau''}^{k_m'}|^2|^{7/4}  \]
where $k_{m}'\ge k+1$ satisfies $R_{k_m'}^{-\frac{1}{3}}\le \tilde{s}_{m-1}\le R_{k_m'-1}^{-\frac{1}{3}}$. Then using $\|\cdot\|_{\ell^{7/4}}\le\|\cdot\|_{\ell^{1}}$ and Cauchy-Schwarz, we have 
\begin{align*} 
\sum_{\substack{\tau\in{\bf{S}}(\sigma^{-1}s_{m-1})}} \sum_{V\|V_{\tau,s_{m-1}^{-3}}}|V|\Big(\fint_V&\sum_{\substack{\tau'\subset\tau \\\tau'\in{\bf{S}}(s_{m-1})}}(\sum_{\substack{\tau''\subset\tau'\\\tau''\in{\bf{S}}(\tilde{s}_{m-1})}}|f_{\tau''}^{k_m'}|^2)^{7/4} \Big)^{2}\\
&\le  \sum_{\substack{\tau\in{\bf{S}}(\sigma^{-1}s_{m-1})}} \sum_{V\|V_{\tau,s_{m-1}^{-3}}}|V|\Big(\fint_V(\sum_{\substack{\tau''\subset\tau\\\tau''\in{\bf{S}}(\tilde{s}_{m-1})}}|f_{\tau''}^{k_m'}|^2)^{7/4} \Big)^{2}\\
&\le \sum_{\substack{\tau\in{\bf{S}}(\sigma^{-1}s_{m-1})}} \int_{\R^3}(\sum_{\substack{\tau''\subset\tau\\\tau''\in{\bf{S}}(\tilde{s}_{m-1})}}|f_{\tau''}^{k_m'}|^2)^{7/2} . 
\end{align*}
If $\tilde{s}_{m-1}=R^{-\frac{1}{3}}$, then the algorithm terminates with the inequality 
\begin{align}\label{term2} 
(\text{L.H.S. of \eqref{algline1}})\le (C\log &R)^{4\e^{-1}(m-1)+2\e^{-1}+1}(B_{\e^7}R^{\e^7})^{m-1}(E_{\e^7}R^{\e^7})^{a+1}(R^{\e^3}\S_1(R^\e))^{b}  \int_{\R^3}(\sum_\theta|f_\theta|^2)^{7/2} 
\end{align}
in which $a+b=m-1$ and $R^{-\frac{1}{3}} \le \min(R_k^{-\frac{1}{3}}R^{-\frac{\e^3}{2}a},R_k^{-\frac{1}{3}}R^{-\frac{\e}{3}b})$. Otherwise, suppose that $\tilde{s}_{m-1}>R^{-\frac{1}{3}}$ and apply Lemma \ref{algo2}, again leading to two cases. In the case that 
\[\sum_{\substack{\tau\in{\bf{S}}(\sigma^{-1}s_{m-1})}} \int_{\R^3}(\sum_{\substack{\tau''\subset\tau\\\tau''\in{\bf{S}}(\sigma^{\frac{1}{2}}s_{m-1})}}|f_{\tau''}^{k_m'}|^2)^{7/2}\le (C\log R)^{2\e^{-1}}\int_{\R^3}(\sum_\theta|f_\theta|^2)^{7/2}\]
then the algorithm terminates with the same inequality as  is recorded in \eqref{term2}. 

The other case from Lemma \ref{algo2} is that there is some dyadic value $s_m$, $R^{-\frac{1}{3}}\le s_m\le \sigma^{\frac{1}{2}}s_{m-1}$, which satisfies
\begin{align*}
\sum_{\tau\in{\bf{S}}( \sigma^{-1}s_{m-1})}\int_{\R^3}(\sum_{\substack{\tau''\subset\tau\\\tau''\in{\bf{S}}(\sigma^{\frac{1}{2}}s_{m-1})}}|f_{\tau''}^{k_m'}|^2)^{7/2}&\le (C\log R)^{\e^{-1}}B_{\e^7}R^{\e^7} \\
&\qquad\times \sum_{s_m\le \sigma_0\le 1}\sum_{\substack{\tau\in{\bf{S}}(\sigma_0^{-1}s_m)}} \sum_{V\|V_{\tau,s_m^{-3}}}|V|\Big(\fint_V\sum_{\substack{\tau'\subset\tau \\\tau'\in{\bf{S}}(s_m)}}|f_{\tau'}^{k_m}|^{7/2} \Big)^{2}
\end{align*}
where $k_m\ge k_{m-1}^*$ satisfies $R_{k_m}^{-\frac{1}{3}}\le  s_m\le R_{k_m-1}^{-\frac{1}{3}}$. Since $s_{m-1}\le \min(R_k^{-\frac{1}{3}}R^{-\frac{\e^3}{2}a},R_k^{-\frac{1}{3}}R^{-\frac{\e}{3}b})$ and $s_m\le \sigma^{\frac{1}{2}}s_{m-1}$, we have $s_m\le \min(R_k^{-\frac{1}{3}}R^{-\frac{\e^3}{2}(a+1)},R_k^{-\frac{1}{3}}R^{-\frac{\e}{3}b})$. We have shown in this case that 
\begin{align*} 
(\text{L.H.S. of \eqref{algline1}})\le (C\log R)^{4\e^{-1}(m+1)}&(B_{\e^7}R^{\e^7})^{m}(E_{\e^7}R^{\e^7})^{a+1}(R^{\e^3}\S_1(R^\e))^{b}  \\
&\times \sum_{s_m\le \sigma_0\le 1}\sum_{\substack{\tau\in{\bf{S}}(\sigma_0^{-1}s_m)}} \sum_{V\|V_{\tau,s_m^{-3}}}|V|\Big(\fint_V\sum_{\substack{\tau'\subset\tau \\\tau'\in{\bf{S}}(s_m)}}|f_{\tau'}^{k_m}|^{7/2} \Big)^{2}\nonumber
\end{align*}
in which $(a+1)+b=m+1$ and $R^{-\frac{1}{3}}\le s_m\le \min(R_k^{-\frac{1}{3}}R^{-\frac{\e^3}{2}(a+1)},R_k^{-\frac{1}{3}}R^{-\frac{\e}{3}b})$ and $k_m\ge k+1$ satisfies $R_{k_m}^{-\frac{1}{3}}\le s_m\le R_{k_m-1}^{-\frac{1}{3}}$, which verifies \eqref{stepm}.

\noindent\fbox{Termination criteria.} It remains to check that the criteria for termination from both of the above cases implies Proposition \ref{algo}. The first case in which the algorithm terminates is \eqref{term1}. Then $R^{-\frac{1}{3}}\le \min(R_k^{-\frac{1}{3}}R^{-\frac{\e^3}{2}a},R_k^{-\frac{1}{3}}R^{-\frac{\e}{3}b})$ implies that $a\le\frac{2}{3}\e^{-3}(N-k)$ and $b\le N-k$, which implies that $m-1=a+b\le \e^{-3}(N-k)$. Then the constants in the upper bound from \eqref{term1} are bounded by 
\begin{align*} 
(C\log R)^{4\e^{-1}(m-1)+2\e^{-1}+1}&(B_{\e^7}R^{\e^7})^{m-1}(E_{\e^7}R^{\e^7})^{a}(R^{\e^3}\S_1(R^\e))^{b+1}\\
&\le (C\e^{-7}R^{\e^7})^{\e^{-5}}B_{\e^7}^{\e^{-4}}R^{\e^3}E_{\e^7}^{\e^{-4}}R^{\e^3}R^{\e^2}\S_1(R^\e)^{N-k+1}\le C_\e R^{3\e^2}\S_1(R^\e)^{N-k+1} .
\end{align*}
The second case in which the algorithm terminates with \eqref{term2} proves the proposition by an analogous argument.

\end{proof}

\section{Bounding the broad part of $U_{\a,\b}$ \label{broad} }

For three canonical blocks $\tau^1,\tau^2,\tau^3$ (with dimensions $\sim R^{-\e/3}\times R^{-2\e/3}\times R^{-\e}$) which are pairwise $\ge R^{-\e/3}$-separated, define the broad part of $U_{\a,\b}$ to be
\[ \text{Br}_{\a,\b}^K=\{x\in U_{\a,\b}: \a\le K|f_{\tau^1}(x)f_{\tau^2}(x)f_{\tau^3}(x)|^{\frac{1}{3}},\quad\max_{\tau^i}|f_{\tau^i}(x)|\le \a\}. \]

We bound the broad part of $U_{\a,\b}$ in the following proposition. Recall that the parameter $N_0$ was used in the definition of the sets $\Omega_k$ and $L$. 

\begin{proposition}\label{mainprop} Let $R,K\ge1$. Suppose that $\|f_\theta\|_{L^\infty(\R^3)}\le 2$ for all $\theta\in{\bf{S}}(R^{-\frac{1}{3}})$. Then  
\begin{equation*} 
\a^{7}|\text{\emph{Br}}_{\a,\b}^{K}|\le \big[CR^{10 \e N_0}A^{\e^{-1}} + K^{50}R^{4\e^2+10\e}A^{\e^{-1}}  \emph{S}_1(R^\e)^{\e^{-1}-N_0}\big]\int\Big|\sum_\theta|f_\theta|^2*\w_\theta\Big|^{\frac{7}{2}} .
\end{equation*} 
\end{proposition}

We will use the following version of a local trilinear restriction inequality for the moment curve, which was proved in Proposition 6 of \cite{M3smallcap}. 
The weight function $W_{B_r}$ in the following theorem decays by a factor of $10$ off of the ball $B_r$. It is defined in Definition \ref{M3ballweight}. 
\begin{thm}\label{trirestprop}
Let $s\ge 10r\ge10$ and let $f:\R^3\to\C$ be a Schwartz function with Fourier transform supported in $\mc{N}_{r^{-1}}(\mc{M}^3)$. Suppose that $\tau^1,\tau^2,\tau^3\in{\bf{S}}(R^{-\e/3})$ satisfy  $\text{dist}(\tau^i,\tau^j)\ge s^{-1}$ for $i\not=j$. Then 
\[  \int_{B_r}|f_{\tau^1}f_{\tau^2}f_{\tau^3}|^2\lesssim s^3|B_r|^{-2}\big(\int|f_{\tau^1}|^2W_{B_r}\big)\big(\int|f_{\tau^2}|^2W_{B_r}\big)\big(\int|f_{\tau^3}|^2W_{B_r}\big) \]
for any Schwartz function $f:\R^3\to\C$ with Fourier transform supported in $\mc{M}(r)$. 
\end{thm}

\begin{proof}[Proof of Proposition \ref{mainprop}]
Note that
\[ \text{Br}_{\a,\b}^K=(L\cap \text{Br}_{\a,\b}^K)\cup( \sqcup_{k=B}^{N-1}\Omega_k\cap \text{Br}_{\a,\b}^K)\]
We bound each of the sets $\text{Br}_{\a,\b}^K\cap\Omega_k$ and $\text{Br}_{\a,\b}^K\cap L$ in separate cases. It suffices to consider the case that $R$ is at least some constant depending on $\e$ since if $R\le C_\e$, we may prove the proposition using trivial inequalities. 

\vspace{2mm}
\noindent\underline{Case 1: bounding $| \text{Br}_{\a,\b}^{K}\cap\Omega_k|$. } 
By Lemma \ref{ftofk}, 
\[ |\text{Br}_{\a,\b}^K\cap\Omega_k|\le|\{x\in  U_{\a,\b}\cap\Omega_k:\a\lesssim K|f_{\tau^1}^{k+1}(x)f_{\tau^2}^{k+1}(x)f_{\tau^3}^{k+1}(x)|^{\frac{1}{3}},\quad\max_{\tau^i}|f_{\tau^i}(x)|\le \a\}.\]
By Lemma \ref{pruneprop}, the Fourier supports of $f_{\tau^1}^{k+1},f_{\tau^2}^{k+1},f_{\tau^3}^{k+1}$ are contained in $2\tau^1,2 \tau^2,2\tau^3$ respectively, which are $\ge5 R^{-\frac{\e}{3}}$-separated blocks of the moment curve. Let $\{B_{R_k^{\frac{1}{3}}}\}$ be a finitely overlapping cover of $\text{Br}_{\a,\b}^{K}\cap\Omega_k$ by $R_k^{\frac{1}{3}}$-balls. For $R$ large enough depending on $\e$, apply Theorem \ref{trirestprop} to get
\begin{align*}
    \int_{B_{R_k^{\frac{1}{3}}}}|f_{\tau^1}^{k+1}f_{\tau^2}^{k+1}f_{\tau^3}^{k+1}|^2&\lesssim_\e R^{\e} |B_{R_k^{\frac{1}{3}}}|^{-2}\Big(\int|f_{\tau^1}^{k+1}|^2W_{B_{R_k^{\frac{1}{3}}}}\Big)\Big(\int|f_{\tau^2}^{k+1}|^2W_{B_{R_k^{\frac{1}{3}}}}\Big)\Big(\int|f_{\tau^3}^{k+1}|^2W_{B_{R_k^{\frac{1}{3}}}}\Big).
\end{align*}
Using local $L^2$-orthogonality (Lemma \ref{L2orth}), each integral on the right hand side above is bounded by 
\[ \lesssim \int\sum_{\tau_k}|f_{\tau_k}^{k+1}|^2W_{B_{R_k^{\frac{1}{3}}}}. \]
If $x\in \text{Br}_{\a,\b}^{K}\cap\Omega_k\cap B_{R_k^{\frac{1}{3}}}$, then the above integral is bounded by 
\[ \lesssim  \int \sum_{\tau_k}|f_{\tau_k}^{k+1}|^2*\w_{\tau_k}W_{B_{R_k^{\frac{1}{3}}}}\lesssim C |B_{R_k^{\frac{1}{3}}}| \sum_{\tau_k}|f_{\tau_k}^{k+1}|^2*\w_{\tau_k}(x) \]
by the locally constant property (Lemma \ref{locconst}) and properties of the weight functions. The summary of the inequalities so far is that 
\[ \a^6|\text{Br}_{\a,\b}^{K}\cap\Omega_k\cap B_{R_k^{\frac{1}{3}}}|\lesssim_\e K^6\int_{B_{R_k^{\frac{1}{3}}}}|f_{\tau^1}^{k+1}f_{\tau^2}^{k+1}f_{\tau^3}^{k+1}|^2\lesssim_\e R^\e K^6 |B_{R_k^{\frac{1}{3}}}|g_k(x)^3 \]
where $x\in \text{Br}_{\a,\b}^{K}\cap\Omega_k\cap B_{R_k^{\frac{1}{3}}}$. 

Recall that since $x\in\Omega_k$, we have the lower bound $A^{M-k}r\le g_k(x)$ (where $A$ is from Definition \ref{impsets}), which leads to the inequality
\[ \a^6|\text{Br}_{\a,\b}^{K}\cap\Omega_k\cap B_{R_k^{\frac{1}{3}}}|\lesssim_\e K^6 R^{\e} \frac{1}{A^{M-k}r}|B_{R_k^{\frac{1}{3}}}|g_k(x)^{3+1} .\]
By Corollary \ref{highdom}, we also have the upper bound $|g_k(x)|\le 2|g_k^h(x)|$, so that 
\[ \a^6|\text{Br}_{\a,\b}^{K}\cap\Omega_k\cap B_{R_k^{\frac{1}{3}}}|\lesssim_\e K^6 R^{\e} \frac{1}{A^{M-k}r}|B_{R_k^{\frac{1}{3}}}||g_k^h(x)|^{4} .\]
By the locally constant property applied to $g_k^h$, $|g_k^h|^{4}\lesssim_\e |g_k^h*w_{ B_{R_k^{\frac{1}{3}}}}|^{4}$ and by Cauchy-Schwarz, $|g_k^h*w_{ B_{R_k^{\frac{1}{3}}}}|^{4}\lesssim |g_k^h|^{4}*w_{B_{R_k^{\frac{1}{3}}}}$. Combine this with the previous displayed inequality to get 
\[ \a^6|\text{Br}_{\a,\b}^{K}\cap\Omega_k\cap B_{R_k^{\frac{1}{3}}}|\lesssim_\e K^6 R^{\e} \frac{1}{A^{M-k}r}\int|g_k^h|^{4}W_{ B_{R_k^{\frac{1}{3}}}} .\]
Summing over the balls $B_{R_k^{\frac{1}{3}}}$ in our finitely-overlapping cover of $\text{Br}_{\a,\b}^{K}\cap\Omega_k$, we conclude that
\begin{equation}\label{peqn} \a^6|\text{Br}_{\a,\b}^{K}\cap\Omega_k|\lesssim_\e K^6 R^{\e} \frac{1}{A^{M-k}r}\int_{\R^3}|g_k^h|^{4} .\end{equation}
We are done using the properties of the set $\text{Br}_{\a,\b}^{K}\cap\Omega_k$, which is why we now integrate over all of $\R^3$ on the right hand side. We will now use Lemma \ref{high1} to analyze the high part $g_k^h$. In particular, Lemma \ref{high1} gives 
\begin{equation}\label{highapp} \int|g_k^h|^4\lesssim_\e R^{\e}\sum_{R_k^{-1/3}\le\sigma\le 1}\sum_{\tau\in{\bf{S}}(\sigma^{-1}R_k^{-\frac{1}{3}})}\sum_{V\|V_{\tau,R_k}}|V|\left(\fint_V\sum_{\tau_k\subset\tau}|f_{\tau_k}^{k+1}|^4\right)^2 .\end{equation}
Next use \eqref{item2} from Lemma \ref{pruneprop} to note that $\|f_{\tau_k}^{k+1}\|_\infty\le \sum_{\tau_{k+1}\subset\tau_k}\|f_{\tau_{k+1}}^{k+1}\|_\infty\lesssim R^\e A^{\e^{-1}}K^3\frac{\b}{\a}$: for each $R_k^{-1/3}\le\sigma\le1$, $\tau\in{\bf{S}}(\sigma^{-1}R_k^{-\frac{1}{3}})$ and $V\|V_{\tau,R_k}$, 
\begin{align*}
&|V|\left(\fint_V\sum_{\tau_k\subset\tau}|f_{\tau_k}^{k+1}|^4\right)^2\lesssim(A^{\e^{-1}}R^{\e}K^3\frac{\b}{\a}) |V|\left(\fint_V\sum_{\tau_k\subset\tau}|f_{\tau_k}^{k+1}|^{7/2}\right)^2. 
\end{align*} 
Using this and applying Proposition \ref{algo} gives the upper bound 
\begin{align*}
\sum_{R_k^{-1/3}\le\sigma\le 1}\sum_{\ell(\tau)=\sigma^{-1}R_k^{-\frac{1}{3}}}\sum_{V\|V_{\tau,R_k}}|V|&\left(\fint_V\sum_{\tau_k\subset\tau}|f_{\tau_k}^{k+1}|^4\right)^2\\
&\lesssim_\e (A^{\e^{-1}}R^{\e}K^3\frac{\b}{\a}) R^{3\e^2}\S_1(R^\e)^{N-k+1}\int(\sum_{\theta\in{\bf{S}}(R^{-1/3})}|f_\theta|^2)^{7/2}. 
\end{align*} 
Combining this with \eqref{peqn} and \eqref{highapp}, the summary of the argument from this case is 
\[\a^7|U_{\a,\b}|\lesssim_\e K^6 R^{2\e}(A^{\e^{-1}}R^{\e}K^3)\S_1(R^\e)^{N-k+1}\int(\sum_{\theta\in{\bf{S}}(R^{-1/3})}|f_\theta|^2)^{7/2}. \]
Since $k>B$, this upper bound has the desired form.

\noindent\underline{Case 2: bounding $|U_{\a,\b}\cap L|$.} 
Begin by using Lemma \ref{ftofk} to bound
\[ \a^{7}|\text{Br}_{\a,\b}^K\cap L|\lesssim K^{7} \int_{U_{\a,\b}\cap L}|f^{B+1}|^{7}. \]
Then use Cauchy-Schwarz and the locally constant property for $g_B$ :
\[ \int_{U_{\a,\b}\cap L}|f^{B}|^{7}\lesssim (R^{ B\e/3})^{7/2} \int_{U_{\a,\b}\cap L}(\sum_{\tau_B}|f_{\tau_B}^{B+1}|^2)^{7/2}\lesssim (R^{ B\e/3})^{7/2} \int_{U_{\a,\b}\cap L}(g_B)^{7/2}.\]
Using the definition the definition of $L$, we bound the factors of $G_B$ by
\[ \int_{U_{\a,\b}\cap L} (A^{\e^{-1}}\b)^{7/2}. \]
Finally, by the definition of $U_{\a,\b}$, conclude that
\[    \a^{7}|\text{Br}_{\a,\b}^K\cap L|\lesssim_\e K^{7}R^{2B\e}A^{\e^{-1}}\int_{\R^3}|\sum_\theta|f_\theta|^2*\w_\theta|^{7/2}. \]

\end{proof}

\section{Proof of Theorem \ref{main} from Proposition \ref{mainprop} \label{mainsec}}

First, we prove Proposition \ref{S1bd} below, which is that $\S_1(R)\lesssim_\e R^\e$. This follows using various reductions from pigeonholing, a broad-narrow argument, and the broad estimate Proposition \ref{mainprop}. Then, in \textsection\ref{S2}, we use induction to show that Proposition \eqref{S1bd} implies that $\S_2(R)\lesssim_\e R^\e$, which is equivalent to Theorem \ref{main}.

\begin{proposition}\label{S1bd} For any $\e>0$ and $R\ge 1$, 
\[ \emph{S}_1(R)\lesssim_\e R^\e. \]
\end{proposition}
In order to make use of Proposition \ref{mainprop}, we need to reduce to the case that our function $f$ is localized to a ball, its wave packets have been pigeonholed so that $\|f_\theta\|_\infty\lesssim 1$ for all $\theta\in{\bf{S}}(R^{-1/3})$, and we have approximated $\|f\|_7$ by an expression involving a superlevel set. This is the content of the following subsection. 

\subsection{Wave packet decomposition and pigeonholing \label{M3pigeon}}

Begin with the spatial localization. 
\begin{lemma}\label{loc} For any $R$-ball $B_R\subset\R^3$, suppose that 
\[     \|f\|_{L^7(B_R)}^7\lesssim_\e  R^\e\int\big|\sum_{\theta\in{\bf{S}}(R^{-1/3})}|f_\theta|^2*\w_\theta\big|^{\frac{7}{2}} \]
for any Schwartz function $f:\R^3\to\C$ with Fourier transform supported in $\mc{M}^3(R)$.  Then Proposition \ref{S1bd} is true. 
\end{lemma}
\begin{proof} 

If $\phi_{B_R}$ is a Schwartz function that rapidly decays away from $B_R$, has Fourier transform supported in $B_{R^{-1}}(0)$, and satisfies $\phi_{B_R}\gtrsim 1$ on $B_R$, then we may apply the hypothesis to the function $\phi_{B_R}f$ to obtain 
\begin{align*}
    \sum_{B_R}\int_{B_R}|f|^7\lesssim_\e R^\e \sum_{B_R}\int|\sum_{\theta\in{\bf{S}}_{R^{-1/3}}}|\phi_{B_R}f_\theta|^2*\w_\theta|^{\frac{7}{2}}
\end{align*}
where the sum is over a finitely overlapping cover of $\R^3$ by $R$-balls. Then, since $\|\cdot\|_{\ell^{7/2}}\le\|\cdot\|_{\ell^1}$, 
\[\sum_{B_r}\int|\sum_{\theta\in{\bf{S}}(R^{-1/3})}|\phi_{B_r}f_\theta|^2*\w_\theta|^{\frac{7}{2}}\le \int|\sum_{\theta\in{\bf{S}}(R^{-1/3})}\sum_{B_r}|\phi_{B_r}f_\theta|^2*\w_\theta|^{\frac{7}{2}}.\]
It remains to note that $\sum_{B_r}|\phi_{B_r}|^2\lesssim 1$.

\end{proof}

It further suffices to prove a weak, level-set version of Proposition \ref{S1bd}.  
\begin{lemma}\label{alph} For each $B_R$ and Schwartz function $f:\R^3\to\C$ with Fourier transform supported in $\mc{M}^3(R)$, there exists $\a>0$ such that
\[ \|f\|_{L^7(B_R)}^7\lesssim (\log R)\a^7|\{x\in B_R:\a\le |f(x)|\}|+R^{-500}\int|\sum_{\theta\in{\bf{S}}(R^{-1/3})}|f_\theta|^2*\w_\theta|^{7/2}. \]
\end{lemma}
\begin{proof} Split the integral as follows:
\[ \int_{B_R}|f|^{7}=\sum_{R^{-1000}\le \lambda\le 1}\int_{\{x\in B_R:\a\|f\|_{L^\infty(B_R)}\le |f(x)|\le 2\a \|f\|_{L^\infty(B_R)}\}}|f|^{7}+\int_{\{x\in B_R:|f(x)|\le R^{-1000}\|f\|_{L^\infty(B_R)}\}}|f|^7 \]
in which $\lambda$ varies over dyadic values in the range $[R^{-1000},1]$. If one of the $\lesssim \log R$ many terms in the first sum dominates, then we are done. Suppose instead that the second expression dominates:
\[ \int_{B_R}|f|^{7}\le 2\int_{\{x\in B_R:|f(x)|\le R^{-1000}\|f\|_{L^\infty(B_R)}\}}|f|^7. \]
We have 
\begin{align*}
\int_{\{x\in B_R:|f(x)|\le R^{-1000}\|f\|_{L^\infty(B_R)}\}}|f|^7&\le R^{-1000}|B_R|\|f\|_{L^\infty(B_R)}^7 \\
(\text{Cauchy-Schwarz})\qquad &\lesssim  R^{-1000}|B_R|R^2\|\sum_\theta|f_\theta|^2\|_{L^\infty(B_R)}^{7/2} \\
(\text{loc. const. for $\sum_\theta|f_\theta|^2$})\qquad &\lesssim  R^{-1000}|B_R|R^2\|\sum_\theta|f_\theta|^2*w_{R^{1/3}}\|_{L^\infty(B_R)}^{7/2} \\
(\text{loc. const. for each $|f_\theta|^2$})\qquad&\lesssim  R^{-995}\|\sum_\theta|f_\theta|^2*\w_\theta*w_{R^{1/3}}\|_{L^\infty(B_R)}^{7/2} \\
(\text{H\"{o}lder's inequality})\qquad&\lesssim  R^{-995}\||\sum_\theta|f_\theta|^2*\w_\theta|^{7/2}*w_{R^{1/3}}\|_{L^\infty(B_R)} \\
&\lesssim  R^{-995}\int|\sum_\theta|f_\theta|^2*\w_\theta|^{7/2}.
\end{align*}

\end{proof}

Continue to use the notation 
\[ U_{\a}=\{x\in B_R:\a\le |f(x)|\}. \]
We will show that to estimate the size of $U_{\a}$, it suffices to replace $f$ with a version whose wave packets at scale $\theta$ have been pigeonholed. Write 
\begin{align}\label{sum} f=\sum_\theta\sum_{T\in\T_\theta}\s_Tf_\theta \end{align}
where for each $\theta\in{\bf{S}}(R^{-1/3})$, $\{\s_T\}_{T\in\T_\theta}$ is the partition of unity from \textsection\ref{prusec}. If $\a\le C_\e R^{-100}\max_\theta\|f_\theta\|_{\infty}$, then using a similar argument that bounds the second expression in the proof of Lemma \ref{alph}, the inequality
\[ \a^7|U_{\a}|\lesssim_\e R^\e\int|\sum_\theta|f_\theta|^2*\w_\theta|^{7/2} \]
is trivial. 

\begin{proposition}[Wave packet decomposition] \label{wpd} Let ${\a}>C_\e R^{-100}\max_\theta\|f_\theta\|_{L^\infty(\R^3)}$. There exist subsets $\tilde{\T}_\theta\subset\T_\theta$, as well as a constant $A>0$ with the following properties:
\begin{align} 
|U_{\a}|\lesssim (\log R)|\{x\in U_{\a}:\,\,{{\a}}&\lesssim |\sum_{\theta\in{\bf{S}}(R^{-1/3})}\sum_{T\in\tilde{\T}_\theta}\s_T(x)f_\theta (x)|\,\,\}|, \\
R^\e T\cap U_{\a}\not=\emptyset\qquad&\text{for all}\quad\theta\in{\bf{S}}(R^{-1/3}),\quad T\in\tilde{\T}_\theta\\
A\lesssim \|\sum_{T\in\tilde{\T}_\theta}\s_Tf_\theta\|_{L^\infty(\R^3)}&\lesssim R^{3\e} A\qquad\text{for all}\quad  \theta\in{\bf{S}}(R^{-1/3})\label{propM}\\
\|\s_Tf_\theta\|_{L^\infty(\R^3)}&\sim A\qquad\text{for all}\quad  \theta\in{\bf{S}}(R^{-1/3}),\quad T\in\tilde{\T}_\theta. \label{prop'M}
\end{align}

\end{proposition}

\begin{proof} 
Split the sum (\ref{sum}) into 
\begin{equation}\label{step1C} f=\sum_\theta\sum_{T\in\T_\theta^c}\s_Tf_\theta+\sum_\theta\sum_{T\in\T_\theta^f}\s_Tf_\theta\end{equation} 
where the close set is
\[ \T_\theta^c:=\{T\in\T_\theta:R^\e T\cap U_{ {\a}}\not=\emptyset\}\]
and the far set is 
\[ \T_\theta^f:=\{T\in\T_\theta:R^\e T\cap U_{ {\a}}=\emptyset\} . \]
Using the rapid decay of the partition of unity, for each $x\in U_{ {\a}}$, 
\[ |\sum_\theta\sum_{T\in\T_\theta^f}\s_T(x)f_\theta(x)|\lesssim_\e R^{-1000}\max_\theta\|f_\theta\|_{L^\infty(B_R)}. \] 
Therefore, using the assumption that ${{\a}}$ is at least $C_\e R^{-100}\max_\theta\|f_\theta\|_{L^\infty(\R^3)}$, 
\[ |U_{ {\a}}|\le 2|\{x\in U_{ {\a}}:\,\,{ {\a}}\le 2 |\sum_\theta\sum_{T\in\tilde{\T}_\theta^c}\s_T(x)f_\theta(x)|\,\,\}|. \]
Now we sort the close wave packets according to amplitude. Let
 \begin{equation}\label{eq: defMC}
	M=\max_\theta\max_{T\in\T_\theta^c}\|\s_Tf_\theta\|_{L^\infty(\R^3)}.
	\end{equation} 
Split the remaining wave packets into 
\begin{equation} \label{step2C}
    \sum_\theta\sum_{T\in\T_\theta^c}\s_Tf_\theta=\sum_\theta\sum_{R^{-10^3}\le \lambda\le 1}\sum_{T\in\T_{\theta,\lambda}^c}\s_Tf_\theta+\sum_\theta\sum_{T\in\T_{\theta,s}^c}\s_Tf_\theta
\end{equation}
where $\lambda$ is a dyadic number in the range $[R^{-10^3},1]$,  
\[ \T_{\theta,\lambda}^c:=\{T\in\T_\theta^c:\|\s_Tf_\theta\|_{L^\infty(\R^3)}\sim \lambda M \},\]
and
\[ \T_{\theta,s}^c:= \{T\in\T_\theta^c:\|\s_Tf_\theta\|_{L^\infty(\R^3)}\le R^{-1000}M \} . \]
Again using the lower bound for ${\a}$ (and the fact that the number of $T\in\T_\theta^c$ is bounded by $R^4$), the small wave packets cannot dominate and we have 
\[ |U_{ {\a}}|\le 4|\{x\in U_{ {\a}}:\,\,{ {\a}}\le 4|\sum_\theta \sum_{R^{-10^3}\le \lambda\le 1}\sum_{T\in\T_{\theta,\lambda}^c}\s_T(x)f_\theta(x)|\,\,\}|.\]
By dyadic pigeonholing, for some $\lambda\in [R^{-1000},1]$, 
\[ |U_{ {\a}}|\lesssim (\log R)|\{x\in U_{ {\a}}:\,\,{ {\a}}\lesssim (\log R)|\sum_\theta \sum_{T\in\T_{\theta,\lambda}^c}\s_T(x)f_\theta(x)|\,\,\}|. \]
Note that we have the pointwise inequality 
\begin{align*} 
|\sum_{T\in\T_{\theta,\lambda}^c}\s_T(x)f_\theta(x)|&= |\sum_{\substack{T\in\T_{\theta,\lambda}^c\\ x\in R^\e T}}\s_T(x)f_\theta(x)|+|\sum_{\substack{T\in\T_{\theta,\lambda}^c\\ x\not\in R^\e T}}\s_T(x)f_\theta(x)| \\
&\le  |\sum_{\substack{T\in\T_{\theta,\lambda}^c\\ x\in R^\e T}}\s_T(x)f_\theta(x)|+C_\e R^{-1000}|f_\theta(x)| .
\end{align*}
We know that $\lambda M\ge C_\e R^{-1000}\max_\theta\|f_\theta\|_{L^\infty(\R^3)}$ since if this did not hold, we would violate the lower bound for $\a$. It follows that
\[ \lambda M\le \|\sum_{T\in\T_{\theta,\lambda}^c}\s_Tf_\theta\|_{L^\infty(\R^3)}\le 3R^{3\e}\lambda M. \]
The statement of the lemma is now satisfied with $A=\lambda M$ and $\tilde{\T}_\theta=\T_{\theta,\lambda}^c$.

\end{proof}

\begin{corollary}\label{wpdcor} Let $f$, $\a$, $\tilde{T}_\theta$ and $A>0$ be as in Proposition \ref{wpd}. Then for each $x\in U_\a$,
\[ \a\le R^{103\e}\frac{1}{A} \sum_{\theta\in\mc{S}}|\sum_{T\in\tilde{\T}_\theta}\s_T f_\theta|^2*\w_\theta(x). \]
\end{corollary}

\begin{proof} Let $x\in U_{\a}$. Then using the rapid decay of $\s_T$ off of $T$, 
\[\a\lesssim \big|\sum_{\theta}\sum_{\substack{T\in\tilde{\T}_\theta\\ x\in R^\e T}}\s_T(x)f_\theta(x)\big|+C_\e R^{-1000}\max_\theta\|f_\theta\|_{L^\infty(\R^3)} .\]
Then the lower bound for $\a$ and \eqref{prop'M} imply that
\[ \a\lesssim \#\{\theta:\exists T\in\tilde{\T}_\theta\quad\text{satisfying}\quad x\in R^\e T\}A. \]
By the locally constant property, if $\eta_\theta$ is a standard bump function equal to $1$ on $\theta$, then 
\[ \|\s_Tf_\theta\|_{L^\infty(\R^3)}^2\lesssim \int|\s_T(y)f_\theta(y)|^2|\widecheck{\eta}_\theta|(z_\theta-y)dy \]
for some $z_\theta\in\R^3$. Note that if either $y\not\in \frac{1}{2}R^\e T$ or $z_\theta-y\not\in \frac{1}{2}R^\e \theta^*$, then $\s_T^2(y)|\widecheck{\eta}_\theta|(z_\theta-y)\le C_\e R^{-1000}\s_T(y)$. It follows from the lower bound on $\a$ that
\[ \a\lesssim \frac{1}{A}\#\{\theta:\exists T\in\tilde{\T}_\theta\quad\text{satisfying}\quad x,z_\theta\in R^\e T\}\int_{R^\e T}|\s_T(y)f_\theta(y)|^2*|\widecheck{\eta}_\theta|(z_\theta-y)dy. \]
It remains to note that for each $\theta$ counted on the right hand side above, and for each $y\in R^\e T$, $|\widecheck{\eta}_\theta|(z_\theta-y)\lesssim R^{100\e}\w_\theta(x-y)$.

\end{proof}

\begin{lemma}\label{bet} For each $\a>0$, $B_R$, and Schwartz function $f:\R^3\to\C$ with Fourier transform supported in $\mc{M}^3(R)$, there exists $\b>0$ such that $\a^7|\{x\in B_R:\a\le|f(x)|\}|$ is bounded by 
\[C(\log R)\a^{7}|\{x\in B_R:\a\le|f(x)|,\quad\b/2\le\sum_{\theta}|f_\theta|^2*\w_\theta(x)\le \b\}|+R^{-500}\int|\sum_{\theta}|f_\theta|^2*\w_\theta|^{7/2} .\]
\end{lemma}

\begin{proof} First fix the notation 
\[ U_\a=\{x\in B_R:\a\le|f(x)|\} \]
and $g=\sum_{\theta\in{\bf{S}}(R^{-1/3})}|f_\theta|^2*\w_\theta$. Then write 
\[ 
|U_\a|=\sum_{R^{-1000}\le\lambda\le 1}|\{x\in U_\a: g(x)\sim\lambda\|g\|_{L^\infty(B_R)}\}|+|\{x\in U_\a: g(x)\le R^{-1000}\|g\|_{L^\infty(B_R)}\}|\]
where $\lambda$ takes dyadic values and $g(x)\sim\lambda\|g\|_{L^\infty(B_R)}$ means $\lambda\|g\|_{L^\infty(B_R)}/2\le g(x)\le \lambda\|g\|_{L^\infty(B_R)}$.  If one of the first $\le C(\log R)$ terms dominates the sum, then the lemma is proved. Suppose instead that the last term dominates, so that 
\[ |U_\a|\lesssim |\{x\in U_\a: g(x)\le R^{-1000}\|g\|_{L^\infty(B_R)}\}|. \]
Then 
\begin{align*}
\a^7|U_\a|&\lesssim \int_{\{x\in U_\a: g(x)\le R^{-1000}\|g\|_{L^\infty(B_R)}\}}|f|^7   \\
(\text{Cauchy-Schwarz})\qquad &\lesssim  R^2\int_{\{x\in U_\a: g(x)\le R^{-1000}\|g\|_{L^\infty(B_R)}\}}|\sum_\theta|f_\theta|^2|^{7/2}   \\
(\text{loc. const. for each $|f_\theta|^2$})\qquad &\lesssim  R^2\int_{\{x\in U_\a: g(x)\le R^{-1000}\|g\|_{L^\infty(B_R)}\}}|\sum_\theta|f_\theta|^2*\w_\theta|^{7/2}   \\
&\lesssim  R^2|B_R|R^{-1000}\|\sum_\theta|f_\theta|^2*\w_\theta\|_{L^\infty(B_R)}^{7/2}   \\
(\text{loc. const. for $\sum_\theta|f_\theta|^2$})\qquad&\lesssim  R^{-995}\|\sum_\theta|f_\theta|^2*\w_\theta*w_{R^{1/3}}\|_{L^\infty(B_R)}^{7/2}  \\
(\text{H\"{o}lder's inequality})\qquad&\lesssim  R^{-995}\||\sum_\theta|f_\theta|^2*\w_\theta|^{7/2}*w_{R^{1/3}}\|_{L^\infty(B_R)} \\
&\lesssim  R^{-995}\int|\sum_\theta|f_\theta|^2*\w_\theta|^{7/2}.
\end{align*}

\end{proof}

\subsection{A multi-scale inequality for $\S_1(R)$ implying Proposition \ref{S1bd} \label{ind} }

First we use a broad/narrow analysis to prove a multi-scale inequality for $\S_1(R)$. 

\begin{lemma}\label{multiscale'} For any $1\le K^3\le R$ and $1\le N_0\le \e^{-1}$,
\[ \emph{S}_1(R)\lesssim(\log R)^2\left( K^{53}\big[R^{10 \e N_0}A^{\e^{-1}} + R^{4\e^2+200\e}A^{\e^{-1}}  \emph{S}_1(R^\e)^{\e^{-1}-N_0}\big]+\emph{S}_1(R/K^3)\right) .\]
\end{lemma}

\begin{proof}[Proposition \ref{mainprop} implies Lemma \ref{multiscale'}]
Let $f:\R^3\to\C$ be a Schwartz function with Fourier transform supported in $\mc{M}^3(R)$. By Lemma \ref{loc}, it suffices to bound $\|f\|_{L^7(B_R)}^7$ instead of $\|f\|_{L^7(\R^3)}^7$. By Lemma \ref{alph}, we may fix $\a>0$ so that $\|f\|_{L^7(B_R)}^7\lesssim(\log R)^2\a^7|U_{\a}|$. By Proposition \ref{wpd}, we may replace $\a$ by $\a/A$ and replace $f$ by $\tilde{f}=\frac{1}{A}\sum_{\theta\in{\bf{S}}(R^{-1/3})}\sum_{T\in\tilde{\T}_\theta}\s_Tf_\theta$ where $\tilde{\T}_\theta$ satisfies the properties in that proposition. From here, we will take $f$ to mean $\tilde{f}$. By Lemma \ref{bet}, we may fix $\b>0$ so that $\a^7|U_\a|\lesssim (\log R)\a^7|U_{\a,\b}|$. Finally, 
by Corollary \ref{wpdcor}, we have $\a\lesssim R^{103}$. 

Write $f=\sum_{\tau\in{\bf{S}}(K^{-1})}f_\tau$. The broad-narrow inequality is
\begin{align}\label{brnar}
    |f(x)|&\le 6\max_{\tau\in{\bf{S}}(K^{-1})}|f_{\tau}(x)|+K^3\max_{\substack{d(\tau^i,\tau^j)\ge K^{-1}\\\tau^i\in{\bf{S}}(K^{-1})}}|f_{\tau^1}(x)f_{\tau^2}(x)f_{\tau^3}(x)|^{\frac{1}{3}} .
\end{align}
Indeed, suppose that the set $\{\tau\in{\bf{S}}(K^{-1}):|f_{\tau}(x)|\ge K^{-1}\max_{\tau'\in{\bf{S}}(K^{-1})}|f_{\tau'}(x)|\}$ has at least $5$ elements. Then we can find three $\tau^1,\tau^2,\tau^3$ which are pairwise $\ge K^{-1}$-separated and satisfy $|f(x)|\le K^3|f_{\tau^1}(x)f_{\tau^2}(x)f_{\tau^3}(x)|^{\frac{1}{3}}$. If there are fewer than $5$ elements, then $|f(x)|\le 6 \max_{\tau\in{\bf{S}}(K^{-1})}|f_{\tau}(x)|$.

The broad-narrow inequality leads to two possibilities. In one case, we have  
\begin{equation}\label{case1brn} |U_{\a,\b}|\lesssim |\{x\in U_{\a,\b}:|f(x)|\le 6\max_{\tau\in{\bf{S}}(K^{-1})}|f_{\tau}(x)|\}| . \end{equation}
Then the summary of inequalities from this case is
\begin{align*}
    \int_{B_R}|f|^7\lesssim(\log R)^2\a^7|U_{\a,\b}|\lesssim(\log R)^2\sum_{\tau\in{\bf{S}}(K^{-1})}\int_{\R^3}|f_{\tau\in{\bf{S}}(K^{-1})}|^7 .
\end{align*}
By rescaling for the moment curve and the definition of $\S_1(\cdot)$, we may bound each integral in the final upper bound by 
\[ \int_{\R^3}|f_{\tau}|^7\le \S_1(R/K^3)\int_{\R^3}|\sum_{\theta\subset\tau}|f_\theta|^2*\w_\theta|^{7/2}. \]
Noting that $\sum_\tau\int_{\R^3}|\sum_{\theta\subset\tau}|f_\theta|^2*\w_\theta|^{7/2}\le\int_{\R^3}|\sum_{\theta}|f_\theta|^2*\w_\theta|^{7/2} $ finishes this case.

The remaining case from the broad-narrow inequality is that
\[ |U_{\a,\b}|\lesssim |\{x\in U_{\a,\b}:|f(x)|\le K^3 \max_{\substack{
d(\tau^i,\tau^j)\ge K^{-1}\\\tau^i\in{\bf{S}}(K^{-1})}}|f_{\tau^1}(x)f_{\tau^2}(x)f_{\tau^3}(x)|^{1/3}\}| .\]
We may further assume that 
\[ |U_{\a,\b}|\lesssim |\{x\in U_{\a,\b}:|f(x)|\le K^3 \max_{\substack{
d(\tau^i,\tau^j)\ge K^{-1}\\\tau^i\in{\bf{S}}(K^{-1})}}|f_{\tau^1}(x)f_{\tau^2}(x)f_{\tau^3}(x)|^{1/3},\quad \max_{\tau\in{\bf{S}}(K^{-1})}|f_\tau(x)|\le \a\}| \]
since otherwise, we would be in the first case \eqref{case1brn}. The size of the set above is now bounded by a sum over pairwise $K^{-1}$-separated 3-tuples $(\tau^1,\tau^2,\tau^3)$ of $|\text{Br}_{\a,\b}^K|$ from \textsection\ref{broad}. Using Proposition \ref{mainprop} to bound $|\text{Br}_{\a,\b}^K|$, the summary of the inequalities from this case is 
\[\int_{B_R}|f|^7\lesssim(\log R)^2  K^{53}\big[R^{10 \e N_0}A^{\e^{-1}} + R^{4\e^2}A^{\e^{-1}}  \emph{S}_1(R^\e)^{\e^{-1}-N_0}\big]\int|\sum_\theta|f_\theta|^2*\w_\theta|^{7/2}, \]
which finishes the proof.

\end{proof}

With Lemma \ref{multiscale'} in hand, we may now prove Proposition \ref{S1bd}. 
\begin{proof}[Proof of Proposition \ref{S1bd}] 
Let $\eta$ be the infimum of the set
\[ \mc{S}=\{\d\ge 0:\sup_{R\ge 1}\frac{\S_1(R)}{R^\d}<\infty\}. \]
Suppose that $\eta>0$. Let $\e_1$, $\eta>\e_1>0$, be a parameter we will specify later. By Lemma \ref{multiscale'}, we have
\begin{align*} 
\sup_{R\ge 1}\frac{\S_1(R)}{R^{\eta-\e_1}}&\lesssim_\e\sup_{R\ge 1}\frac{1}{R^{\eta-\e_1}}\Big[(\log R)^2\left( K^{53}\big[R^{10 \e N_0} + R^{4\e^2+10\e}  \S_1(R^\e)^{\e^{-1}-N_0}\big]+\S_1(R/K^3)\right)  \Big] 
\end{align*}
where we are free to choose $\e>0$, $1\le N_0\le\e^{-1}$, and $1\le K^3\le R$. Continue to bound the expression on the right hand side by
\begin{align*}
\sup_{R\ge 1}(\log R)^2\Big( K^{53}\frac{R^{10 \e N_0}}{R^{\eta-\e_1}} +& K^{53}\frac{R^{4\e^2+10\e}}{R^{N_0\e(\eta+\e_1)-2\e_1}}  \big[\frac{\S_1(R^\e)}{R^{\e(\eta+\e_1)}}\big]^{\e^{-1}-N_0}+\frac{1}{K^{3(\eta+\e_1)}R^{-2\e_1}}\frac{{\S_1(R/K^3)}}{(R/K^3)^{\eta+\e_1}}  \Big).
\end{align*}
By definition of $\eta$, 
\[ \sup_{R\ge 1}\frac{\S_1(R^\e)}{R^{\e(\eta+\e_1)}}+\sup_{R\ge 1}\frac{{\S_1(R/K^3)}}{(R/K^3)^{\eta+\e_1}} <\infty, \]
so it suffices to check that 
\begin{align*}
\sup_{R\ge 1}(\log R)^2\Big( K^{53}\frac{R^{10 \e N_0}}{R^{\eta-\e_1}} +& K^{53}\frac{R^{4\e^2+10\e}}{R^{N_0\e(\eta+\e_1)-2\e_1}}  +\frac{1}{K^{3(\eta+\e_1)}R^{-2\e_1}}  \Big)<\infty
\end{align*}
to obtain a contradiction. From here, choose $N_0=\e^{-1/2}$ and $K=R^{\e_1}$ so that it suffices to check 
\begin{align*}
\sup_{R\ge 1}(\log R)^2\Big( \frac{1}{R^{\eta-10\e^{1/2}-54\e_1}} +& \frac{1}{R^{\e^{1/2}\eta-4\e^2-10\e-55\e_1}}  +\frac{1}{R^{\e_1}}  \Big)<\infty .
\end{align*}
This is clearly true if we choose $\e>0$ to satisfy $\min(\eta-10\e^{1/2},\eta-4\e^{3/2}-10\e^{1/2})>\eta/2$ and then choose $\e_1$ to be smaller than $\frac{1}{55}\e^{1/2}\eta/4$. Our reasoning has shown that $\eta-\e_1\in\mc{S}$, which is a contradiction. Conclude that $\eta=0$, as desired.

\end{proof}

\subsection{Proof of Theorem \ref{main} \label{S2}}
We will show that $\S_1(R)\lesssim_\e R^\e$ implies $\S_2(R)\lesssim_\e R^\e$, which proves Theorem \ref{main}. See Definitions \ref{S1df} and \ref{S2df} in \textsection\ref{const} for the definitions of $\S_1(R)$ and $\S_2(R)$. 

The following is a multi-scale inequality relating $\S_1(R)$ to $\S_2(\cdot)$ and $\S_1(\cdot)$ evaluated at parameters smaller than $R$. 
\begin{proposition}\label{multiscaleS2} For $R\ge 10$,  
\[\emph{S}_2(R)\lesssim R^\e \emph{S}_1(R^{1/3})\max_{1\le \lambda\le R^{2/3}}\emph{S}_2(\lambda). \]
\end{proposition}
Granting this proposition, we now prove Theorem \ref{main}. 
\begin{proof}[Proof of Theorem \ref{main}] Let $\eta$ be the infimum of the set
\[ \mc{S}=\{\d\ge 0:\sup_{R\ge 1}\frac{\S_2(R)}{R^\d}<\infty\}. \]
Suppose that $\eta>0$. Let $\e_1>0$ satisfy $(1/3)\eta-\e_1>0$. By Proposition \ref{multiscaleS2}, we have
\begin{align*} 
\sup_{R\ge 1}\frac{\S_2(R)}{R^{\eta-\e_1}}&\le C_\e\sup_{R\ge 1}\Big[\frac{\S_1(R)R^{\e}}{R^{\eta-\e_1}}\max_{1\le \lambda\le R^{2/3}}\S_2(\lambda) \Big] \\
&=C_\e\sup_{R\ge 1}\frac{\S_1(R)R^{\e}}{R^{(1/3)\eta-\e_1-(2/3)\d}}\max_{1\le \lambda\le R^{2/3}}\frac{\S_2(\lambda)}{\lambda^{\eta+\d}}\frac{\lambda^{\eta+\d}}{R^{(2/3)(\eta+\d)}}  
\end{align*}
where $\d$ is any positive number in the last line. By definition of $\eta$, 
\[ \max_{1\le \lambda}\frac{S_2(\lambda)}{\lambda^{\eta+\d}}\le C_\d<\infty. \]
Therefore, 
\begin{align*} 
\sup_{R\ge 1}\frac{\S_1(R)R^{\e}}{R^{(1/3)\eta-\e_1-(2/3)\d}}\max_{1\le \lambda\le R^{2/3}}\frac{\S_2(\lambda)}{\lambda^{\eta+\d}}\frac{\lambda^{\eta+\d}}{R^{(2/3)(\eta+\d)}}  &\lesssim_{\e,\d}\sup_{R\ge 1}\frac{\S_1(R)R^{\e}}{R^{(1/3)\eta-\e_1-(2/3)\d}}. 
\end{align*}
Choose $\d>0$ so that $(1/3)\eta-\e_1-(2/3)\d>0$. Then choose $\e>0$ so that $(1/3)\eta-\e_1-(2/3)\d>0-\e>0$. Since by Proposition \ref{S1bd}, we know that $\S_1(R)$ grows slower than any power in $R$, conclude that 
\[ \sup_{R\ge 1}\frac{\S_1(R)}{R^{(1/3)\eta-\e_1-(2/3)\d-\e}}<\infty. \]
The sequence of inequalities shows that $\eta-\e_1\in{\mc{S}}$, which contradicts the fact that $\eta$ is the infimum. Conclude that $\eta=0$. 

\end{proof}

It remains to prove Proposition \ref{multiscaleS2}, which we do presently. 
\begin{proof}[Proof of Proposition \ref{multiscaleS2}] 
Let $f:\R^3\to\C$ be a Schwartz function with Fourier transform supported in $\mc{M}^3(R)$. Begin with the defining inequality for $\S_1(R^{1/3})$:
\begin{equation}\label{S1app}
    \int_{\R^3}|f|^7\le S_1(R^{1/3})\int_{\R^3}|\sum_{\tau\in{\bf{S}}(R^{-1/9})}|f_\tau|^2*\w_{\tau}|^{7/2}. 
\end{equation}
We choose the scale $R^{1/3}$ because each $\w_{\tau}$ is localized to an $R^{1/9}\times R^{2/9}\times R^{1/3}$ plank, which is contained in an $R^{1/3}$ ball. The square function we are aiming for, $\sum_{\theta\in{\bf{S}}(R^{-1/3})}|f_\theta|^2$, is locally constant on $R^{1/3}$ balls, so we will be able to eliminate the weights and therefore obtain a bound for $\S_2(R)$. The idea for going from the right hand side of \eqref{S1app} to our desired right hand side is to perform an algorithm similar to Proposition \ref{algo}. We may vastly simplify the process since here we have not performed any pruning steps which give $k$th versions of $f$ with altered Fourier support.  

Begin with the assumption that 
\begin{equation}\label{assS2}
\int_{\R^3}|\sum_{\tau\in{\bf{S}}(R^{-1/9})}|f_\tau|^2*\w_{\tau}|^{7/2}\lesssim  \int_{\R^3}|\sum_{\tau\in{\bf{S}}(R^{-1/9})}|f_\tau|^2*\w_{\tau}*\widecheck{\eta}_{>R^{-1/3}}|^{7/2}.\end{equation}
Indeed, if this does not hold, then we may assume
\[\int_{\R^3}|\sum_{\tau\in{\bf{S}}(R^{-1/9})}|f_\tau|^2*\w_{\tau}|^{7/2}\lesssim  \int_{\R^3}|\sum_{\tau\in{\bf{S}}(R^{-1/9})}|f_\tau|^2*\w_{\tau}*\widecheck{\eta}_{\le R^{-1/3}}|^{7/2}.\]
Then we have reached our termination criterion since by local $L^2$-orthogonality, for each $\tau\in{\bf{S}}(R^{-1/9})$,
\[ ||f_\tau|^2*\w_{\tau}*\widecheck{\eta}_{\le R^{-1/3}}(x)|\lesssim \sum_{\theta\subset\tau}|f_\theta|^2*\w_\tau*|\widecheck{\eta}_{R^{-1/3}}|(x).  \]
Finally, simply note that $\w_{\tau}*|\widecheck{\eta}_{R^{-1/3}}|\lesssim w_{R^{1/3}}$ and by Young's convolution inequality, 
\[ \int_{\R^3}|\sum_\theta|f_\theta|^2*w_{R^{1/3}}|^{7/2}\lesssim \int_{\R^3}|\sum_\theta|f_\theta|^2|^{7/2}. \]
From here on, assume that \eqref{assS2} holds.

Now we describe the simplified algorithm. Let $\d>0$ be a constant that we specify later in the proof. At intermediate step $m$, we have the inequality 
\begin{align}\label{stepmS2} \text{(R.H.S. of \eqref{assS2})}
\le (C_\d R^{2\d}(\log R)^{10})^m \sum_{s_m\le \sigma\le 1}\sum_{\substack{\tau\in{\bf{S}}(\sigma^{-1}s_m)}} \sum_{V\|V_{\tau,s_m^{-3}}}|V|\Big(|V|^{-1}\int \sum_{\substack{\tau'\subset\tau \\\tau'\in{\bf{S}}(s_m)}}|f_{\tau'}|^{7/2} W_V \Big)^{2}
\end{align}
in which $R^{-\frac{1}{3}}\le s_m\le R^{-\frac{1}{9}}R^{-\frac{\e^2}{2}m}$. We begin by showing that \eqref{stepmS2} holds with $m=1$. 

\noindent\fbox{Step 1:} The Fourier support of $\sum_{\tau\in{\bf{S}}(R^{-1/9})}|f_\tau|^2*\w_\tau*\widecheck{\eta}_{>R^{-1/3}}$ is in the annulus $\{R^{-1/3}\le |\xi|\le 2R^{-1/9}\}$. Since $\sum_{\substack{R^{-1/3}\le s\le 3R^{-1/9}}}\eta_s\equiv 1$ (where $s$ takes dyadic values) on this annulus, there is some dyadic $s_1$, $R^{-1/3}\le s_1\le 3R^{-1/9}$ such that
\[ \text{(R.H.S. of \eqref{assS2})}\lesssim (\log R)^{7/2}\int_{\R^3}|\sum_{\tau\in{\bf{S}}(R^{-1/9})}|f_\tau|^2*\w_\tau*\widecheck{\eta}_{s_1}|^{7/2}\]
By pointwise local $L^2$-orthogonality (or the proof of Lemma \ref{low}), 
\[\int_{\R^3}|\sum_{\tau\in{\bf{S}}(R^{-1/9})}|f_\tau|^2*\w_\tau*\widecheck{\eta}_{s_1}|^{7/2}=\int_{\R^3}|\sum_{\tau\in{\bf{S}}(R^{-1/9})}\sum_{\substack{\tau',\tau''\in{\bf{S}}(s_1)\\\tau'\tau''\subset\tau\\\tau'\sim\tau''}}f_{\tau'}\overline{f_{\tau''}}*\w_\tau*\widecheck{\eta}_{s_1}|^{7/2}\]
where $\tau'\sim\tau''$ means $d(\tau',\tau'')\le 3 s_1$. For each $\tau\in{\bf{S}}(R^{-1/9})$ and $\tau'\subset\tau$, $\tau'\in{\bf{S}}(s_1)$, the Fourier support of 
\[ \sum_{\substack{\tau',\tau''\in{\bf{S}}(s_1)\\\tau'\tau''\subset\tau\\\tau'\sim\tau''}}f_{\tau'}\overline{f_{\tau''}}*\w_\tau*\widecheck{\eta}_{s_1} \]
is contained in $(10\tau'-10\tau')\setminus B_{s_1}$, which, after dilating by a factor of $s_1^{-1}$, we may identify with a conical cap as we did in \textsection\ref{geo}. Therefore, we may apply Proposition \ref{cone3.5} (with space rescaled by a factor of $s_1^{-1}$) to obtain 
\begin{align*}
\int_{\R^3}|\sum_{\tau\in{\bf{S}}(R^{-1/9})}&\sum_{\substack{\tau',\tau''\in{\bf{S}}(s_1)\\\tau'\tau''\subset\tau\\\tau'\sim\tau''}}f_{\tau'}\overline{f_{\tau''}}*\w_\tau*\widecheck{\eta}_{s_1}|^{7/2}\lesssim_\d R^{\d} \\
&\times \sum_{s_1\le \sigma\le 1}\sum_{\substack{\underline{\tau}\in{\bf{S}}(\sigma^{-1}s_1)}} \sum_{V\|V_{\underline{\tau},s_1^{-3}}}|V|\Big(|V|^{-1}\int\sum_{\substack{\tau'\subset\underline{\tau} \\\tau'\in{\bf{S}}(s_1)}}|\sum_{\substack{\tau',\tau''\in{\bf{S}}(s_1)\\\tau'\tau''\subset\tau\\\tau'\sim\tau''}}f_{\tau'}\overline{f_{\tau''}}*\w_\tau*\widecheck{\eta}_{s_1}|^{7/4}W_V \Big)^{2}. \end{align*}
For each $s_1\le \sigma\le 1$, $\underline{\tau}\in{\bf{S}}(\sigma^{-1}s_1)$, and $V\|V_{\underline{\tau},s_1^{-3}}$, by Cauchy-Schwarz and H\"{o}lder's inequality, 
\[\int\sum_{\substack{\tau'\subset\underline{\tau} \\\tau'\in{\bf{S}}(s_1)}}|\sum_{\substack{\tau',\tau''\in{\bf{S}}(s_1)\\\tau'\tau''\subset\tau\\\tau'\sim\tau''}}f_{\tau'}\overline{f_{\tau''}}*\w_\tau*\widecheck{\eta}_{s_1}|^{7/4}W_V \lesssim \int\Big(\sum_{\substack{\tau'\subset\underline{\tau} \\\tau'\in{\bf{S}}(s_1)}}|f_{\tau'}|^{7/2}*\w_\tau*|\widecheck{\eta}_{s_1}| \Big)W_V. \]
Note that since $\tau'\subset\tau$, we have $\tau^*\subset 2(\tau')^*$ and each $(\tau')^*\subset 2V_{\tau,s_1^{-3}}$. Also note that $B_{s_1}(0)\subset V_{\tau,s_1^{-3}}$. By properties of weight functions, it follows that 
\[ \int\Big(\sum_{\substack{\tau'\subset\underline{\tau} \\\tau'\in{\bf{S}}(s_1)}}|f_{\tau'}|^{7/2}*\w_\tau*|\widecheck{\eta}_{s_1}| \Big)W_V\lesssim \int\sum_{\substack{\tau'\subset\underline{\tau} \\\tau'\in{\bf{S}}(s_1)}}|f_{\tau'}|^{7/2}W_V.\]
This concludes Step $1$. 

From here, we assume that for $m>1$, Step $m-1$ holds. We will show that either the algorithm terminates with the desired inequality or Step $m$ (i.e. \eqref{stepmS2}) holds. 

\noindent\fbox{Step m:} Let $\sigma\in [s_{m-1},1]$ be a dyadic value satisfying
\begin{align}\label{stephypS2} 
(\text{L.H.S. of \eqref{algline1}})\le (C_\d R^{2\d}\log &R)^{10(m-1)}  \\
&\times (C\log R)\sum_{\substack{\tau\in{\bf{S}}(\sigma^{-1}s_{m-1})}} \sum_{V\|V_{\tau,s_{m-1}^{-3}}}|V|\Big(\fint_V\sum_{\substack{\tau'\subset\tau \\\tau'\in{\bf{S}}(s_{m-1})}}|f_{\tau'}^{k_{m-1}}|^{7/2} \Big)^{2}\nonumber
\end{align}
in which $R^{-\frac{1}{3}}\le s_{m-1}\le 3R^{-\frac{1}{9}}R^{-\frac{\e^2}{2}(m-1)}$. 
The analysis splits into two cases depending on whether $\sigma\ge R^{-\e^2}$ or $\sigma<R^{-\e^2}$.
\newline\noindent\fbox{Step m: $\sigma\ge R^{-\e^2}$.} Using Cauchy-Schwarz and then H\"{o}lder's inequality, for each $\tau\in{\bf{S}}(\sigma^{-1}s_{m-1})$,
\begin{align*} 
\sum_{V\|V_{\tau,s_{m-1}^{-3}}}|V|\Big(\fint_V\sum_{\substack{\tau'\subset\tau \\\tau'\in{\bf{S}}(s_{m-1})}}|f_{\tau'}|^{7/2} \Big)^{2}&\le\sum_{V\|V_{\tau,s_{m-1}^{-3}}}|V|(\#\tau'\subset\tau)\sum_{\substack{\tau'\subset\tau \\\tau'\in{\bf{S}}(s_{m-1})}}\Big(\fint_V|f_{\tau'}|^{7/2} \Big)^{2}\\
    &\lesssim\sum_{V\|V_{\tau,s_{m-1}^{-3}}}|V|(R^{\e^2})\sum_{\substack{\tau'\subset\tau \\\tau'\in{\bf{S}}(s_{m-1})}}\fint_V|f_{\tau'}|^{7} \\
    &\lesssim R^{\e^2}\sum_{\substack{\tau'\subset\tau \\\tau'\in{\bf{S}}(s_{m-1})}}\int_{\R^3}|f_{\tau'}|^7.  \end{align*}
By rescaling for the moment curve, for each $\tau'\in{\bf{S}}(s_{m-1})$,
\[ \int_{\R^3}|f_{\tau'}|^7\le \S_2(s_{m-1}^3R)\int_{\R^3}(\sum_{\theta\subset\tau'}|f_{\theta}|^2)^{7/2}. \]
Since $\sum_{\tau'\in{\bf{S}}(\sigma^{-1}s_{m-1})}\int_{\R^3}(\sum_{\theta\subset\tau'}|f_{\theta}|^2)^{7/2}\le \int_{\R^3}(\sum_{\theta}|f_{\theta}|^2)^{7/2}$, the algorithm terminates in this case with the inequality
\begin{align}\label{termS2} 
(\text{L.H.S. of \eqref{algline1}})\le (C_\d R^{2\d}\log &R)^{10m}R^{\e^2}\S_1(s_{m-1}^3R)\int_{\R^3}(\sum_{\theta\in{\bf{S}}(R^{-1/3})}|f_\theta|^2)^{7/2}.\end{align}

\noindent\fbox{Step m: $s_{m-1}\le \sigma\le R^{-\e^2}$.} Let $\tilde{s}_{m-1}=\max(R^{-\frac{1}{3}},\sigma^{\frac{1}{2}}s_{m-1})$. By Lemma \ref{algo1} (which applies equally well to functions which have not been pruned), for each $\tau\in{\bf{S}}(\sigma^{-1}s_{m-1})$ and $V\|V_{\tau,s_{m-1}^{-3}}$, 
\[\fint_V\sum_{\substack{\tau'\subset\tau \\\tau'\in{\bf{S}}(s_{m-1})}}|f_{\tau'}|^{7/2} \le C_{\d}R^{\d}|V|^{-1}\fint \sum_{\substack{\tau'\subset\tau \\\tau'\in{\bf{S}}(s_{m-1})}}|\sum_{\substack{\tau''\subset\tau'\\\tau''\in{\bf{S}}(\tilde{s}_{m-1})}}|f_{\tau''}|^2|^{7/4} W_V  .\]
Then using $\|\cdot\|_{\ell^{7/4}}\le\|\cdot\|_{\ell^{1}}$ and Cauchy-Schwarz, we have 
\begin{align*} 
\sum_{\substack{\tau\in{\bf{S}}(\sigma^{-1}s_{m-1})}} \sum_{V\|V_{\tau,s_{m-1}^{-3}}}|V|\Big(|V|^{-1}\int&\sum_{\substack{\tau'\subset\tau  \\\tau'\in{\bf{S}}(s_{m-1})}}(\sum_{\substack{\tau''\subset\tau'\\\tau''\in{\bf{S}}(\tilde{s}_{m-1})}}|f_{\tau''}|^2)^{7/4}W_V \Big)^{2}\\
&\le  \sum_{\substack{\tau\in{\bf{S}}(\sigma^{-1}s_{m-1})}} \sum_{V\|V_{\tau,s_{m-1}^{-3}}}|V|^{-1}\Big(\int(\sum_{\substack{\tau''\subset\tau\\\tau''\in{\bf{S}}(\tilde{s}_{m-1})}}|f_{\tau''}|^2)^{7/4}W_V \Big)^{2}\\
&\le \sum_{\substack{\tau\in{\bf{S}}(\sigma^{-1}s_{m-1})}} \int_{\R^3}(\sum_{\substack{\tau''\subset\tau\\\tau''\in{\bf{S}}(\tilde{s}_{m-1})}}|f_{\tau''}|^2)^{7/2} . 
\end{align*}
If $\tilde{s}_{m-1}=R^{-\frac{1}{3}}$, then the algorithm terminates with the inequality 
\begin{align}\label{term2S2} 
(\text{R.H.S. of \eqref{assS2}})\le (C_\d R^{2\d}(\log &R)^{10})^{ m} \int_{\R^3}(\sum_\theta|f_\theta|^2)^{7/2} .
\end{align}
Otherwise, suppose that $\tilde{s}_{m-1}>R^{-\frac{1}{3}}$. Then the inequality so far reads 
\begin{align*}
(\text{L.H.S. of \eqref{algline1}})\le (C_\d R^{2\d}(\log &R)^{10})^{(m-1)}(C_\d R^\d\log R)\int_{\R^3}(\sum_{\tau''\in{\bf{S}}(\sigma^{\frac{1}{2}}s_{m-1})}|f_{\tau''}|^2)^{7/4}.
\end{align*}
By the same reasoning as in Step $1$, either we are in a termination case analogous to when \eqref{assS2} does not hold, or there is some $s_m\le \sigma^{\frac{1}{2}}s_{m-1}$ for which 
\[ \int_{\R^3}(\sum_{\tau''\in{\bf{S}}(\sigma^{\frac{1}{2}}s_{m-1})}|f_{\tau''}|^2)^{7/4}\le (C_\d R^{\d} \log R)\sum_{s_m\le \sigma\le 1}\sum_{\substack{\tau\in{\bf{S}}(\sigma^{-1}s_m)}} \sum_{V\|V_{\tau,s_m^{-3}}}|V|\Big(\fint_V\sum_{\substack{\tau'\subset\tau \\\tau'\in{\bf{S}}(s_m)}}|f_{\tau'}|^{7/2} \Big)^{2}. \]
By the hypothesis that $s_{m-1}\le R^{-\frac{1}{9}}R^{\frac{-\e^2}{2}(m-1)}$ and $\sigma\le R^{-\e^2}$, we have $s_m\le R^{-\frac{1}{9}}R^{\frac{-\e^2}{2}m}$. This completes the justification of Step $m$.

\noindent\fbox{Termination criteria.} It remains to summarize the termination criteria. In each of the termination scenarios, if $\S_2(\cdot)$ appears, then it is multiplied by one factor of $R^{\e^2}$ and is evaluated at some parameter $\lambda$ between $1$ and $R^{2/3}$. Finally, consider the accumulated constant 
\[ (C_\d R^{\d}(\log R)^{10})^{M}\]
after $M$ steps of the algorithm. Since $R^{-1/3}\le R^{-1/9}R^{-\e^2 M/2}$, we have $M\le \e^{-2}/2$. Therefore, taking $\d=\e^4$ gives 
\[ (C_\d R^{\d}(\log R)^{10})^{M}\le \tilde{C}_\e R^{\e^{3/2}}, \]
which suffices to prove the proposition. 

\end{proof}

\section{Appendix A}

\subsection{An $L^{3/2}$ square function estimate for the parabola. \label{Papp}}
Let $\mb{P}^1=\{(t,t^2):0\le t\le 1\}$ and for $r\ge 1$, let $\mc{N}_{r^{-1}}(\P^1)$ denote the $r^{-1}$-neighborhood of $\P^1$ in $\R^2$. Define the collection of canonical $\sim r^{-1/2}\times r^{-1}$ parabola blocks as follows. Let $s\in 2^\Z$ be the smallest number satisfying $r^{-1}\le s^2$. Then write $\mc{N}_{s^2}(\P^1)$ as 
\begin{equation*} 
\bigsqcup_{1\le l\le s^{-1}-2} \{(\xi_1,\xi_2)\in\mc{N}_{s^2}(\P^1):ls\le \xi_1<(l+1)s \}  \end{equation*}
and the two end pieces
\[  \{(\xi_1,\xi_2)\in\mc{N}_{s^2}(\P^1):\xi_1<s\} \sqcup  \{(\xi_1,\xi_2)\in\mc{N}_{s^2}(\P^1):1-s\le \xi_1\} \]
We use the notation $\ell(\tau)=r^{-1/2}$ in two ways: (1) to describe $\tau$ as one of the sets from the above partition and (2) to index the set of $\tau$ from the above partition. 

Define the $L^{7/2}$ parabolic square function constant $\text{PS}(r)$ to be the infimum of $A>0$ which satisfy 
\[ \int_{\R^2} |h|^{7/2}\le A^{7/2}\int|\sum_{\ell(\theta)=r^{-1/2}}|h_\theta|^2|^{7/4} \]
for any Schwartz function $h:\R^2\to\C$ with Fourier transform supported in $\mc{N}_{r^{-1}}(\P^1)$. Note that $A\sim (r^{1/2})^{7/4}$ satisfies the displayed inequality for any $h$, so $\text{PS}(r)\lesssim r^{7/8}$. The next proposition says that $\text{PS}(r)$ satisfies a multiscale inequality, which follows from a broad-narrow analysis and bilinear restriction for the parabola. 

\begin{proposition}\label{multiscale} For each $1\le L^2\le r$,
\[ \text{\emph{PS}}(r)\le 2\text{\emph{PS}}(16r/L^2)+L^3 (\log r). \]
\end{proposition}

\begin{proof} Write $f=\sum_{\ell(\tau)={L^{-1}}}f_\tau$. The broad-narrow inequality is 
\begin{align}\label{brnar1}
    |f(x)|^2&\le (1+\frac{8}{L})\max_{\ell(\tau)=_{4L^{-1}}}|f_{\tau}(x)|^2+L^{3} \max_{\substack{\ell(\tau)=\ell(\tau')=L^{-1}\\ d(\tau,\tau')\ge L^{-1}}}|f_{\tau}(x)f_{\tau'}(x)|. 
\end{align}
Indeed, suppose that 
\[|f(x)|^2\ge L^{3}\max_{\substack{\ell(\tau)=\ell(\tau')=L^{-1}\\ d(\tau,\tau')\ge L^{-1}}}|f_{\tau}(x)f_{\tau'}(x)|  . \]
Let $\tau_0$, $\ell(\tau_0)=L^{-1}$ satisfy $|f_{\tau_0}(x)|=\max_{\ell(\tau)=L^{-1}}|f_{\tau}(x)|$. Let $\ell(\tau_0')=4L^{-1}$ contain $\tau_0$ and its neighbors. Then 
\begin{align*}
    |f(x)|^2&=\sum_{\tau,\tau'}f_{\tau}(x)\overline{f_{\tau'}}(x) \\
    &\le \sum_{\substack{\tau,\tau'\subset\tau_0'}}f_{\tau}(x)\overline{f_{\tau'}}(x)+4L|f_{\tau_0}(x)|\sum_{\tau\not\subset\tau_0'}|f_\tau(x)|\\
    &\le
|f_{\tau_0'}(x)|^2+\frac{4}{L}|f(x)|^2 .
\end{align*}
Then $|f(x)|^2\le (1+\frac{8}{L})|f_{\tau_0'}(x)|^2$, giving the first term from the upper bound in \eqref{brnar1} and finishing the justification for \eqref{brnar1}. 

Suppose that 
\[ \Big(\int_{\R^2}|h|^{7/2}\Big)^{2/7}\le 2\Big(\int_{\R^2}\sum_{\ell(\tau)={4L^{-1}}}|h_\tau|^{7/2}\Big)^{2/7}. \]
For each $\ell(\tau)={4L^{-1}}$, by parabolic rescaling, 
\[ \int_{\R^2}|h_\tau|^{7/2}\le \text{PS}(16r/L^2)^{7/2}\int(\sum_{\substack{\ell(\theta)={r^{-1/2}}\\\theta\subset\tau}}|h_\theta|^2)^{7/4}.\]

The alternative is that 
\[ \Big(\int_{\R^2}|h|^{7/2}\Big)^{2/7}\le 2L^{3/2}\Big(\int_{\text{Br}}|h|^{7/2}\Big)^{2/7}\]
where $\text{Br}=\{x\in\R^2:|h(x)|\le 2L^3\max_{\substack{\ell(\tau)=\ell(\tau')=L^{-1} \\d(\tau,\tau')\ge L^{-1}}}|h_{\tau}(x)h_{\tau'}(x)|\}$, which we call the broad set. Let $B_{r^{1/2}}\subset\R^2$ be a ball of radius $r^{1/2}$. Write
\[ \int_{B_{r^{1/2}}\cap \text{Br}}|h|^{7/2}=\sum_{r^{-1000}\le\lambda\le 1}\int_{B_{r^{1/2}}(\lambda)\cap\text{Br}}|h|^{7/2}+\int_{B_{r^{1/2}}(L)\cap\text{Br}} |h|^{7/2}\]
where $B_{r^{1/2}}(\lambda)=\{x\in B_{r^{1/2}}:|h(x)|\sim \lambda\|h\|_{L^\infty(B_r)}\}$ and $B_{r^{1/2}}(L)=\{x\in B_{r^{1/2}}:|h(x)|\le r^{-1000}\|h\|_{L^\infty(B_{r^{1/2}})}\}$. 
Using the locally constant property, note that 
\[\int_{B_{r^{1/2}}(L)\cap\text{Br}} |h|^{7/2}\le r^{-3500}\|h\|_{L^\infty(B_{r^{1/2}})}^{7/2}|B_{r^{1/2}}|\lesssim r^{-395}\int(\sum_{\theta}|h_\theta|^2)^{7/4}W_{B_{r^{1/2}}}. \]
Finally, suppose that $\lambda\in[r^{-1000},1]$ is a dyadic number satisfying 
\[ \int_{B_{r^{1/2}}\cap \text{Br}}|h|^{7/2}\lesssim(\log r)\int_{B_{r^{1/2}}(\lambda)\cap\text{Br}}|h|^{7/2}. \]
Then we have 
\begin{equation}\label{piga} (\lambda\|h\|_{L^\infty(B_{r^{1/2}})})^{7/2}|B_{r^{1/2}}(\lambda)\cap\text{Br}|\sim \int_{B_{r^{1/2}}(\lambda)\cap\text{Br}}|h|^{7/2} \end{equation}
and by bilinear restriction for the parabola (see for example Theorem 15 of \cite{lvlsets}), 
\begin{equation}\label{bilrest} (\lambda\|h\|_{L^\infty(B_{r^{1/2}})})^{4}|B_{r^{1/2}}(\lambda)\cap\text{Br}|\lesssim \sum_{\tau,\tau'\text{ non-adj.}}\int_{B_{r^{1/2}}(\lambda)\cap\text{Br}}|h_{\tau}h_{\tau'}|^{2}\lesssim L^2\int_{B_{r^{1/2}}}|\sum_{\theta}|h_\theta|^2*w_{r^{1/2}}|^2. \end{equation}
If $\lambda\|h\|_{L^\infty(B_{r^{1/2}})}\le \|\sum_\theta|h_\theta|^2*w_{r^{1/2}}\|_{L^\infty(B_{r^{1/2}})}^{1/2}$, then from \eqref{piga},  
\[  \int_{B_{r^{1/2}}(\lambda)\cap\text{Br}}|h|^{7/2} \lesssim |B_{r^{1/2}}|\|\sum_\theta|h_\theta|^2*w_{r^{1/2}}\|_{L^\infty(B_{r^{1/2}})}^{7/4}\lesssim \int(\sum_\theta|h_\theta|^2)^{7/4}W_{B_{r^{1/2}}}.  \]
If $\lambda\|h\|_{L^\infty(B_{r^{1/2}})}\le \|\sum_\theta|h_\theta|^2*w_{r^{1/2}}\|_{L^\infty(B_{r^{1/2}})}^{1/2}$, then using \eqref{bilrest}, we have
\begin{align*} 
(\lambda\|h\|_{L^\infty(B_{r^{1/2}})})^{7/2}|B_{r^{1/2}}(\lambda)\cap\text{Br}|&\lesssim \sum_{\tau,\tau'\text{ non-adj.}}\int_{B_{r^{1/2}}(\lambda)\cap\text{Br}}|h_{\tau}h_{\tau'}|^{2}\\
&\lesssim L^2|B_{r^{1/2}}|\|\sum_{\theta}|h_\theta|^2*w_{r^{1/2}}\|_{L^\infty(B_{r^{1/2}})}^{7/4}\lesssim L^2\int(\sum_\theta|h_\theta|^2)^{7/4}W_{B_{r^{1/2}}}.
\end{align*}
Combined with \eqref{piga}, we have shown that 
\[ \int_{B_{r^{1/2}}(\lambda)\cap\text{Br}}|h|^{7/2}\lesssim L^2\int(\sum_\theta|h_\theta|^2)^{7/4}W_{B_{r^{1/2}}}. \]
Summing over $B_{r^{1/2}}$ in a finitely overlapping cover of $\R^2$ finishes the argument.

\end{proof}

\begin{corollary}\label{PS} For each $r\ge 1$, 
\[  \text{\emph{PS}}(r)\lesssim_\d r^\d. \]
\end{corollary}
\begin{proof} Let $L$ be a constant we allow to depend on $\d>0$. If $1\le r\le L^2$, then by Cauchy-Schwarz, 
\[ \int_{\R^2}|h|^{7/2}\lesssim (r^{1/2})^{7/4}\int_{\R^2}(\sum_{\ell(\theta)={r^{-1/2}}}|h_\theta|^2)^{7/4}\lesssim_\d \int_{\R^2}(\sum_{\ell(\theta)={r^{-1/2}}}|h_\theta|^2)^{7/4}, \]
which means that $\text{PS}(r)\le C_\d$ for all $r\le L^2$ (where $C_\d$ is permitted to depend on $\d$ and $L$).  Now suppose that $r>L^2$. Let $m$ satisfy
\[ 1\le 16^mr/L^{2m}\le L^2\le 16^{m-1}r/L^{2(m-1)}. \]
Then by iterating Proposition \ref{multiscale}, we have
\begin{align*}
\text{{PS}}(r)&\le 2\text{{PS}}(16r/L^2)+L^3 (\log r)\\
&\le 2^m\text{{PS}}(16^mr/L^{2m})+2^mL^3\log(16^mr/L^{2m})+
\cdots +2L^3\log (16r/L^2)+L^3 \log r. 
\end{align*}
Since $16^mr/L^{2m}\le L$, we have $\text{PS}(16^mr/L^{2m})\le C_\d$. By definition, $m$ satisfies $m\le c\log r/\log L$. Thus the above inequality implies
\[ \text{PS}(r)\le C_\d r^{\frac{c\log 2}{\log L}}+cr^{\frac{c\log 2}{\log L}}(\log r)^2, \]
which finishes the proof by taking $L$ large enough depending on $\d$.

\end{proof}

\subsubsection{Proofs of Theorem \ref{cylPS} and Corollary \ref{Pcor} \label{cylPSsec}}
Recall the statement of Theorem \ref{cylPS}. 
\begin{theorem*}[Cylindrical $L^{3.5}$ square function estimate over $\P^1$] 
Let $\mb{P}^1=\{(t,t^2):0\le t\le 1\}$ and for $r\ge1$, let $\mc{N}_{r^{-1}}(\P^1)$ denote the $r^{-1}$-neighborhood of $\mc{P}^1$ in $\R^2$. If $h:\R^3\to\C$ is a Schwartz function with Fourier transform supported in $\mc{N}_{r^{-1}}(\P^1)\times\R$, then
\[ \int_{\R^3}|h|^{7/2}\lesssim_\e r^{\e}\int_{\R^3}(\sum_{\z}|h_\z|^2)^{7/4} \]
where the $\z$ are products of approximate rectangles $\theta$, $\ell(\theta)={r^{-1/2}}$, with $\R$.
\end{theorem*}

\begin{proof} Begin by using Fourier inversion to write
\[ h(x',x_3)=\int_{\mc{N}_\d(\P^1)}\int_{\R}\widehat{h}(\xi',\xi_3)e^{2\pi i \xi\cdot x'}e^{2\pi i \xi_3x_3} d\xi_3d\xi' . \]
For each $x_3$, the function $x'\mapsto \int_{\mc{N}_{r^{-1}}(\P^1)}\int_\R\widehat{h}(\xi',\xi_3)e^{2\pi i \xi_3x_3} d\xi_3e^{2\pi i \xi\cdot x'}d\xi'$ satisfies the hypotheses of the decoupling theorem for $\P^1$. Use Fubini's theorem to apply Corollary \ref{PS} to the inner integral
\[  \int_{\R}\int_{\R^2}|h(x',x_3)|^{7/2}dx'dx_3\lesssim_\e r^\e\int_{\R}\int_{\R^2}\big(\sum_\theta\big|\int_{\theta}\int_{\R}\widehat{h}(\xi',\xi_3)e^{2\pi i \xi\cdot x'}e^{2\pi i \xi_3x_3} d\xi_3d\xi' \big|^2\big)^{7/4}dx'  dx_3\]
where the sum in $\theta$ is over $\ell(\theta)={r^{-1/2}}$. The sets $\theta\times\R$ are the $\z$ in the statement of the lemma, so we are done.  

\end{proof}

Recall the statement of Corollary \ref{Pcor}. 
\begin{cor*} Let $B_r$ be an $r$-ball in $\R^3$. If $h:\R^3\to\C$ is a Schwartz function with Fourier transform supported in $\mc{N}_{r^{-1}}(\P^1)\times\R$, then
\[ \int_{\R^3}|h|^{7/2}W_{B_r}\lesssim_\e r^{\e}\int_{\R^3}(\sum_{\z}|h_\z|^2)^{7/4}W_{B_r} \]
where the $\z$ are products of approximate rectangles $\theta$, $\ell(\theta)={r^{-1/2}}$, with $\R$.
\end{cor*}
\begin{proof} Write
\[\int_{\R^3}|h|^{7/2}W_{B_r}\lesssim \int_{B_r}|h|^{7/2}+\sum_{k=1}^\infty\int_{2^kB_{r}\setminus 2^{k-1}B_r}|h|^{7/2}W_{B_r}.  \]
Let $\rho_r:\R^2\to[0,\infty)$ be a compactly supported smooth function supported in the ball of radius $r$ centered at the origin. Define $\phi_k:\R^2\to[0,\infty)$ by $\phi_k=1_{A_k}*|\widecheck{\rho_r}|^2$ where $A_k=2^kB_r\setminus 2^{k-1}B_r$ and let $\phi_0=1_{B_r}*|\widecheck{\rho_r}|^2$. Then the right hand side in the previous displayed math is bounded by 
\[ \int_{\R^3}|h\p_0|^{7/2}+\sum_{k=1}^\infty\frac{1}{2^{10k}}\int_{\R^3}|h\p_k|^{7/2} . \]
We may partition the $\z$ into $O(1)$ many collections $\Xi_i$ so that the Fourier support of $h_\zeta\p_k$ is pairwise disjoint for all $\z\in\Xi_i$. Finally, by the triangle inequality, for some $i$, the previous displayed line is bounded by a constant multiple of  
\[  \int_{\R^3}|\sum_{\z\in{\Xi}_i}h_\z\p_0|^{7/2}+\sum_{k=1}^\infty\frac{1}{2^{10k}}\int_{\R^3}|\sum_{\z\in{\Xi}_i}h_\z\p_k|^{7/2}, \]
which, by Theorem \ref{cylPS}, is bounded by 
\[ \lesssim_\e r^\d \big[\int_{\R^3}(\sum_{\z\in{\Xi}_i}|h_\z\p_0|^2)^{7/4}+\sum_{k=1}^\infty\frac{1}{2^{10k}}\int_{\R^3}(\sum_{\z\in{\Xi}_i}|h_\z\p_k|^2)^{7/4}\Big]. \]
Finally, observe that $\sum_{k=0}^\infty\frac{1}{2^{10k}}\p_k^{7/2}\lesssim W_{B_r}$.

\end{proof}

\subsection{An $\ell^{7/4}$-estimate for the cone. \label{coneappsec}}

Let $\Gamma=\{(\xi_1,\xi_2,\xi_3)\in\R^3:\xi_1^2+\xi_2^2=\xi_3^2,\quad\frac{1}{2}\le\xi_3\le 1\}$ be the truncated cone. In this section, let ${\bf{S}}_{r^{-1/2}}$ denote the collection of $1\times r^{-1/2}\times r^{-1}$ blocks which tile $\mc{N}_{r^{-1}}(\Gamma)$, as defined in \textsection5 of \cite{ampdep}. 

Recall the statement of Proposition \ref{cone3.5}. 
\begin{proposition*} For any Schwartz function $h:\R^3\to\C$ with $\widehat{h}$ supported in $\mc{N}_{r^{-1}}(\Gamma)$, we have 
\[ \int_{\R^3}|h|^{7/2}\lesssim_\e r^\e\sum_{r^{-1/2}\le \sigma\le 1}\sum_{\tau\in{\bf{S}}_{\sigma^{-1}r^{-1/2}}}\sum_{U\|U_{\tau_k,r}}|U|\Big(|U|^{-1}\int\sum_{\substack{\theta\subset\tau\\\theta\in{\bf{S}}_{r^{-1/2}}}}|h_\theta|^{7/4}W_U\Big)^{2} .\]
\end{proposition*}

It suffices to prove a local version of Proposition \ref{cone3.5}. 
\begin{lemma}\label{locC} Suppose that for any $r$-ball $B_r\subset\R^3$, 
\[ \int_{B_r}|h|^{7/2}\lesssim_\e r^\e\sum_{r^{-1/2}\le \sigma\le 1}\sum_{\tau\in{\bf{S}}_{\sigma^{-1}r^{-1/2}}}\sum_{U\|U_{\tau_k,r}}|U|\Big(|U|^{-1}\int\sum_{\substack{\theta\subset\tau\\\theta\in{\bf{S}}_{r^{-1/2}}}}|h_\theta|^{7/4}W_U\Big)^{2} \]
for any Schwartz function $h:\R^3\to\C$ with Fourier transform supported in $\mc{N}_{r^{-1}}(\Gamma)$.  Then Proposition \ref{cone3.5} is true. 
\end{lemma}
\begin{proof} Write 
\[ \int_{\R^3}|h|^{7/2}\lesssim \sum_{B_r}\int_{B_r}|h|^{7/2}
\]
where the sum is over a finitely overlapping cover of $\R^3$ by $r$-balls. Let $\phi_{B_r}$ be a weight function decaying by order $100$ away from $B_r$, satisfying $\phi_{B_r}\gtrsim 1$ on $B_r$, and with Fourier transform supported in an $R^{-1}$ neighborhood of the origin. For each $\theta\in{\bf{S}}_{r^{-1/2}}$, the Fourier support of $h_\theta\phi_{B_r}$ is contained in a $1\times 2r^{-1/2}\times 2r^{-1}$ conical block. By the triangle inequality, there is a subset $\mc{S}$ of the cone blocks $\theta$ so that for each $\theta\in\mc{S}$, the Fourier support of $f_{\theta}\phi_{B_r}$ is contained in a unique cone block and  
\[ \int_{B_r}|h|^{7/2}\lesssim \int_{B_r}|\sum_{\theta\in\mc{S}}h_\theta\phi_{B_r}|^{7/2}. \]
Then by applying the hypothesized local version of Proposition \ref{cone3.5},  
\[ \int_{B_r}|\sum_{\theta\in\mc{S}}h_\theta\phi_{B_r}|^{7/2}\lesssim_\e r^\e\sum_{r^{-1/2}\le \sigma\le 1}\sum_{\tau\in{\bf{S}}_{\sigma^{-1}r^{-1/2}}}\sum_{U\|U_{\tau_k,r}}|U|\Big(|U|^{-1}\int\sum_{\substack{\theta\subset\tau\\\theta\in\mc{S}}}|h_\theta\phi_{B_r}|^{7/4}W_U\Big)^{2}. \] 
It remains to note that 
\begin{align*}
\sum_{B_r}\Big(|U|^{-1}\int\sum_{\substack{\theta\subset\tau\\\theta\in\mc{S}}}|h_\theta\phi_{B_r}|^{7/4}W_U\Big)^{2}&\lesssim \Big(|U|^{-1}\int\sum_{B_r}\sum_{\substack{\theta\subset\tau\\\theta\in\mc{S}}}|h_\theta\phi_{B_r}|^{7/4}W_U\Big)^{2} \\
&\lesssim \Big(|U|^{-1}\int\sum_{\substack{\theta\subset\tau\\\theta\in{\bf{S}}_{r^{-1/2}}}}|h_\theta|^{7/4}W_U\Big)^{2}.
\end{align*}

\end{proof}

It suffices to prove a weak, level-set version of Proposition \ref{cone3.5}. 
\begin{lemma}\label{alphC} For each $B_r$ and Schwartz function $h:\R^3\to\C$ with Fourier transform supported in $\mc{N}_{r^{-1}}(\Gamma)$, there exists $\a>0$ such that
\[ \int_{B_r}|h|^{7/2}\lesssim (\log R)\a^{7/2}|\{x\in B_r:\a\le |h(x)|\}|+r^{-500}\sum_{\theta\in{\bf{S}}_{r^{-1/2}}}\sum_{U\|U_{\theta,r}}|U|\Big(|U|^{-1}\int|h_\theta|^{7/4}W_U\Big)^2. \]
\end{lemma}
\begin{proof} By an analogous proof as for Lemma \ref{alphC} for the cone, the argument reduces to bounding
\[\int_{\{x\in B_r:|h(x)|\le r^{-1000}\|f\|_{L^\infty(B_r)}\}}|h|^{7/2}. \]
By H\"{o}lder's inequality, we have
\[ \int_{B_r}|h|^{7/2}\lesssim r^3r^{-1000(7/2)+(1/2)(5/2)}\|\sum_{\theta\in{\bf{S}}_{r^{-1/2}}}|h_\theta|^{7/2}\|_{L^\infty(B_r)}. \]
Finally, by the locally constant property and H\"{o}lder's inequality, 
\begin{align*} \|\sum_{\theta\in{\bf{S}}_{r^{-1/2}}}|h_\theta|^{7/2}\|_{L^\infty(B_r)}&\lesssim \|\sum_{\theta\in{\bf{S}}_{r^{-1/2}}}(|\theta^*|^{-1}|h_\theta|^{7/4}*W_{\theta^*})^2*W_r\|_{L^\infty(B_r)}\\
&\lesssim  \sum_{\theta\in{\bf{S}}_{r^{-1/2}}}\int_{\R^3}\Big(|\theta^*|^{-1}|h_\theta|^{7/4}*W_{\theta^*}\Big)^2\\
&\lesssim  \sum_{\theta\in{\bf{S}}_{r^{-1/2}}}\sum_{U\|U_{\theta,r}}\int_{U}\Big(|\theta^*|^{-1}|h_\theta|^{7/4}*W_{\theta^*}\Big)^2\\
&\lesssim  \sum_{\theta\in{\bf{S}}_{r^{-1/2}}}\sum_{U\|U_{\theta,r}}|U|\Big(\||\theta^*|^{-1}|h_\theta|^{7/4}*W_{\theta^*}\|_{L^\infty(U)}\Big)^2\\
&\lesssim  \sum_{\theta\in{\bf{S}}_{r^{-1/2}}}\sum_{U\|U_{\theta,r}}|U|\Big(|U|^{-1}\int |h_\theta|^{7/4}W_{U}\Big)^2 \end{align*}
where we note that for $U\|U_{\theta,r}$, $|U|=|\theta^*|$. 
\end{proof}

Use the notation 
\[ U_{\underline{\a}}=\{x\in B_r:\underline{\a}\le |h(x)|\}. \]
We will show that to estimate the size of $U_{\underline{\a}}$, it suffices to replace $h$ with a version whose wave packets have been pigeonholed. Write 
\begin{align}\label{sumC} h=\sum_{\theta\in{\bf{S}}_{r^{-1/2}}}\sum_{T\in\T_\theta}\s_Th_\theta \end{align}
where for each $\theta$, $\{\s_T\}_{T\in\T_\theta}$ is analogous to the partition of unity from \textsection\ref{prusec}, but adapted to $\theta^*$, the dual sets of conical blocks $\theta$. If $\underline{\a}\le C_\e (\log r)r^{-500}\max_\theta\|h_\theta\|_{\infty}$, then by an analogous argument as dealing with the low integral over $\{x:|h(x)|\le r^{-1000}\|h\|_{L^\infty(B_r)}\}$ in the proof of Lemma \ref{alphC}, bounding $\underline{\a}^{7/2}|U_{\underline{\a}}|$ by the right hand side of Proposition \ref{cone3.5} is trivial. Let $\phi_{B_r}$ be the weight function from Lemma \ref{locC}. 

\begin{proposition}[Wave packet decomposition] \label{wpdC} Let $\underline{\a}>C_\e (\log r) r^{-100}\max_\theta\|h_\theta\|_{L^\infty(\R^3)}$. There exist subsets $\mc{S}\subset{\bf{S}}_{r^{-1/2}}$ and  $\tilde{\T}_\theta\subset\T_\theta$, as well as a constant $A>0$ with the following properties:
\begin{align} 
|U_{\underline{\a}}|\lesssim (\log r)|\{x\in U_{\underline{\a}}:\,\,{\underline{\a}}&\lesssim |\sum_{\theta\in\mc{S}}\sum_{T\in\tilde{\T}_\theta}\s_T(x)\phi_{B_r}(x)h_\theta (x)|\,\,\}|, \\
r^\e T\cap U_{\underline{\a}}\not=\emptyset\qquad&\text{for all}\quad\theta\in\mc{S},\quad T\in\tilde{\T}_\theta\\
A\lesssim \|\sum_{T\in\tilde{\T}_\theta}\s_T\phi_{B_r}h_\theta\|_{L^\infty(\R^3)}&\lesssim r^{3\e} A\qquad\text{for all}\quad  \theta\in\mc{S}\label{prop}\\
\|\s_T\phi_{B_r}h_\theta\|_{L^\infty(\R^3)}&\sim A\qquad\text{for all}\quad  \theta\in\mc{S},\quad T\in\tilde{\T}_\theta. \label{prop'}
\end{align}
\end{proposition}

The justification for this proposition follows an analogous argument as for Proposition \ref{wpd}, so we omit the proof.

\begin{proof}[Proof of Proposition \ref{cone3.5}] By Lemma \ref{locC}, Lemma \ref{alphC}, and Proposition \ref{wpdC}, it suffices to show that for some $\underline{\a}\ge C_\e r^{-100}\max_{\theta\in{\bf{S}}_{r^{-1/2}}}\|h_\theta\|_{L^\infty(\R^3)}$,
\begin{equation} \label{goal}
    \underline{\a}^{7/2}|U_{\underline{\a}}'|\lesssim_\e r^\e \sum_{r^{-1/2}\le \sigma\le 1}\sum_{\tau\in{\bf{S}}_{\sigma^{-1}r^{-1/2}}}\sum_{U\|U_{\tau_k,r}}|U|\Big(|U|^{-1}\int\sum_{\substack{\theta\subset\tau\\\theta\in{\bf{S}}_{r^{-1/2}}}}|h_\theta|^{7/4}W_U\Big)^{2}
\end{equation}
where 
\[ U_{\underline{\a}}'=\{x\in U_{\underline{\a}}:\,\,\underline{\a}\lesssim (\log r)^2|\sum_{\theta\in\mc{S}} \sum_{T\in\tilde{T}_{\theta}}\s_T(x)\phi_{B_r}(x)h_\theta(x)|\,\,\} \]
and $\mc{S}\subset{\bf{S}}_{r^{-1/2}}$ and $\tilde{\T}_{\theta}$ are from the proof of Proposition \ref{wpdC}. The wave envelope estimate (Theorem 1.3 from \cite{locsmooth}) gives the inequality
\[ \underline{\a}^{4}|U_{\underline{\a}}'|\lesssim_\e r^\e \sum_{r^{-1/2}\le \sigma\le 1}\sum_{\tau\in{\bf{S}}_{\sigma^{-1}r^{-1/2}}}\sum_{U\|U_{\tau_k,r}}|U|\Big(|U|^{-1}\int\sum_{\substack{\theta\subset\tau\\\theta\in\mc{S}}}|\sum_{T\in\T_{\theta,\lambda}^c}\s_T(x)\phi_{B_r}(x)h_\theta|^{2}W_U\Big)^{2}. \]
The inequality \eqref{goal} follows from the above inequality if 
\begin{equation}\label{iff} \max_{\theta\in\mc{S}}\|\sum_{T\in\T_{\theta,\lambda}^c}\s_T(x)\phi_{B_r}(x)h_\theta\|_{L^\infty(\R^3)}\le \underline{\a}. \end{equation}
Suppose that this is not the case. By Proposition \ref{wpdC} and the assumption that $\underline{\a}\ge C_\e r^{-100} \max_\theta\|h_\theta\|_{L^\infty(\R^3)}$ is bounded below, we have the set inclusion $U_{\underline{\a}}'\subset\cup_{\theta\in\mc{S}}\cup_{T\in\tilde{\T}_\theta}r^\e T$. Using this with the assumption that \eqref{iff} does not hold, we have
\begin{align*}
    \underline{\a}^{7/2}|U_{\underline{\a}}'|\le\sum_{\theta\in\mc{S}}\sum_{T\in\tilde{\T}_\theta}r^{20\e} |T|\|\s_T\phi_{B_r}h_\theta\|_{L^\infty(\R^3)}^{7/2}
\end{align*}
where we also used \eqref{prop} and \eqref{prop'} from Proposition \ref{wpdC}. Then using the locally constant property and H\"{o}lder's inequality, we have
\begin{align*}
    \|\s_T\phi_{B_r}h_\theta\|_{L^\infty(\R^3)}^{7/2}&\lesssim \||\theta^*|^{-1}|\s_T\phi_{B_r}h_\theta|*W_{\theta^*}\|_{L^\infty(\R^3)}^{7/2}\\
    &\lesssim \||\theta^*|^{-1}|\s_T\phi_{B_r}h_\theta|^{7/4}*W_{\theta^*}\|_{L^\infty(\R^3)}^2\\
    &\lesssim \Big(|\theta^*|^{-1}\int |h_\theta|^{7/4} W_T\Big)^2. 
\end{align*}
We are done after noting that $|T|=|\theta^*|$ and $T\in\tilde{\T}_\theta$ is a subset of $U\|U_{\theta,r}$.

\end{proof}

\section{Appendix B}

We sketch the adaptation of Theorem \ref{main} to general curves $\g$ with torsion. Let $\g:[0,1]\to\R^3$ be a $C^4$ curve satisfying \begin{equation}\label{torsion} \det\big[\g'(t)\,\g''(t)\, \g'''(t)\big]\not=0\qquad\text{for all}\quad t\in[0,1]. \end{equation}
For each $R\in 8^\N$, define the anisotropic neighborhood 
\[ \mc{M}^3_\g(R)=\{\g(t)+B\g''(t)+C\g'''(t):t\in[0,1],\quad|B|\le R^{-2/3},\quad|C|\le R^{-1}\}  . \]
Partition the neighborhood into blocks $\theta\in{\bf{S}}_\g(R^{-1/3})$ of the form 
\[  \theta= \{\g(t)+B\g''(t)+C\g'''(t):lR^{-1/3}\le t<(l+1)R^{-1/3},\quad|B|\le R^{-2/3},\quad|C|\le R^{-1}\} \]
where $B$ and $C$ are $\R$-valued parameters and $l\in\{0,\ldots,R^{1/3}-1\}$. The general version of Theorem \ref{main} is the following. 
\begin{thm} \label{maingen} Let $\g$ satisfy \eqref{torsion}. For any $\e>0$, there exists $C_\e<\infty$ such that
\begin{equation}\label{sqfngen}    \int_{\R^3}|f|^7\le C_\e R^\e\int_{\R^3}|\sum_\theta|f_\theta|^2|^{\frac{7}{2}}   \end{equation}
For any Schwartz function $f:\R^3\to\C$ with Fourier transform supported in $\mc{M}_\g^3(R)$.
\end{thm}
We note that the condition that $\g$ is $C^4$ may be relaxed to $C^3$ if we describe how to carry out each step of the proof of Theorem \ref{main} adapted to $\g$. By assuming that $\g$ is $C^4$, we may use a vastly simpler argument involving Taylor approximation and resembling the iteration in the proof of Proposition \ref{multiscaleS2}. The constant $C_\e$ in \eqref{sqfngen} is permitted to depend on $\e$ and on $\g$. To prove Theorem \ref{maingen}, we will adapt the proof of Theorem \ref{main} to hold for a specialized class $\mc{C}$ of curves $\g$ which is closed under certain affine rescalings. The $C_\e$ constant will be uniform for $\g$ in the class $\mc{C}$. 

Let $\frac{1}{2}\ge a>0$ and $\frac{a}{4}\ge \nu>0$. Define the class ${\mc{C}}$ to be the collection of $C^3$ curves $\g:[0,1]\to\R^3$ satisfying 
\begin{align}
    \g(t)&=(t,\g_2(t),\g_3(t)),\\
    \g_2''(t)&\ge a\quad\text{and}\quad \g_3'''(t)\ge a\qquad\text{for all}\quad t\in[0,1],\\
    \|\g_2\|_{C^4}&\le 1\quad\text{and}\quad\|\g_3\|_{C^4}\le 1,\\
    \det\big[\g'(t)\,\g''(t)\, \g'''(t)\big]&\ge \nu\qquad\text{for all}\quad t\in[0,1]. 
\end{align}
By the inverse function theorem and by possible interchanging the roles of $\g_2$ and $\g_3$, it is always possible to divide a general $\g$ satisfying \eqref{torsion} into sub-pieces $\left.\g\right|_{[c,c+\e_0]}$ (for some $\e_0>0$ depending on $\g$), each of which may be reparameterized so that it is contained in $\mc{C}$ for some $\nu,a$. Then \eqref{sqfngen} will hold with a factor of $\e_0^{-1}$ times the maximum of the $C_\e$ which work for each sub-piece of $\g$. 

For each $R\ge 1$, let $S(R)$ denote the smallest constant so that 
\[ \int_{\R^3}|\sum_{\theta\in{\bf{S}}_\g(R^{-1/3})} f_\theta|^7\le S(R)\int_{\R^3}|\sum_{\theta\in{\bf{S}}_\g(R^{-1/3})}|f_\theta|^2|^{7/2}\]
where $f:\R^3\to\C$ is a Schwartz function with $\supp\widehat{f}\subset\mc{M}^3_\g(R)$ and $\g$ is any curve in $\mc{C}$. 

\begin{proof}[Proof of Theorem \ref{maingen}]
Our goal is to show that $S(R)\lesssim_\e R^\e$. By possibly replacing $S(R)$ by $\sup_{1\le r\le R}S(r)$, it is no loss of generality to assume that $S(R)$ is a nondecreasing function of $R$. Let $\e>0$. Let $\g\in\mc{C}$ and let $f:\R^3\to\C$ be a Schwartz function with Fourier transform supported in $\mc{M}^3_\g(R)$. Write $K>0$ for a parameter we will choose to be a small power of $R$ later in the proof. We will show the multiscale inequality that for any $\d_1,\d_2>0$,
\[ S(R)\le S(R/K)C_{\d_1,\d_2}R^{\d_1}\Big[S(K^{1-\e})+R^{\d_2}\Big]. \]
The bound $S(R)\lesssim_\e R^\e$ then follows via an analogous argument as the proof of Theorem \ref{main} from Proposition \ref{multiscaleS2}. To prove the multiscale inequality, we use a simpler version of the argument from Proposition \ref{multiscaleS2}, highlighting steps which are different. Begin with the defining inequality for $S(R/K)$:
\begin{equation}\label{genS1app}
    \int_{\R^3}|f|^7\le S(R/K)\int_{\R^3}|\sum_{\tau\in{\bf{S}}_\g((R/K)^{-1/3})}|f_\tau|^2|^{7/2}. 
\end{equation}
The square function we are aiming for is $\sum_{\theta\in{\bf{S}}_\g(R^{-1/3})}|f_\theta|^2$. Assume that 
\begin{equation}\label{genassS2}
\int_{\R^3}|\sum_{\tau\in{\bf{S}}_\g((R/K)^{-1/3})}|f_\tau|^2|^{7/2}\lesssim  \int_{\R^3}|\sum_{\tau\in{\bf{S}}_\g((R/K)^{-1/3})}|f_\tau|^2*\widecheck{\eta}_{>R^{-1/3}}|^{7/2}.\end{equation}
As in the proof of Proposition \ref{multiscaleS2}, if this does not hold, then 
\[\int_{\R^3}|\sum_{\tau\in{\bf{S}}_\g((R/K)^{-1/3})}|f_\tau|^2|^{7/2}\lesssim  \int_{\R^3}|\sum_{\tau\in{\bf{S}}_\g((R/K)^{-1/3})}|f_\tau|^2*\widecheck{\eta}_{\le R^{-1/3}}|^{7/2}\lesssim \int_{\R^3}|\sum_{\theta\in{\bf{S}}_\g(R^{-1/3})}|f_\theta|^2|^{7/2}\]
and we have shown that $S(R)\lesssim S(R/K)$ in this case.

Next we describe an iterative procedure. The initial step is special, so we describe the first two steps and, if the process does not already terminate, iterate the second step. The Fourier support of $\sum_{\tau\in{\bf{S}}_\g((R/K)^{-1/3})}|f_\tau|^2 *\widecheck{\eta}_{>R^{-1/3}}$ is in an annulus $\{R^{1/3}\lesssim  |\xi|\lesssim (R/K)^{-1/3}\}$. Let $s_1$ be a dyadic value in the range $R^{-1/3}\lesssim s_1\lesssim (R/K)^{-1/3}$ satisfying
\[ \text{(R.H.S. of \eqref{genassS2})}\lesssim (\log R)^{7/2}\int_{\R^3}|\sum_{\tau\in{\bf{S}}_\g((R/K)^{-1/3})}|f_\tau|^2 *\widecheck{\eta}_{s_1}|^{7/2}\]
Then by pointwise local $L^2$-orthogonality, 
\[\int_{\R^3}|\sum_{\tau\in{\bf{S}}_\g((R/K)^{-1/3})}|f_\tau|^2 *\widecheck{\eta}_{s_1}|^{7/2}=\int_{\R^3}|\sum_{\tau\in{\bf{S}}_\g((R/K)^{-1/3})}\sum_{\substack{\tau',\tau''\in{\bf{S}}_\g(s_1)\\\tau'\tau''\subset\tau\\\tau'\sim\tau''}}f_{\tau'}\overline{f_{\tau''}} *\widecheck{\eta}_{s_1}|^{7/2}\]
where $\tau'\sim\tau''$ means $d(\tau',\tau'')\lesssim s_1$. For each $\tau\in{\bf{S}}_\g((R/K)^{-1/3})$ and $\tau'\subset\tau$, $\tau'\in{\bf{S}}_\g(s_1)$, the Fourier support of 
\[ \sum_{\substack{\tau',\tau''\in{\bf{S}}_\g(s_1)\\\tau'\tau''\subset\tau\\\tau'\sim\tau''}}f_{\tau'}\overline{f_{\tau''}} *\widecheck{\eta}_{s_1} \]
is contained in $(10\tau'-10\tau')\setminus B_{s_1}$, which, after dilating by a factor of $s_1^{-1}$, we may identify with a conical cap as we did in \textsection\ref{geo}. Therefore, we may apply a version of Proposition \ref{cone3.5} uniform for $\g\in\mc{C}$ (see \textsection\ref{conegensec} for the adaptation to general $\g$) to obtain 
\begin{align}
\int_{\R^3}|\sum_{\tau\in{\bf{S}}_\g((R/K)^{-1/3})}&\sum_{\substack{\tau',\tau''\in{\bf{S}}_\g(s_1)\\\tau'\tau''\subset\tau\\\tau'\sim\tau''}}f_{\tau'}\overline{f_{\tau''}} *\widecheck{\eta}_{s_1}|^{7/2}\lesssim_\d R^{\d} \nonumber \\
&\times \sum_{s_1\le \sigma\le 1}\sum_{\substack{\underline{\tau}\in{\bf{S}}_\g(\sigma^{-1}s_1)}} \sum_{V\|V_{\underline{\tau},s_1^{-3}}^\g}|V|\Big(\fint_V\sum_{\substack{\tau'\subset\underline{\tau} \\\tau'\in{\bf{S}}_\g(s_1)}}|\sum_{\substack{\tau',\tau''\in{\bf{S}}_\g(s_1)\\\tau'\tau''\subset\tau\\\tau'\sim\tau''}}f_{\tau'}\overline{f_{\tau''}} *\widecheck{\eta}_{s_1}|^{7/4} \Big)^{2}. \label{coneapp}\end{align}
For each $s_1\le \sigma\le 1$, $\underline{\tau}\in{\bf{S}}_\g(\sigma^{-1}s_1)$, and $V\|V_{\underline{\tau},s_1^{-3}}^\g$, by Cauchy-Schwarz, H\"{o}lder's inequality, and properties of weight functions (using that $B_{s_1}(0)\subset V_{\tau,s_1^{-3}}^\g$),  
\[\int_V \sum_{\substack{\tau'\subset\underline{\tau} \\\tau'\in{\bf{S}}_\g(s_1)}}|\sum_{\substack{\tau',\tau''\in{\bf{S}}_\g(s_1)\\\tau'\tau''\subset\tau\\\tau'\sim\tau''}}f_{\tau'}\overline{f_{\tau''}} *\widecheck{\eta}_{s_1}|^{7/4} \lesssim \int \sum_{\substack{\tau'\subset\underline{\tau} \\\tau'\in{\bf{S}}_\g(s_1)}}|f_{\tau'}|^{7/2} W_V. \]
Suppose a certain $\sigma$ term dominates, so that \eqref{coneapp} is bounded by 
\begin{equation}\label{bddby} C(\log R) \sum_{\substack{\underline{\tau}\in{\bf{S}}_\g(\sigma^{-1}s_1)}} \sum_{V\|V_{\underline{\tau},s_1^{-3}}^\g}|V|^{-1}\Big(\int \sum_{\substack{\tau'\subset\underline{\tau} \\\tau'\in{\bf{S}}_\g(s_1)}}|f_{\tau'}|^{7/2} W_V \Big)^{2}. \end{equation}
If $\sigma<K^{-\e}$, then use a general Corollary \ref{Pcor}, Cauchy-Schwarz, and H\"{o}lder's inequality to bound the above expression by 
\begin{align*} 
C(\log R)C_\d R^{\d} \sum_{\substack{\underline{\tau}\in{\bf{S}}_\g(\sigma^{-1}s_1)}} &\sum_{V\|V_{\underline{\tau},s_1^{-3}}^\g}|V|^{-1}\Big(\int \big(\sum_{\substack{\tau''\subset\underline{\tau} \\\tau'\in{\bf{S}}_\g(K^{-\e/2}s_1)}}|f_{\tau''}|^2\big)^{7/4} W_V \Big)^{2} \\
&\lesssim C(\log R)C_\d R^{\d} \int \big(\sum_{\substack{\tau''\in{\bf{S}}_\g(K^{-\e/2}s_1)}}|f_{\tau''}|^2\big)^{7/2}  . \end{align*}
The other case is that $\sigma\ge K^{-\e}$. Then \eqref{bddby} is bounded by 
\[ CK^\e\sum_{\tau\in{\bf{S}}_\g(s_1)}\int|f_\tau|^7. \]
By Taylor approximation, the Fourier support of $f_{\tau}$ is contained in the $C(R/K)^{-4/3}$-neighborhood of 
\begin{align*}
    \{\g(c)+\g'(c)(t-c)+\g''(c)(t-c)^2+\g'''(c)(t-c)^3&+C_2[\g''(c)+\g'''(c)(t-x)]+C_3\g'''(c):\\
    &c\le t\le c+(R/K)^{-1/3},\quad|C_2|\le R^{-2/3},\quad|C_3|\le R^{-1}\}. 
\end{align*}
Assume that $(R/K)^{-4/3}\le R^{-1}$. Then after an affine transformation with determinant uniform over $\mc{C}$ (see \textsection\ref{genrescalesec}), this is a subset of $\mc{M}^3(R)$, the anisotropic neighborhood of the moment curve. Therefore, by Theorem \ref{main}, 
\[ \sum_{\tau\in{\bf{S}}_\g(s_1)}\int|f_\tau|^7\lesssim C_\d R^{\d}\sum_{\tau\in{\bf{S}}_\g(s_1)}\int|\sum_{\substack{\theta\subset\tau\\\theta\in{\bf{S}}_\g(R^{-1/3})}}|f_{\theta}|^2|^{7/2}. \]
Since $\|\cdot\|_{\ell^{7/2}}\le\|\cdot\|_{\ell^1}$, the right hand side is bounded by 
\[ C_\d R^\d\int|\sum_{\theta\in{\bf{S}}_\g(R^{-1/3})}|f_\theta|^2|^{7/2}.\]
The iteration terminates with the conclusion that $S(R)\le C_\d (\log R)R^{2\d}K^\e S(R/K)$ in this case.

Assume from now that $\sigma<K^{-\e}$ and the outcome of the first step was 
\[ C_\d (\log R) R^{2\d}\int\big(\sum_{\tau''\in{\bf{S}}_\g(K^{-\e/2}s_1)}|f_{\tau''}|^2\big)^{7/2} .\]
 We may now repeat the argument laid out so far but with $\sum_{\tau''\in{\bf{S}}_\g(K^{-\e/2}s_1)}|f_{\tau''}|^2$ in place of $\sum_{\tau\in{\bf{S}}_\g((R/K)^{-1/3})}|f_\tau|^2$. After $m$ iterations, conclude that 
 \[ (\text{L.H.S. of \eqref{genassS2}})\lesssim [C_\d R^{2\d}(\log R)]^m\Big[\int|\sum_{\tau\in{\bf{S}}_\g(s_m)}|f_\tau|^2|^{7/2}+K^\e\int|\sum_{\theta\in{\bf{S}}_\g(R^{-1/3})}|f_\theta|^2|^{7/2}\Big]\]
in which $s_m\le K^{-m\e/2}(R/K)^{-1/3}$. The algorithm terminates in at most $M$ steps where $K^{-M\e/2}(R/K)^{-1/3}\sim R^{-1/3}$, so $M\sim \e^{-1}$. The conclusion is then 
\[ S(R)\le [C_\d R^{2\d}(\log R)]^{\e^{-1}} K^\e S(R/K),\]
recalling the the condition from earlier that $K\le R^{1/4}$. If for $\eta>0$, $\sup_{R\ge 1}S(R)\lesssim_\eta R^\eta$, then we also have that for each $R\gtrsim_\e 1$, 
\[ S(R)\lesssim_{\e,\eta} [C_\d R^{2\d}(\log R)]^{\e^{-1}} K^\e (R/K)^\eta \]
where we are free to choose $\d>0$, $\e>0$, and $K\le R^{1/4}$. Letting $\d=\e^3$, $\e=\eta/2$, and $K=R^{\e}$, we see that the above inequality implies 
\[ S(R)\lesssim_{\e,\eta} R^{\eta-\eta^2/10} . \]
This means that the infimum of $\eta>0$ such that $\sup_{R\ge 1}S(R)\lesssim_\eta R^\eta$ is zero, as desired. 

\end{proof}

\subsection{Affine rescaling for $\g\in\mc{C}$ \label{genrescalesec}}

Recall the definition of $\mc{C}$. Let $\frac{1}{2}\ge a>0$ and $\frac{a}{4}\ge \nu>0$. The class ${\mc{C}}$ is the collection of $C^4$ curves $\g:[0,1]\to\R^3$ satisfying 
\begin{align}
\label{cond1}    \g(t)&=(t,\g_2(t),\g_3(t)),\\
\label{cond2}    \g_2''(t)&\ge a\quad\text{and}\quad \g_3'''(t)\ge a\qquad\text{for all}\quad t\in[0,1],\\
\label{cond3}    \|\g_2\|_{C^4}&\le 1\quad\text{and}\quad \|\g_3\|_{C^4}\le 1\\
\label{cond4}    \det\big[\g'(t)\,\g''(t)\, \g'''(t)\big]&\ge \nu\qquad\text{for all}\quad t\in[0,1]. 
\end{align}

Let $C_0>0$ to be a constant that is permitted to depend on $a$ and $\nu$. Consider $C_0\le S<R$ and $\tau\in{\bf{S}}_\g(S^{-1/3})$. We will show that there exists $\tilde{\g}\in\mc{C}$ and an affine transformation $A_\tau:\R^3\to\R^3$ mapping $\tau$ to $[0,1]^3$ and mapping each $\theta\in{\bf{S}}_\g(R^{-1/3})$ with $\theta\subset\tau$ to $A(\theta)\in{\bf{S}}_{\tilde{\g}}((R/S)^{-1/3})$. 
Suppose that $\tau$ is the $l$th piece in ${\bf{S}}(S^{-1/3})$ so that
\[ \tau=\{\g(t)+C_2\g''(t)+C_3\g'''(t):lS^{-1/3}\le t<(l+1)S^{-1/3},\quad|C_2|\le S^{-2/3},\quad |C_3|\le S^{-1} \}. \]
Let $t_0=lS^{-1/3}$ and define $A_\tau$ to be the map 
\begin{equation}\label{genAtau}\begin{cases}
\xi_1\qquad\mapsto\qquad S^{1/3}(\xi_1-t_0)\\
\xi_2\qquad\mapsto\qquad S^{2/3}[\xi_2-\g_2(t_0)-\g_2'(t_0)(\xi_1-t_0)]\\
\xi_3\qquad\mapsto\qquad c_\tau S\big[\xi_3-\g_3(t_0)-\g_3'(t_0)(\xi_1-t_0)-\frac{\g_3''(t_0)}{\g_2''(t_0)}[\xi_2-\g_2(t_0)-\g_2'(t_0)(\xi_1-t_0)] \big]  \end{cases}  \end{equation}
in which $c_\tau=\frac{{\g}''(t_0)}{2[\g_2''(t_0)\g_3'''(t_0)-\g_3''(t_0)\g_2'''(t_0)]}$. Define $s=S^{1/3}(t-t_0)$. For $0\le s\le 1$, let $\tilde{\g}(s)= A_\tau\g(t)$ so that we may write  $\tilde{\g}(s)=(s,\tilde{\g}_2(s),\tilde{\g}_3(s))$, where
\begin{align*} 
\tilde{\g}_2(s)&=S^{2/3}[\g_2(t)-\g_2(t_0)-S^{-1/3}\g_2'(t_0)s]\\
\quad\text{and}\quad \tilde{\g}_3(s)&=S\big[\g_3(t)-\g_3(t_0)-S^{-1/3}\g_3'(t_0)s-S^{-2/3}\frac{\g_3''(t_0)}{\g_2''(t_0)}\tilde{\g}_2(s)\big]. \end{align*} 
We now verify that $\tilde{\g}\in\mc{C}$. Clearly, $\tilde{\g}$ is $C^4$ and has the form \eqref{cond1}. Since $\|\g_2\|_{C^4}\le 1$, it is straightforward to verify that $\|\tilde{\g}_2\|_{C^4}\le 1$. Since $\|\g_3\|_{C^4}\le 1$, $\|\tilde{\g}_2\|_{C^4}\le 1$, and $S$ is sufficiently large depending on $a$ and $\nu$, it is then easy to check that $\|\tilde{\g}_3\|_{C^4}\le 1$, which verifies \eqref{cond3}.  Note the form of the following derivatives 
\[ \tilde{\g}_2''(s)= \g_2''(t)\quad\text{and}\quad\tilde{\g}_3'''(s)= c_\tau\big[\g_3'''(t)-\frac{\g_3''(t_0)}{\g_2''(t_0)}\g_2'''(t)\big] . \]
Clearly $\g_2''(t)\ge a$ implies that $\tilde{\g}_2''(s)\ge a$. By Taylor approximation, the above expression shows that $\tilde{\g}_3'''(s)=1+O(S^{-1/3})$. Since $a\le 1/2$ and $S$ is at least a certain size depending on $a$ and $\nu$, we have that $\tilde{\g}_3'''(s)\ge a$, which verifies \eqref{cond2}. Using similar reasoning, it is straightforward to check \eqref{cond3}. Finally, we have 
\begin{align*}    \det[\tilde{\g}'(s)\,\tilde{\g}''(s)\,\tilde{\g}'''(s)]&=c_\tau\g_2''(t)\big[\g_3'''(t)-\frac{\g_3''(t_0)}{\g_2''(t_0)}\g_2'''(t)\big]-c_\tau \g_2'''(t)\big[\g_3''(t)-\frac{\g_3''(t_0)}{\tilde{\g}_2''(t_0)}\g_2''(t)\big]\\
    &= c_\tau \det[\g'(t)\,\g''(t)\,\g'''(t)]\ge \frac{a}{2}+O(S^{-1/3}).   
\end{align*} 
where we used Taylor approximation in the final inequality. Since we assumed that $\nu<\frac{a}{4}$ and $S$ is large depending on $a$ and $\nu$, property \eqref{cond4} holds. This concludes the verification that $\tilde{\g}\in\mc{C}$. We record this in the following lemma. 

\begin{lemma}[General affine rescaling]\label{genrescale} Let $C_0$ be a sufficiently large constant depending on $\nu$ and $a$. 
Let $C_0\le S\le R$ and $\tau\in{\bf{S}}_\g(S^{-1/3})$. There exists $\tilde{\g}\in\mc{C}$ such that for each $\theta\in{\bf{S}}_\g(R^{-1/3})$, $A_\tau\theta\in{\bf{S}}_{\tilde{\g}}((R/S)^{-1/3})$. 
\end{lemma}

\begin{proof} From the discussion preceding the lemma, it remains to check how $A_\tau$ transforms the blocks $\theta\in{\bf{S}}_\g(R^{-1/3})$. For $m\in\{0,\ldots,R^{1/3}-1\}$ with $lS^{-1/3}\le mR^{-1/3}<(l+1)S^{-1/3}$, consider 
\[ \theta=\{\g(t)+C_2\g''(t)+C_3\g'''(t):mR^{-1/3}\le t<(m+1)R^{-1/3},\quad|C_2|\le R^{-2/3},\quad |C_3|\le R^{-1} \}. \]
Applying $A_\tau$ yields 
\begin{align*} 
A_\tau\theta&=\{A_\tau\g(t)+C_2A_\tau\g''(t)+C_3A_\tau\g'''(t):mR^{-1/3}\le t<(m+1)R^{-1/3},\quad|C_2|\le R^{-2/3},\quad |C_3|\le R^{-1} \}   \\
&= \{\tilde{\g}(s)+C_2\tilde{\g}''(s)+C_3\tilde{\g}'''(s):m(R/S)^{-1/3}-l\le s< (m+1)(R/S)^{-1/3}-l,\\
&\qquad\qquad\qquad\qquad\qquad\qquad\qquad\qquad\qquad\qquad\qquad\qquad \quad|C_2|\le (R/S)^{-2/3},\quad |C_3|\le (R/S)^{-1} \} ,
\end{align*} 
which is clearly an element of ${\bf{S}}_{\tilde{\g}}((R/S)^{-1/3})$.

\end{proof}

\subsection{The general versions of Theorem \ref{cylPS} and Corollary \ref{Pcor} \label{Pgensec}}

We claim that for each $\g\in\mc{C}$, Theorem \ref{cylPS} and Corollary \ref{Pcor} hold with $(t,\g_2(t))$ in place of the parabola $(t,t^2)$. We following the argument from \textsection\ref{Papp}. First define $\text{PS}_\g(r)$ to be the analogue of $\text{PS}(r)$ for each $\g\in\mc{C}$. Then let $\text{PS}_{\mc{C}}(r)=\sup_{\g\in\mc{C}}\text{PS}_\g(r)$. The quantity $\text{PS}_{\mc{C}}(r)$ satisfies an analogous multiscale inequality as $\text{PS}(r)$ does in Proposition \ref{multiscale}. In the \emph{narrow} case of the proof of Proposition \ref{multiscale}, we have
\[ \int_{\R^2}|h|^{7/2}\lesssim \int_{\R^2} \sum_{\ell(\tau)=4L^{-1}}|h_\tau|^{7/2} \]
in which $h:\R^2\to\C $ is a Schwartz function with Fourier support in $\mc{N}_{r^{-1}}(\{(t,\g_2(t)):0\le t\le 1\})$ and $\tau$ are approximate rectangles of the form 
\[ \{(\xi_1,\xi_2)\in\mc{N}_{16/L^2}(\{(t,\g_2(t)):0\le t\le 1\}): l4L^{-1}\le \xi_1\le (l+1)4L^{-1}\} . \]
Instead of invoking parabolic rescaling, we rescale the Fourier side using the affine map $B_\tau:\R^2\to\R^2$ defined by 
\begin{align*}
\begin{cases}    \xi_1\qquad\mapsto\qquad 4^{-1}L(\xi_1-l4L^{-1})\\
\xi_2\qquad\mapsto\qquad 16^{-1}L^2[\xi_2-\g_2(l4L^{-1})-\g_2'(l4L^{-1})(\xi_1-l4L^{-1})]\end{cases}. 
\end{align*}
The map $B_\tau$ takes each $\theta$ block of $(t,\g_2(t))$ of dimensions $\sim r^{-1/2}\times r^{-1}$ to a block of $(t,\tilde{\g}_2(t))$ with dimensions $\sim (4r/L)^{-1/2}\times (4r/L)$, where $\tilde{\g}$ is the same curve defined in \textsection\ref{genrescalesec}. This allows us to conclude that 
\[ \int_{\R^2}|h_{\tau}|^{7/2}\le \text{PS}_{\mc{C}}(16r/L^2)^{7/2}\int_{\R^2}(\sum_{\theta\subset\tau}|h_\theta|^2){7/4}, \] 
as we did in the parabola case. The final adaptation in the proof of Proposition \ref{multiscale} is to note that the bilinear restriction estimate that is referenced also holds uniformly for $\g\in\mc{C}$, which is clear from the proof of Theorem 15 in \cite{lvlsets}. 

Concluding that $\text{PS}_{\mc{C}}(r)\lesssim_\e r^\e$ after proving the multiscale inequality follows from the same argument as for the parabola. With the boundedness of $\text{PS}_{\mc{C}}(r)$, the proofs of Theorem \ref{cylPS} and Corollary \ref{Pcor} for $(t,\g_2(t))$ in place of $(t,t^2)$ are unchanged.

\subsection{The general version of Proposition \ref{cone3.5} \label{conegensec}}

In this section, we describe how to obtain a version of Proposition \ref{cone3.5} that holds uniformly for the general cones that arise from analyzing $\g\in\mc{C}$. The argument proving Proposition \ref{cone3.5} adapts immediately to those cones, except that we must invoke a general version of the $L^4$ wave envelope estimate Theorem 1.3 from \cite{locsmooth} (which was proven for the light cone). It therefore suffices to prove a certain generalization of Theorem 1.3 from \cite{locsmooth}. 

Before stating the theorem we need to prove, we present an abbreviated version of \textsection\ref{geo} which describes the geometric relationship between moment curve blocks and cone planks. When we describe sets as comparable or use $O(\cdot)$ notation, we shall always mean up to constants which are permitted to depend on $\nu$ and $a$ from the definition of $\mc{C}$. Let $\g\in\mc{C}$ and consider $\theta\in{\bf{S}}_\g(R^{-1/3})$ given by 
\[ \theta=\{\g(t)+C_2\g''(t)+C_3\g'''(t):lR^{-1/3}\le t<(l+1)R^{-1/3},\quad|C_2|\le R^{-2/3},\quad|C_3|\le R^{-1}\} \]
where $l$ is some integer in the range $0\le l\le R^{1/3}-1$. Up to a $O(R^{-1})$ error, the above set is approximately
\[ \theta\approx \{\g(lR^{-1/3})+C_1\g'(lR^{-1/3})+C_2\g''(lR^{-1/3})+C_3\g'''(lR^{-1/3}):|C_i|\le R^{-i/3}\quad\text{for}\quad i=1,2,3\}. \]
The Fourier support of $\sum_{\theta\in{\bf{S}}_\g(R^{-1/3})}|f_\theta|^2$ is contained in the union $\cup_{\theta\in{\bf{S}}_\g(R^{-1/3})}(\theta-\theta)$. The sets $\theta-\theta$ are 
\[ \theta-\theta\approx \{C_1\g'(lR^{-1/3})+C_2\g''(lR^{-1/3})+C_3\g'''(lR^{-1/3}):|C_i|\le R^{-i/3}\quad\text{for}\quad i=1,2,3\}. \]
Furthermore, on the annulus $|\xi|\sim R^{-1/3}$, $\theta-\theta$ is comparable to 
\[ (\theta-\theta)\cap\{|\xi|\sim R^{-1/3}\}\approx R^{-1/3}\mc{N}_{R^{-2/3}}(\{\lambda\g'(t):lR^{-1/3}\le t<(l+1)R^{-1/3}),\quad \frac{1}{2}\le |\lambda|\le 2\}).  \]
The right hand side is an $R^{-1/3}$-dilation of the $R^{-2/3}$-neighborhood of a plank from the generalized cone 
\[ \Gamma_\g=\{\lambda\g'(t):0\le t\le 1,\quad\frac{1}{2}\le \lambda\le 2\}.  \]

For $S\ge 1$, let ${\bf{T}}_\g(S^{-1/2})$ denote a collection of finitely overlapping planks $\z$ with dimensions $\sim 1\times S^{-1/2}\times S^{-1}$ whose union is approximately the $S^{-1}$ neighborhood of $\Gamma_\g$. Each $\z\in{\bf{T}}_\g(S^{-1/2})$ is comparable to the $S^{-1}$ neighborhood of a sector
\[ \{\lambda\g'(t):lS^{-1/2}\le t<(l+1)S^{-1/2},\quad\frac{1}{2}\le \lambda\le 2\}.  \]
For each $\z\in{\bf{S}}_\g(S^{-1/2})$, let $\z^*$ be a parallelogram with right angles that is dual to $\z$, centered at the origin, and has dimensions $\sim 1\times S^{1/2}\times S$. For each dyadic $\sigma\in[S^{-1/2},1]$ and each $\tau\in{\bf{T}}_\g(\sigma^{-1}S^{-1/2})$, let $U_{\tau,S}$ denote an anisotropically dilated version of $\tau^*$ which is comparable to 
\[ \text{Convex Hull }\big( \bigcup_{\substack{\z\in{\bf{S}}_\g(S^{-1/2})\\ \z\subset {\tau}}}\z^*\big).\] The wave envelope $U_{\tau,S}$ has dimensions $\sim \sigma^{-2}\times \sigma^{-1}S^{1/2}\times S$. Write $U\|U_{\tau,S}$ for a tiling of $\R^3$ by translates of $U_{\tau,S}$. Now we may state the theorems we must prove for our general cones. A general $L^4$ wave envelope estimate for $\Gamma_\g$ (the analogue of Theorem 1.3 from \cite{locsmooth}) follows from the following two theorems. 

\begin{theorem} \label{gensqfn}For each $\e>0$, there exists $C_\e<\infty$ such that the following holds. For any $S\in 4^\N$ and $\g\in\mc{C}$, if $f:\R^3\to\C$ is a Schwartz function with $\supp\widehat{f}\subset\mc{N}_{S^{-1}}(\Gamma_\g)$, then 
\[ \int_{\R^3}|\sum_{\z\in{\bf{T}}_\g(S^{-1/2})}f_{\z}|^4\le C_\e S^\e \int_{\R^3}|\sum_{\z\in{\bf{T}}_\g(S^{-1/2})}|f_{\z}|^2|^2 . \]
\end{theorem}

\begin{theorem}\label{genwee}For each $\e>0$, there exists $C_\e<\infty$ such that the following holds. For any $S\in 4^\N$ and $\g\in\mc{C}$, if $f:\R^3\to\C$ is a Schwartz function with $\supp\widehat{f}\subset\mc{N}_{S^{-1}}(\Gamma_\g)$, then 
\begin{equation}\label{genweeineq} \int_{\R^3}|\sum_{\z\in{\bf{T}}_\g(S^{-1/2})}|f_{\z}|^2|^2\le C_\e S^\e \sum_{S^{-1/2}\le \sigma\le 1}\sum_{\tau\in{\bf{S}}_\g(\sigma^{-1}S^{-1/2})}\sum_{U\|U_{\tau,S}}|U|\Big(\fint_U\sum_{\z\subset\tau}|f_\z|^2\Big)^2. \end{equation}

\end{theorem}

First we prove Theorem \ref{gensqfn} assuming Theorem \ref{genwee}. 
\begin{proof}[Proof of Theorem \ref{gensqfn}]
Let $T(S)$ be the infimum of constants $A>0$ such that 
\[ \int_{\R^3}|\sum_{\z\in{\bf{T}}_\g(S^{-1/2})}f_{\z}|^4\le A \int_{\R^3}|\sum_{\z\in{\bf{T}}_\g(S^{-1/2})}|f_{\z}|^2|^2 \]
for any Schwartz function $f$ with Fourier support in $\mc{N}_{S^{-1}}(\Gamma_\g)$, for any $\g\in\mc{C}$. Let $K<S$ be a large parameter we will specify later. Fix a Schwartz $f:\R^3\to\C$ with Fourier transform supported in $\mc{N}_{S^{-1}}(\Gamma_\g)$, for some $\g\in\mc{C}$. Using the definition of $T(\cdot)$, we have
\[   \int_{\R^3}|f|^4\le T(S/K) \int_{\R^3}|\sum_{\z_0\in{\bf{T}}_\g((S/K)^{-1/2})}|f_{\z_0}|^2|^2 \]
in which the $\z_0$ vary over the set ${\bf{T}}_\g((S/K)^{-1/2})$. Apply Theorem \ref{genwee} to obtain 
\[  \int_{\R^3}|f|^4\le T(S/K)C_\d S^\d  \sum_{(S/K)^{-1/2}\le \sigma\le 1}\sum_{\tau\in{\bf{S}}_\g(\sigma^{-1}(S/K)^{-1/2})}\sum_{U\|U_{\tau,S/K}}|U|\Big(\fint_U\sum_{\z_0\subset\tau}|f_{\z_0}|^2\Big)^2.\]

If $\sigma^{-2}>S^{1/2}$, then using local $L^2$ orthogonality, we have
\[ \sum_{\tau\in{\bf{S}}_\g(\sigma^{-1}(S/K)^{-1/2})}\sum_{U\|U_{\tau,S/K}}|U|\Big(\fint_U\sum_{\z_0\subset\tau}|f_{\z_0}|^2\Big)^2\lesssim  \sum_{\tau\in{\bf{S}}_\g(\sigma^{-1}(S/K)^{-1/2})}\sum_{U\|U_{\tau,S/K}}|U|\Big(|U|^{-1}\int\sum_{\z\subset\tau}|f_{\z}|^2W_U\Big)^2\]
where the $\z$ are in ${\bf{T}}_\g(S^{-1/2})$. Then by Cauchy-Schwarz and $\|\cdot\|_{\ell^2}\le\|\cdot\|_{\ell^1}$, the right hand side above is bounded by a constant times $ \int|\sum_\z|f_\z|^2|^2$. 

Now suppose that $\sigma^{-2}\le S^{1/2}$. Consider $\tau\in{\bf{T}}_\g(\sigma^{-1}S^{-1/2})$ which is a neighborhood of the sector 
\[ X_\tau= \{\lambda\dot{\g}(t):l\sigma^{-1}S^{-1/2}\le t\le (l+1)l\sigma^{-1}S^{-1/2},\quad\frac{1}{2}\le \lambda\le 2\}. \]
Write $r=\sigma S^{1/2}$. The affine transformation $B_\tau:\R^3\to\R^3$ defined by 
\[ \begin{cases}
\xi_1\mapsto \xi_1\\
\xi_2\mapsto r[\xi_2-\dot{\g}_2(lr^{-1})]\\
\xi_3\mapsto r^2[\xi_3-\dot{\g}_3(lr^{-1})-\frac{\ddot{\g}_3(lr^{-1})}{\ddot{\g}_2(lr^{-1})}(\xi_2-\dot{\g}_2(lr^{-1}))]
\end{cases} \]
maps $X_\tau$ to $\Gamma_{\tilde{\g}}$, where $\tilde{\g}\in\mc{C}$ is the same curve that $\g$ maps to after using the affine transformation from \textsection\ref{genrescalesec}. Note that 
\[ \sum_{U\|U_{\tau,S/K}}|U|\Big(\fint_U\sum_{\z_0\subset\tau}|f_{\z_0}|^2\Big)^2\lesssim \int_{\R^3}|\sum_{\z_0\subset\tau}|f_{\z_0}|^2|^2. \]
Rescaling the Fourier side using $B_\tau$ and using the definition of $T(\cdot)$ again, we may use Khintchin's inequality to select $c_{\z_0}\in\{\pm1\}$ satisfying  
\[ \int|\sum_{\z_0\subset\tau}|f_{\z_0}|^2|^2\lesssim \int|\sum_{\z_0\subset\tau} c_{\z_0}f_{\z_0}|^4\lesssim  T(\sigma^{-1}K^{1/2})\int|\sum_{\z}|f_\z|^2|^2.  \]

Conclude from this argument that 
\[ T(S)\lesssim_\d S^\d T(S/K)\max_{S^{-1/4}<\sigma\le 1}T(\sigma^{-1}K^{1/2}). \]
Choose $K=S^{3/4}$ and assume without loss of generality that $S(\cdot)$ is a nondecreasing function so that the above inequality implies that $T(S)\lesssim_\d S^\d T(S^{1/4})T(S^{5/8})$. Since $\d>0$ is arbitrarily small, conclude that the infimum of $\eta>0$ such that $\sup_{S\ge 1}T(S)\lesssim_\eta S^\eta$ must be zero. 

\end{proof}

Now we prove Theorem \ref{genwee}, which is a succinct version of the proof of Lemma 1.4 in \cite{locsmooth}. 
\begin{proof}[Proof of Theorem \ref{genwee}] Fix $\g\in\mc{C}$ and a Schwartz $f:\R^3\to\C$ with $\supp\widehat{f}\subset\mc{N}_{S^{-1}}(\Gamma_\g)$. By Plancherel's theorem, we have
\[ \int|\sum_{\z\in{\bf{T}}_\g(S^{-1/2})}|f_\z|^2|^2=\int|\sum_{\z\in{\bf{T}}_\g(S^{-1/2})}\widehat{|f_\z|^2}|^2 . \]
The Fourier transforms $\widehat{|f_\z|^2}$ are supported in $\z-\z$, which is essentially $\z$ translated to the origin. For each dyadic $\sigma$, $S^{-1/2}< \sigma\lesssim 1$ and each $\z\in{\bf{T}}_\g(S^{-1/2})$, define 
\[ \z_\sigma=\{A\dot{\g}(lS^{-1/2})+B\ddot\g(lS^{-1/2})+C\dot\g(lS^{-1/2})\times\ddot\g(lS^{-1/2}): |A|\lesssim \sigma^2,\quad|B|\sim \sigma S^{-1/2},\quad |C|\lesssim S^{-1}\}. \]
For $\sigma=S^{-1/2}$, define $\z_{S^{-1/2}}$ to be  
\[ \sigma_{S^{-1/2}}=\{A\dot{\g}(lS^{-1/2})+B\ddot\g(lS^{-1/2})+C(\dot\g(lS^{-1/2})\times\ddot\g(lS^{-1/2})): |A|\lesssim S^{-1},\quad|B|\lesssim S^{-1},\quad |C|\lesssim S^{-1}\}.\] 
Note that $\z-\z$ is contained in $\cup_\sigma \z_\sigma$. Write $\Omega_\sigma=\cup_{\z}\z_\sigma$. Suppose that 
\[ \int|\sum_{\z\in{\bf{T}}_\g(S^{-1/2})}\widehat{|f_\z|^2}|^2 \lesssim (\log S)\int_{\Omega_\sigma}|\sum_{\z\in{\bf{T}}_\g(S^{-1/2})}\widehat{|f_\z|^2}|^2 \]
First we will show that  
\begin{equation}\label{show} \int_{\Omega_\sigma}|\sum_{\z\in{\bf{T}}_\g(S^{-1/2})}\widehat{|f_\z|^2}|^2 \lesssim_\e S^\e \sum_{\tau\in{\bf{T}}_\g(\sigma^{-1}S^{-1/2})}\int |\sum_{\z\subset\tau}\widehat{|f_\z|^2}|^2. \end{equation}
The above inequality is permitted to have implicit constants depending on the parameters $a,\nu$ from the definition of $\mc{C}$. Therefore, it suffices to show that for $\d=\d(a,\nu)>0$ that we will choose later, for each $\tau'\in{\bf{S}}_\g(\d)$, we have
\[ \int_{\Omega_\sigma}|\sum_{\substack{\z\subset\tau'}}\widehat{|f_\z|^2}|^2\lesssim_\e S^\e \sum_{\substack{\tau\subset\tau'\\ \tau\in{\bf{S}}_\g(\sigma^{-1}S^{-1/2})}}\int_{\Omega_\sigma}|\sum_{\substack{\z\subset\tau}}\widehat{|f_\z|^2}|^2.\]
Consider the intersection $\Omega_\sigma\cap\{\xi\in\R^3:\xi_1=h\}$ where $h\ge 0$ (the $h<0$ case is analogous). Suppose first that $h\gg \sigma S^{-1/2}$. If $\{\xi:\xi_1=h\}\cap\z_\sigma\cap\z_\sigma'$ is nonempty, then since $\ddot{\g}$ has $0$ in the first component, there is some choice of parameters $B,B'$ with $|B|,|B'|\lesssim\sigma S^{-1/2}$ such that
\begin{align}\label{Taylor} |h\dot{\g}(lS^{-1/2})+B\ddot{\g}(lS^{-1/2})-h\dot{\g}(l'S^{-1/2})-B'\ddot{\g}(l'S^{-1/2})|\lesssim S^{-1}, \end{align}
where $\z$ is the $l$th plank and $\z'$ is the $(l')$th plank. It follows from Taylor's theorem and the definition of $\mc{C}$ that $|lS^{-1/2}-l'S^{-1/2}|\lesssim \sigma^{-1}S^{-1/2}$, so $\z_\sigma,\z\sigma'$ are contained in the same $\tau\in{\bf{S}}_\g(\sigma^{-1}S^{-1/2})$. 

Similarly, if $h\ll \sigma S^{-1/2}$, then for some parameters $B,B'$ with $|B|\sim|B'|\sim \sigma S^{-1/2}$, \eqref{Taylor} holds. In this case, the left hand side is dominated by the difference in the $\ddot{\g}$ terms, which is bounded below by $\sim |B||lS^{-1/2}-l'S^{-1/2}|$. Conclude again that $\z_\sigma$ and $\z_{\sigma'}$ are contained in the same $\tau\in{\bf{S}}_\g(\sigma^{-1}S^{-1/2})$. 

It remains to analyze the case in which $h\sim \sigma S^{-1/2}\sim |B|$. Fix a $\tau'\in{\bf{S}}_\g(\d)$. The set $(\cup_{\z\subset\tau'}\z_\sigma)\cap\{\xi:\xi_1=h\}$ is contained in an $\sim S^{-1}$-neighborhood of 
\[ \{h\dot{\g}(t)+B\ddot{\g}(t): t\in I,\quad|B|\sim\sigma S^{-1/2}\} \]
where $I$ is the $\d$ interval corresponding to the sector defining $\tau'$. Fix $B$ satisfying $|B|\sim\sigma S^{-1/2}$ and $t\in I$. Consider the property that $h\dot{\g}(t)+B\ddot{\g}(t)\in\mc{N}_{S^{-1}}(\z_\sigma)$, where $\z_\sigma$ has corresponding parameter $lS^{-1/2}\in I$. 
Letting $t_k=k\sigma^{-1}S^{-1/2}$, $t_k\in I$, we note that if $|lS^{-1/2}-t_k|\le \sigma^{-1}S^{-1/2}$, then 
\begin{equation}\label{close} |h\dot{\g}(lS^{-1/2})+B'\ddot{\g}(lS^{-1/2})-h\dot{\g}(t_k)-B'\ddot{\g}(t_k)|\lesssim S^{-1}\end{equation}
for any $B'$ with $|B'|\sim\sigma S^{-1/2}$. Therefore, to bound the number of $\z_\sigma$ within $\sim S^{-1}$ of $h\dot{\g}(t)+B\ddot{\g}(t)$, it suffices to bound the number of $t_k$ for which there exists $B_k$ with $|B_k|\sim \sigma S^{-1/2}$ satisfying 
\[ |h\dot{\g}(t)+B\ddot{\g}(t)-h\dot{\g}(t_k)-B_k\ddot{\g}(t_k)|\lesssim S^{-1}. \]
This is because for each $t_k$ satisfying the above inequality, there are $\sim \sigma^{-1}$ many $lS^{-1/2}$ within $\sigma ^{-1}S^{-1/2}$ of $t_k$ that satisfy a similar inequality. Dividing the previous displayed inequality through by $h$, we obtain 
\[ |\dot{\g}(t)+\frac{B}{h}\ddot{\g}(t)-\dot{\g}(t_k)-\frac{B_k}{h}\ddot{\g}(t_k)|\lesssim \sigma ^{-1}S^{-1/2}. \]
This implies that 
\begin{equation}\label{alm} |[\dot{\g}(t)+\frac{B}{h}\ddot{\g}(t)]\cdot (\dot{\g}(t_k)\times\ddot{\g}(t_k))|\lesssim \sigma^{-1}S^{-1/2}. \end{equation}
Let $F_t:I\to \R$ be defined by 
\[ F_t(s)=[\dot{\g}(t)+\frac{B}{h}\ddot{\g}(t)]\cdot (\dot{\g}(s)\times\ddot{\g}(s)). \]
Note that $F_t(t)=0$, $F_t'(s)=[\dot{\g}(t)+\frac{B}{h}\ddot{\g}(t)]\cdot (\dot{\g}(s)\times\dddot{\g}(s))$, $|F_t'(t)|\gtrsim \nu>0$, $F_t''(s)=[\dot{\g}(t)+\frac{B}{h}\ddot{\g}(t)]\cdot (\dot{\g}(s)\times{\g}^{(4)}(s)+\ddot{\g}(s)\times\dddot{\g}(s))$, and $|F_t''(s)|\lesssim 1$. It follows that for $\d$ (the length of the domain interval $I$) sufficiently small, depending on $\nu$, $|F_t(s)|\sim |t-s|$. Conclude that \eqref{alm} implies that $|t-t_k|\lesssim \sigma^{-1}S^{-1/2}$. Since the $t_k$ are $\sigma^{-1}S^{-1/2}$-separated, there are $\lesssim 1$ many $t_k$ which satisfy \eqref{alm}. This finishes the justification of \eqref{show}.

It remains to bound the integral 
\[ \int_{\Omega_\sigma}|\sum_{\z\subset\tau}\widehat{|f_\z|^2}|^2\]
for each $\tau\in{\bf{S}}_\g(\sigma^{-1}S^{-1/2})$. The inequality \eqref{close} shows that if $\z\subset\tau$ ($\tau$ with corresponding parameter $t_k$), then $\z_\sigma\subset C(\z_k)_\sigma$, where $\z_k\in{\bf{S}}_\g(S^{-1/2})$ has parameter $lS^{-1/2}=t_k$. Let $\eta_k$ be a bump function equal to $1$ on $C(\z_k)_\sigma$ and with Fourier transform decaying rapidly away from $U_{\tau,S}^*$, the wave envelope centered at the origin. Then by Cauchy-Schwarz and the decay of $|\widecheck{\eta_k}|$, we have
\begin{align*} 
\int_{\Omega_\sigma}|\sum_{\z\subset\tau}\widehat{|f_\z|^2}|^2&=\int_{\Omega_\sigma}|\sum_{\z\subset\tau}\widehat{|f_\z|^2}\eta_k|^2\\
&\le \int_{\R^2}|\sum_{\z\subset\tau}\widehat{|f_\z|^2}|^2 = \int_{\R^2}|\sum_{\z\subset\tau}|f_\z|^2*\widecheck{\eta_k}|^2\\
&\le  \sum_{U\|U_{\tau,S}}\int_U|\sum_{U'\|U_{\tau,S}}\int_{U'}\sum_{\z\subset\tau}|f_\z|^2(y)\|\widecheck{\eta_k}\|_{L^\infty(x-U')}dy|^2dx\\
&\lesssim \sum_{U\|U_{\tau,S}}\int_U\sum_{U'\|U_{\tau,S}}\|\widecheck{\eta}_k\|_{L^\infty(x-U')}\Big(\int_{U'}\sum_{\z\subset\tau}|f_\z|^2\Big)^2|U|^{-1}dx\\
&\lesssim \sum_{U\|U_{\tau,S}}\sum_{U'\|U_{\tau,S}}\|\widecheck{\eta}_k\|_{L^\infty(U-U')}\Big(\int_{U'}\sum_{\z\subset\tau}|f_\z|^2\Big)^2\sim \sum_{U'\|U_{\tau,S}}|U'|\Big(\fint_{U'}\sum_{\z\subset\tau}|f_\z|^2\Big)^2,
\end{align*}
which is the desired upper bound.

\end{proof}

\bibliographystyle{alpha}
\bibliography{AnalysisBibliography}

\end{document}